\documentclass[11pt,letterpaper]{article}
\usepackage[margin=1.0in]{geometry}
\usepackage[utf8]{inputenc} 
\usepackage[T1]{fontenc}    
\usepackage{dsfont}
\usepackage{multirow}
\usepackage{mathrsfs}
\usepackage{wrapfig}
\usepackage[T1]{fontenc}
\usepackage{float}
\usepackage{bm}

\usepackage{booktabs}       
\usepackage{amsfonts}       
\usepackage{nicefrac}       
\usepackage{microtype}      
\usepackage{enumerate}
\usepackage{lipsum}
\usepackage{pifont}
\usepackage{tcolorbox}
\usepackage{mathtools}
\usepackage{cuted}
\usepackage{float}
\usepackage{mdframed}
\usepackage[dvipsnames,table,xcdraw]{xcolor}
\usepackage{bbm}
\usepackage{varioref}
\usepackage{hyperref}
                                
\usepackage{amssymb} 
\usepackage{rotating}
\usepackage{graphicx}
\usepackage{subfigure}
\usepackage{subcaption}
\usepackage{caption}
\usepackage{xcolor}
\usepackage{pifont}
\usepackage[normalem]{ulem}

\usepackage{amsthm}
\usepackage{algorithm}
\usepackage{algpseudocode}
\usepackage{amsmath}
\usepackage{tikz}
\usetikzlibrary{arrows.meta, positioning}
\usetikzlibrary{decorations.pathreplacing, calc}
\usetikzlibrary{shapes, shapes.geometric}

\newtheorem{theorem}{Theorem}[section]
\newtheorem{lemma}{Lemma}[section]
\newtheorem{remark}{Remark}[section]
\newtheorem{proposition}{Proposition}[section]
\newtheorem{definition}{Definition}[section]

\theoremstyle{plain}

\newcommand{\bxi}{\boldsymbol{\xi}}

\def\argmin{\mathop{\rm argmin}}

\title{\Large
Equilibrium Invariance, Proximality, and Surrogation: Moreau-Smoothed Best-Response Pathways in Stochastic Nonsmooth Games
}

\author{Zhuoyu Xiao\footnote{Department of Industrial and Operations Engineering, University of Michigan, Ann Arbor, MI, 48109. Email: zyxiao@umich.edu.} \and Uday V. Shanbhag\footnote{Corresponding author. Department of Industrial and Operations Engineering, University of Michigan, Ann Arbor, MI, 48109. Email: udaybag@umich.edu.}}

\newmdenv[
  backgroundcolor=gray!20,
  linecolor=white,
  linewidth=0pt,
  innertopmargin=10pt,
  innerbottommargin=10pt,
  innerleftmargin=10pt,
  innerrightmargin=10pt,
]{custombox}

\begin{document}

\maketitle
\thispagestyle{empty}

\begin{abstract}
Best-response (BR) schemes represent an important avenue for learning equilibria in noncooperative games. However, extant rate guarantees for BR schemes generally necessitate stringent smoothness requirements on player objectives and the availability of suitably defined eigenvalue bounds, significantly limiting the reach of such schemes, and few schemes if any exist for the efficient resolution of a broad class of nonsmooth and nonconvex games with expectation-valued objectives. This motivates our study of Moreau-smoothed BR schemes that allow for nonsmooth objectives. First, we consider a class of nonsmooth and strongly convex games (but potentially non-monotone) under uncertainty. By presenting an equilibrium invariance claim, we present synchronous and asynchronous schemes, equipped with linear and sublinear rate guarantees and associated complexity statements. Second, faced by weakly convex player objectives, we incorporate surrogation into the Moreau-smoothed best-response and show that the resulting smoothed quasi-Nash equilibrium (QNE) constitutes an $\mathcal{O}(\eta)$-QNE of the original weakly convex game, where $\eta > 0$ denotes the Moreau smoothing parameter. In this setting, we again present synchronous and asynchronous BR schemes, equipped with linear and sublinear rates and analogous complexity bounds. Preliminary numerics on a range of such games appear promising.
\end{abstract}

\textbf{Key words.} nonsmooth games, Nash equilibria, quasi-Nash equilibria, weakly convex, best-response schemes, surrogation

\section{Introduction}\label{Sec-1}

The Nash equilibrium (NE) \cite{nash-1951} represents a fundamental
solution concept in $N$-person noncooperative games, assuming a natural
relevance in imperfectly competitive engineered and economic
systems \cite{fudenberg-levine-1998}, and more recently, in machine learning \cite{lin-jin-jordon-2020}. In this paper, we examine a Moreau-smoothed best-response (BR) framework with a view towards contending with the challenges posed by uncertainty and nonsmoothness. Our focus lies on the $N$-player noncooperative game ${\cal G}(\mathbf{f},X, \bxi)$, where ${\bf f}$ denotes the collection of player-specific objectives, i.e. ${\bf f} \triangleq \{f_i\}_{i=1}^N$, $X$ denotes the Cartesian product of player-specific strategy sets, i.e. $X \triangleq \prod_{i=1}^N X_i$, and the randomness is captured by the random variable $\bxi: \Omega \to \mathbb{R}^m$ defined on the probability space $(\Omega, \mathcal{F}, \mathbb{P})$. In this game, for any $i \in [N]$, the $i$th player solves
\begin{equation} 
    \min_{x_i \in X_i} f_i(x_i, x_{-i}) \triangleq \mathbb{E} \left[ \tilde{f}_i(x_i, x_{-i}, \bxi) \right], 
\end{equation}
where $X_{i} \subseteq \mathbb{R}^{n_{i}}$ is convex and closed, $X_{-i} \triangleq {\prod_{j \ne i}}\ X_j$, and $x_{-i} = (x_j)_{j \ne i} \in X_{-i}$. In addition,  $\Xi \triangleq \left\{ \bxi(\omega) \mid \omega \in \Omega \right\}$ and for any $i \in [ N ]$, $\Tilde{f}_{i}(\bullet, x_{-i}, \xi)$ is a possibly nonsmooth real-valued function for given $x_{-i}$ and $\xi$.

\subsection{Related work}\label{Sec-1.1}
In continuous-strategy games, existence of NE typically relies on the convexity of players' objectives while existence may fail in nonconvex settings \cite{pang-scutari-2011, pang-scutari-2013}. By leveraging Bouligand stationarity, a weaker solution concept of quasi-Nash equilibrium (QNE) was recently introduced \cite{pang-scutari-2011}. Two main algorithmic paradigms have been developed for computing NE: (I) \emph{Gradient-response schemes.} Here,  players update strategies via gradient steps. While early efforts leveraged techniques from monotone variational inequality (VI) theory~\cite{facchinei-pang-2009, koshal-nedic-shanbhag-2013, lei-shanbhag-2022, yousefian-nedic-shanbhag-2016}, subsequent weakenings considered non-monotone VIs under Minty-type conditions \cite{arefizadeh-nedic-2024, huang-zhang-2024, vankov-nedic-sankar-2023} or pseudomonotonicity and its variants \cite{dang-lan-2015, iusem-jofre-oliveira-thompson-2019, iusem-jofre-oliveira-thompson-2017, kannan-shanbhag-2019, kotsalis-lan-li-2022, yousefian-nedic-shanbhag-2017}. Our recent work \cite{xiao-shanbhag-2025-GR} weakened convexity assumptions and considered a range of non-monotone settings, establishing last-iterate convergence guarantees for QNEs in stochastic nonconvex smooth games. (II) \emph{Best-response schemes.} Such procedures~\cite{facchinei-pang-2009, lei-shanbhag-2020, lei-shanbhag-2022, lei-shanbhag-pang-sen-2020} considered synchronous or asynchronous versions but required convexity and smoothness, while weakenings to the QNE regime were presented in \cite{cui-pang-2021, pang-razaviyayn-2016, razaviyayn-2014} by employing surrogation. We summarize much of the prior work in Table 1, qualifying where the novelty arises in proposed schemes. However, these weakenings require a directional derivative consistency condition (see \cite{cui-pang-2021}) at the limit point and are often difficult to verify. Therefore, significant gaps still exist.

\begin{tcolorbox} \textbf{Gap:} No known BR schemes with rate and complexity guarantees exist for computing exact/approximate equilibria in stochastic (non)convex nonsmooth and potentially non-monotone settings under easily verifiable assumptions.
\end{tcolorbox}

\begin{table}
\centering
\scriptsize
{\renewcommand{\arraystretch}{1.4}
\begin{tabular}{@{\hspace{5pt}}c@{\hspace{10pt}}c@{\hspace{5pt}}c@{\hspace{10pt}}c@{\hspace{5pt}}c@{\hspace{4pt}}c@{\hspace{6pt}}c@{\hspace{5pt}}}
\toprule
\toprule
\multicolumn{7}{c}{\small Synchronous BR under contractivity of BR map} \\
\midrule
\midrule
\small{Literature} & \small{Applicability} & \small{Stoch.} & \small{Surrogation} & \small{a.s.} & \small{Rate} & \small{Complex.}\\
\small{\cite{facchinei-pang-2009}} & \small{cvx + C$^{2}$} & \small{\ding{55}} & \small{\ding{55}} & \small{-} & \small{linear}  & \small{\ding{55}}\\
\small{\cite{lei-shanbhag-2022, lei-shanbhag-pang-sen-2020}} & \small{cvx + C$^{2}$}  & \small{\ding{51}} & \small{\ding{55}} & \small{\ding{51}} & \small{linear} & \small{\ding{51}}\\
\small{\cite{cui-pang-2021, pang-razaviyayn-2016, razaviyayn-2014}} & \small{ncvx + d.d.c.} & \small{\ding{55}} & \small{\ding{51}} & \small{-} & \small{linear} & \small{-}\\
\rowcolor{gray!30} \small{\textbf{MS-SBR}} & \small{cvx + C$^{0}$} & \small{\ding{51}} & \small{\ding{55}} & \small{\ding{51}} & \small{linear}  & \small{\ding{51}}\\
\rowcolor{gray!30} \small{\textbf{MS-SSBR}} & \small{weakly cvx + C$^{0}$} & \small{\ding{51}} & \small{\ding{51}} & \small{\ding{51}} & \small{linear} & \small{\ding{51}}\\
\toprule
\toprule
\multicolumn{6}{c}{\small Potentiality-based asynchronous BR} \\
\midrule
\midrule
\small{Literature} & \small{Applicability} & \small{Stoch.} & \small{Surrogation} & \small{a.s.} & \small{Rate} & \small{Complex.}\\
\small{\cite{lei-shanbhag-2020}} & \small{cvx + C$^{1}$} & \small{\ding{51}} & \small{\ding{55}} & \small{\ding{51}} & \small{\ding{55}}& \small{\ding{55}}   \\
\rowcolor{gray!30}
\small{\textbf{MS-ABR}} & \small{cvx + C$^{0}$} & \small{\ding{51}} & \small{\ding{55}} & \small{\ding{51}} & \small{sublinear} & \small{\ding{51}}\\
\rowcolor{gray!30} \small{\textbf{MS-SABR}}
& \small{weakly cvx + C$^{0}$} & \small{\ding{51}} & \small{\ding{51}} & \small{\ding{51}} & \small{sublinear} & \small{\ding{51}}\\
\bottomrule
\bottomrule
\end{tabular}
}
\caption{Comparison with existing  BR schemes (d.d.c.: direct. deriv. consistency).}
\label{MS-SBR-table}
\end{table}

\subsection{Outline and contributions}\label{Sec-1.2}

After providing some preliminaries in Section \ref{Sec-2}, our main algorithms are captured in Sections \ref{Sec-3} and \ref{Sec-4}. Finally, we provide numerics and concluding remarks in Sections \ref{Sec-5} and \ref{Sec-6}, respectively. Our contributions lie in developing inexact sampling-enabled BR schemes for two classes of games.

\vspace{5pt}

(I) \textbf{Nonsmooth non-monotone strongly convex games.} Prior BR schemes for strongly convex games are reliant on C$^2$ requirements for player objectives and access to suitable eigenvalues of the associated second derivative matrices~\cite{facchinei-pang-2009, lei-shanbhag-pang-sen-2020}. In contrast, we present a framework, allowing for nonsmooth player-specific functions and potentially non-monotone concatenated subgradient maps. By observing a surprising equilibrium invariance property by which the Nash equilibria of the original nonsmooth game are equivalent to those of a Moreau-smoothed game for any positive smoothing parameter $\eta$, our work paves the way for presenting two Moreau-smoothed BR schemes. (I.A) \emph{Linearly convergent synchronous BR scheme.} Our proposed Moreau-smoothed synchronous BR ({\bf MS-SBR}) scheme converges linearly with an approximate sample complexity of $\mathcal{O}(N\epsilon^{-2})$, under a guarantee that relies on a suitable contractivity requirement that leverages Lipschitzian properties of the smoothed problem. Notably, our guarantee relies on a more refined analysis rather than bounds on eigenvalues. (I.B) \emph{Sublinearly convergent asynchronous BR scheme.} Prior results \cite{cui-shanbhag-2023, lei-shanbhag-2020, pang-razaviyayn-2016, razaviyayn-2014}, under a potentiality requirement, established asymptotic convergence of the iterates generated by asynchronous BR schemes but provided no rate or complexity guarantees. Instead, we present a Moreau-smoothed asynchronous BR scheme ({\bf MS-ABR}), where we refine the analysis by incorporating a residual function which forms the basis for providing a sublinear rate and an approximate complexity guarantee of $\mathcal{O}(N^{2}\eta^{-3}\epsilon^{-6})$, where $\eta$ need not be close to zero.

\vspace{5pt}

(II) \textbf{Nonsmooth weakly convex games.} One avenue for contending with nonsmoothness and nonconvexity lies in surrogation-based methods, which often employ a consistency condition at the limit point $x^{\infty}$, e.g., \cite{cui-pang-2021, liu-cui-pang-2022, pang-razaviyayn-2016}. However, such conditions are often harder to verify a priori. Instead, we consider a framework reliant on employing surrogations on the Moreau-smoothed counterpart of this game, presenting a class of Moreau-smoothed surrogated BR schemes for computing a quasi-Nash equilibrium. Notably, we show that a QNE of the $\eta$-smoothed game is an $\mathcal{O}(\eta)$-QNE of the original weakly convex game.  (II.A) {\em Linearly convergent synchronous BR scheme.} Akin to (I.A), we present a Moreau-smoothed synchronous surrogated BR ({\bf MS-SSBR}) scheme, providing a linear rate with an approximate sample complexity of $\mathcal{O}(N\eta^{-2}\epsilon^{-2})$. (II.B) {\em Sublinearly convergent asynchronous BR scheme.} An asynchronous counterpart ({\bf MS-SABR}) is developed and is characterized by almost sure convergence, a sublinear rate guarantee, and an approximate overall complexity of $\mathcal{O}(N^2\eta^{-5}\epsilon^{-6})$. The convergence and complexity results are summarized in Table \ref{TAB-contribution}.

\begin{table}[!htb]
  \centering
  \newcommand{\tabincell}[2]{%
    \begin{tabular}{@{}#1@{}}#2\end{tabular}%
  }
  \footnotesize

  {\renewcommand{\arraystretch}{1.4}
  \begin{tabular}{|c|c|c|c|}
    \hline
    Scheme & Convergence & Rate & Sample complexity \\
    \hline

    \tabincell{c}{\textbf{MS-SBR}}
    &
    \tabincell{c}{%
      a.s. cvgn. (Thm. \ref{as-convergence-MS-SBR})\\
      cvgn. in mean (Thm. \ref{linear-rate-MS-SBR})%
    }
    &
    \tabincell{c}{%
    linear \\
    (Thm. \ref{linear-rate-MS-SBR})%
    }
    &
    \tabincell{c}{%
    $\approx \mathcal{O}(N\epsilon^{-2})$\\
    (Thm. \ref{complexity-MS-SBR})%
    }
    \\
    \hline

    \tabincell{c}{\textbf{MS-ABR}}
    &
    \tabincell{c}{%
      a.s. cvgn. (Thm. \ref{MS-ABR-key-recursion})\\
      cvgn. in mean (Thm. \ref{MS-ABR-sublinear-rate})%
    }
    &
    \tabincell{c}{%
    sublinear \\
    (Thm. \ref{MS-ABR-sublinear-rate})%
    }
    &
    \tabincell{c}{%
    $\approx \mathcal{O}(N^2\eta^{-3}\epsilon^{-6})$\\
    (Thm. \ref{MS-ABR-complexity})%
    }
    \\
    \hline

    \tabincell{c}{\textbf{MS-SSBR}}
    &
    \tabincell{c}{%
      a.s. cvgn. (Thm. \ref{MS-SSBR-thm})\\
      cvgn. in mean (Thm. \ref{linear-rate-MS-SSBR})%
    }
    &
    \tabincell{c}{%
    linear \\
    (Thm. \ref{linear-rate-MS-SSBR})%
    }
    &
    \tabincell{c}{%
    $\approx \mathcal{O}(N\eta^{-2}\epsilon^{-2})$\\
    (Thm. \ref{complexity-MS-SSBR})%
    }
    \\
    \hline

    \tabincell{c}{\textbf{MS-SABR}}
    &
    \tabincell{c}{%
      a.s. cvgn. (Thm. \ref{MS-SABR-main-proposition})\\
      cvgn. in mean (Thm. \ref{sublinear-rate-MS-SABR})%
    }
    &
    \tabincell{c}{%
    sublinear \\
    (Thm. \ref{sublinear-rate-MS-SABR})%
    }
    &
    \tabincell{c}{%
     $\approx \mathcal{O}(N^2\eta^{-5}\epsilon^{-6})$\\
    (Thm. \ref{complexity-MS-SABR})%
    }
    \\
    \hline
  \end{tabular}}

  \caption{A summary of contributions.}
  \label{TAB-contribution}
\end{table}

\vspace{5pt}

\noindent \textbf{Notation.} The expectation and a realization of a random variable $\bxi$ are denoted by $\mathbb{E}[\bxi]$ and $\xi$, respectively. The notations $\mathrm{int}(X)$ and ${\bf 1}_X(x)$ denote the interior of $X$ and the indicator function of $X$, respectively. Given a closed convex set $X \subseteq \mathrm{dom} (f) \subseteq \mathbb{R}^{n}$, $x\in X$, and a direction $d\in \mathbb{R}^{n}$, $f^{\prime}(x; d)$, $\partial f(x)$, and $\mathcal{N}_{X}(x)$ denote the directional derivative of $f$ at $x$ along $d$, the subdifferential of $f$ (when convex) at $x$, and the normal cone to $X$ at $x$, respectively.

\section{Preliminaries}\label{Sec-2}

In this section, we introduce some necessary background.

\subsection{Surrogation method}\label{Sec-2.1}
The surrogation method, or successive convex approximation \cite{razaviyayn-2014}, or majorization–minimization \cite{marks-wright-1978, sun-babu-palomar-2016} has been used to solve nonconvex optimization \cite{hong-wang-razaviyayn-luo-2017, mairal-2015, pang-razaviyayn-2016, razaviyayn-2014, razaviyayn-hong-luo-2013, sun-babu-palomar-2016}. We consider the first-order surrogate functions, requiring neither $L$-smoothness nor the majorization property.

\begin{definition}\label{first-order-surrogate}\em
Given a smooth but possibly nonconvex $f: \mathbb{R}^n \to \mathbb{R}$ and $x\in \mathrm{dom}(f)$, a convex function $\widehat{f}(\bullet; x): \mathbb{R}^n \to \mathbb{R}$ is a first-order surrogate of $f(\bullet)$ near $x$ when the two local properties are satisfied: (i) The function values match, i.e.  $\widehat{f}(x; x) = f(x)$; and (ii) The gradients match, i.e.  $\nabla \widehat{f}(x; x) = \nabla f(x)$.$\hfill \Box$
\end{definition}

The simplest surrogate of $f$ near $x$ satisfying Definition \ref{first-order-surrogate} is the linear surrogate defined by $\widehat{f}(\bullet; x) \triangleq f(x)+\nabla f(x)^{\top}(\bullet-x)$, thereby for given given $x_{-i}$, the linear and quadratic surrogate of player-specific function $f_{i}(\bullet, x_{-i})$ for player $i$ near $x_{i}$ are given as follows for $M > 0$:
\begin{align*} 
	(i) \;\; \widehat{f}_i(\bullet, x_{-i}; x_{i}) \, &\triangleq \, f_i(x) + \nabla_{x_i} f_i(x)^\top (\bullet - x_{i}); \mbox{ and }\\
    (ii) \;\; \widehat{f}_i(\bullet,x_{-i}; x_i) &\triangleq f_i(x) + \nabla_{x_i} f_i(x)^{\top}(\bullet - x_i) + \tfrac{M}{2}\| \bullet - x_i \|^{2}.
\end{align*}

\subsection{Moreau envelopes in convex and weakly convex regimes}
A natural way to address nonsmoothness lies in {\em smoothing} the nonsmooth function by a smooth approximation, allowing for leveraging techniques from smooth optimization. Such techniques include Nesterov's smoothing~\cite{nesterov-2005}, Moreau smoothing~\cite{moreau-1965}, Ben-Tal-Teboulle smoothing~\cite{ben-tal-teboulle-2006}, and randomized smoothing~\cite{steklov-1907}. The Moreau envelope $f^{\eta}$ of $f$ is defined as
\begin{equation*}
    f^{\eta}(x) \, \triangleq \, \min_{y} \left\{ \, f(y) + \tfrac{1}{2\eta} \left\|\, y-x \,\right\|^{2} \, \right\}.
\end{equation*}
The proximal point $\widehat{x} \triangleq \mathrm{prox}_{\eta f}(x)$ satisfies (see \cite{beck-2017, davis-drusvyatskiy-2019})
\begin{equation}\label{Moreau-property}
    \begin{aligned}
        \nabla f^{\eta}(x) \in \partial f(\widehat{x}), \;\; 
        \|\widehat{x} - x\| = \eta \|\nabla f^{\eta}(x)\|, \;\; f(\widehat{x}) \leq f(x).
    \end{aligned}
\end{equation}
Note that the above properties hold for both convex and weakly convex functions. Next we recall some properties of the Moreau envelope under convexity (see \cite{beck-2017, jalilzadeh-shanbhag-blanchet-glynn-2022}).
\begin{proposition}\label{cvx-ME-property} \em
    Consider a closed convex proper function $f: \mathbb{R}^{n}\to \bar{\mathbb{R}}$ and its Moreau envelope $f^{\eta}(x)$ with $\eta > 0$. Then the following hold:\\
    (i) $x^{*}$ minimizes $f$ over $\mathbb{R}^{n}$ if and only if $x^{*}$ minimizes $f^{\eta}$ over $\mathbb{R}^{n}$; (ii) $f^{\eta}$ is $\tfrac{1}{\eta}$-smooth;\\
    (iii) if $f$ is $\mu$-strongly convex on $\mathbb{R}^{n}$, $f^{\eta}$ is $\tfrac{\mu}{\eta\mu + 1}$-strongly convex on $\mathbb{R}^{n}$.$\hfill \Box$
\end{proposition}


First introduced by Nurminskii in~\cite{nurminskii-1973}, a function $f: \mathbb{R}^{n}\to \mathbb{R} \cup \{\infty\}$ is said to be $\rho$-weakly convex if the function $f(\bullet) + \tfrac{\rho}{2} \| \bullet \|^{2}$ is convex. Weakly convex functions have broad applications in optimization and learning \cite{davis-drusvyatskiy-2019}. The subdifferential of a $\rho$-weakly convex function $f$ can be defined naturally as 
    $\partial f(x) \, \triangleq \, \partial h(x) - \rho x \, = \, \partial \left[\, f(x) + \tfrac{\rho}{2}\|x\|^{2}\, \right] - \rho x,$
where $\partial h(x)$ is the subdifferential of $h$ in the sense of convex analysis. The following proposition relates directional derivatives and subdifferentials for weakly convex functions, recovering the classical result from convex analysis.

\begin{proposition}\label{ws-dd-subg}
    \em Given a proper, closed, and $\rho$-weakly convex function $f: \mathbb{R}^{n}\to \mathbb{R} \cup \{\infty\}$ and $x\in \mathrm{int}\,(\mathrm{dom}(f))$, $f^{\prime}(x; d) = \max\limits_{u\in \partial f(x)} \{ \langle u, d \rangle \}$ for every $d\in \mathbb{R}^{n}$.
\end{proposition}
\begin{proof}
    We define the convex function $h$ as $h(x) \triangleq f(x) + \tfrac{\rho}{2}\|x\|^{2}$. Hence we have $f(x) = h(x) - \tfrac{\rho}{2}\|x\|^{2}$. By Proposition 4.24 and Theorem 4.30 in \cite{mordukhovich-nam-2023}, we know that the convex function $h$ is directionally differentiable and we have
    \begin{align*}
        f^{\prime}(x; d) = h^{\prime}(x; d) - \langle \rho x, d \rangle = \max_{v\in \partial h(x)} \{ \langle v, d \rangle \} - \langle \rho x, d \rangle = \max_{v\in \partial h(x)} \{ \langle v - \rho x, d \rangle \}.
    \end{align*}
    Since the subdifferential of $f$ is defined as $\partial f(x) = \partial h(x) - \rho x$, it follows that 
    \begin{equation*}
        f^{\prime}(x; d) = \max_{v- \rho x \in \partial h(x) - \rho x} \{ \langle v - \rho x, d \rangle \} = \max_{u\in \partial f(x)} \{ \langle u, d \rangle \},
    \end{equation*}
    which completes the proof.
\end{proof}

If $f$ is $\rho$-weakly convex, the Moreau envelope $f^{\eta}$ is C$^{1}$ if $\eta < \rho^{-1}$ and $\nabla f^{\eta}(x) = \eta^{-1} ( x - \mathrm{prox}_{\eta f}(x) )$. Next, we formalize the $L$-smoothness of Moreau envelopes of weakly convex functions.

\begin{proposition}[\mbox{\cite[Lemma 12]{renaud-leclaire-papadakis-2025}}]\label{Lipschitz-wcvx}\em
    Given a proper closed and $\rho$-weakly convex function $f: \mathbb{R}^{n}\to \bar{\mathbb{R}}$. For any $\eta < \rho^{-1}$, $\nabla f^{\eta}$ is $L$-Lipschitz with
    \begin{equation*}
        L = \begin{cases}
            \tfrac{1}{\eta}, & \text{if }\; \eta\rho \leq \tfrac{1}{2}, \\
            \tfrac{\rho}{1-\eta\rho}, & \text{if }\; \eta\rho \geq \tfrac{1}{2}.
        \end{cases}
    \end{equation*}
    That is, we have that $\| \nabla f^{\eta}(x) - \nabla f^{\eta}(y) \| \leq L\| x - y \|$. $\hfill \Box$
\end{proposition}

\subsection{Quasi-Nash equilibrium under weak convexity}\label{Sec-2.3}

Via Bouligand (or B-) stationarity~\cite{cui-pang-2021}, the quasi-Nash equilibrium (QNE) was introduced in~\cite{pang-scutari-2011}.

\begin{definition}[QNE under B-stationarity]\label{QNE-def}\em
   Consider an $N$-player game where for any $i \in [N]$, $f_{i}(\bullet, x_{-i})$ is directionally differentiable for given $x_{-i}$. Suppose that $X_{i}$ is convex for any $i\in [N]$. Then the tuple $x^{\ast}$ is said to be a quasi-Nash equilibrium (QNE) if $x^{\ast}_{i}$ is a B-stationary point of $f_{i}(\bullet, x^{\ast}_{-i})$, i.e.,
   \begin{equation}\label{QNE-0}
       f_{i}^{\prime}(x^{\ast}_{i}, x^{\ast}_{-i}; x_{i}-x^{\ast}_{i}) \geq 0,
       \ \forall x_{i}\in X_{i}
   \end{equation}
   holds for any $i\in [N]$ and any given $x^{\ast}_{-i}$. $\hfill \Box$
\end{definition}

Now we give an alternative definition of QNE in the context of weakly
convex games. Its proof follows from Proposition \ref{ws-dd-subg}.

\begin{proposition}[QNE under weak convexity]\label{QNE-def-alternative} \em
Consider an $N$-player game where for any $i \in [N]$, $f_{i}(\bullet, x_{-i})$ is weakly convex for given $x_{-i}$. Suppose $X_{i}$ is convex for any $i\in [N]$. Then the tuple $x^{*}$ is a QNE if and only if for any $i\in [N]$, 
    \begin{equation}
        0\in \partial_{x_{i}} f_{i}(x^{*}_{i}, x^{*}_{-i}) + \mathcal{N}_{X_{i}}(x^{*}_{i}).
    \end{equation}
\end{proposition}
\begin{proof}
    Suppose that $x^{*} = (x^{*}_{i})_{i=1}^{N}$ is a QNE. By Proposition \ref{ws-dd-subg}, for any $i\in [N]$, we have
        $\max_{u_{i}\in \partial_{x_{i}} f(x^{*}_{i}, x^{*}_{-i})} \{ \langle u_{i}, x_{i}-x^{\ast}_{i} \rangle \} \geq 0,~ \forall x_{i}\in X_{i}.$
    Therefore, there exists $u^{*}_{i}\in \partial_{x_{i}} f(x^{*}_{i}, x^{*}_{-i})$ such that $\langle u^{\ast}_{i}, x_{i}-x^{\ast}_{i} \rangle \geq 0$ hence $\langle -u^{\ast}_{i}, x_{i}-x^{\ast}_{i} \rangle \leq 0$ for any $x_{i}\in X_{i}$. Since $X_{i}$ is convex, it follows from the definition of the normal cone that
    \begin{equation*}
        -u^{*}_{i} \in \mathcal{N}_{X_{i}}(x^{*}_{i}) \implies 0\in u^{*}_{i} + \mathcal{N}_{X_{i}}(x^{*}_{i}) \implies 0\in \partial_{x_{i}} f_{i}(x^{*}_{i}, x^{*}_{-i}) + \mathcal{N}_{X_{i}}(x^{*}_{i}).
    \end{equation*}
    The result follows by proving the converse (follows by reversing the argument).
\end{proof}

Next, we prove the existence of a QNE under weak convexity and compactness.

\begin{theorem}[Existence of QNE] \em
    Consider an $N$-player noncooperative game where for any $i \in [ N ]$, $f_{i}(\bullet, x_{-i})$ is $\rho_{i}$-weakly convex for any $x_{-i}$ and $X_{i}$ is nonempty, compact, and convex. Then such a weakly convex game has a QNE.
\end{theorem}
\begin{proof}
    For each $\rho_{i}$-weakly convex function $f_{i}(\bullet, x_{-i})$, we can write it as a dc-type (where dc stands for ``difference of convex'') decomposition as 
        $f_{i}(x_{i}, x_{-i}) = \left[ f_{i}(x_{i}, x_{-i}) + \tfrac{\rho_{i}}{2}\|x_{i}\|^{2} \right] - \left[ \tfrac{\rho_{i}}{2}\|x_{i}\|^{2} \right].$
    Then the existence of QNE follows from \cite[Proposition 4.2]{pang-razaviyayn-2016} directly.
\end{proof}

We know that under the smoothness, QNE reduces to
\begin{equation*}
    \nabla_{x_{i}} f_{i}(x^{*}_{i}, x^{*}_{-i})^{\top}(x_{i}-x^{\ast}_{i})\geq 0,~\forall x_{i}\in X_{i},
\end{equation*}
for any $i\in [N]$. Therefore, the solution set of variational inequality $\mathrm{VI}\:(X, F)$ where $X \triangleq \prod_{i=1}^{N}X_{i}$ and $F(x) \triangleq (\nabla_{x_{i}}f_{i}(x_{i},x_{-i}))_{i=1}^{N}$ captures QNE. The existence result under smoothness can be found in \cite{xiao-shanbhag-2025-GR}.

\section{Inexact BR schemes for stochastic strongly convex non-monotone games}\label{Sec-3}

In this section, we consider computing an NE for stochastic strongly convex but nonsmooth and potentially non-monotone games. We first show the invariance of an NE under the Moreau smoothing for any $\eta > 0$, as captured by the following proposition, representing a novel game-theoretic counterpart of Proposition~\ref{cvx-ME-property}-(i). Then we present a synchronous BR scheme (\textbf{MS-SBR}) and an asynchronous variant (\textbf{MS-ABR}) in subsections \ref{Sec-3.1} and \ref{Sec-3.2}, respectively.

\begin{proposition}[Invariance of NE under Moreau smoothing]\label{NE-equals-NE} \em
    Consider the $N$-player game $\mathcal{G}({\bf f}, X, \bxi)$, where for any $i \in [ N ]$, the player-specific function $f_{i}(\bullet, x_{-i})$ is $\sigma_{i}$-strongly convex but nonsmooth for given $x_{-i}$. Then for any $\eta > 0$, $x^{\ast}$ is an NE of $\mathcal{G}({\bf f}, X, \bxi)$ if and only if $x^{\ast}$ is an NE of the Moreau-smoothed game $\mathcal{G}({\bf \bar{f}}^{\eta}, \mathbb{R}^{n}, \bxi)$, where  $\bar{f}^{\eta}_{i}(\bullet, x_{-i})$ is the Moreau envelope of $\bar{f}_{i}(\bullet, x_{-i}) \triangleq f_{i}(\bullet, x_{-i}) + \mathbf{1}_{X_{i}}(\bullet)$ and the $i$th player solves
    \begin{equation*}
        \min_{x_{i}\in \mathbb{R}^{n_{i}}} \bar{f}^{\eta}_{i}(x_{i}, x_{-i}), ~ \forall i\in [N].
    \end{equation*}
\end{proposition}
\begin{proof}
    The invariance of NE follows from \cite[Lemma 2]{jalilzadeh-shanbhag-blanchet-glynn-2022} and the first-order optimality condition of strongly convex nonsmooth optimization \cite[Theorem 7.15]{mordukhovich-nam-2023} immediately. We omit the proof here.
\end{proof}

This equilibrium invariance property holds for any positive smoothing parameter $\eta$ and is crucial in developing efficient algorithms as in the next two sections.

\subsection{MS-SBR under strong convexity}\label{Sec-3.1} 

Based on Proposition \ref{NE-equals-NE}, we propose a Moreau-smoothed synchronous BR (\textbf{MS-SBR}) scheme, where $\varepsilon_{i}^{k} > 0$ is the inexactness employed by player $i$ at iteration $k$.

\begin{algorithm}[htb]\caption{Moreau-smoothed Synchronous BR (\textbf{MS-SBR})}\label{MS-SBR}
{\it Initialize:} Initialize $k=0$, given $\eta, K, \{\varepsilon_i^k\}_{i,k}$, and $x^{0} = (x_{i}^{0})_{i=1}^{N}\in X$.

{\it Iterate until $k\ge K$:} Compute $x^{k+1} = (x_{i}^{k+1})_{i=1}^{N}\in X$ by solving the subproblems
\begin{equation}\label{MS-SBR-eqn1}
    x_{i}^{k+1}\in \{ z\in X_{i}: \mathbb{E}[\| z - \widehat{x}^{\eta}_{i}(x^{k}) \|^2 \,\vert\, x^{k}] \leq (\varepsilon_{i}^{k})^2 ~ \textrm{a.s.} \}, ~ \forall i\in [N],
\end{equation}
where the BR solution $\widehat{x}^{\eta}_{i}(x^{k})$ is defined as
\begin{equation}\label{MS-SBR-eqn2}
    \widehat{x}^{\eta}_{i}(x^{k}) \triangleq \argmin\limits_{z_{i}\in \mathbb{R}^{n_i}} \left\{ \bar{f}^{\eta}_{i}(z_{i}, x^{k}_{-i}) + \tfrac{\mu}{2}\| z_{i} - x^{k}_{i} \|^{2} \right\}, ~ \forall i\, \in \, [ N ].
\end{equation}

{\it Return:} Return $x^{K}$ as the final estimate.
\end{algorithm}

\subsubsection{Convergence analysis} We impose Assumption $\mathrm{A}$ to establish convergence guarantees of \textbf{MS-SBR}.

\vspace{5pt}

\emph{Assumption $\mathrm{A}$.} (A1) For any $i\in [N]$, each $f_{i}(\bullet, x_{-i})$ is $\sigma_{i}$-strongly convex and C$^{0}$ for any $x_{-i}$.  (A2) For any $i\in [N]$, each $\nabla_{x_{i}}\bar{f}^{\eta}_{i}(x_{i}, \bullet)$ is $\bar{L}_{-i}$-Lipschitz for any $x_{i}$. (A3) (Contractive property) We define the matrix
\begin{equation}\label{contraction-matrix_1}
    \Gamma_{1} \triangleq \begin{bmatrix}
        \tfrac{\mu}{\sigma_{1}/(\eta \sigma_{1}+1)+\mu} & \tfrac{\bar{L}_{-1}}{\sigma_{1}/(\eta \sigma_{1}+1)+\mu} & \dots & \tfrac{\bar{L}_{-1}}{\sigma_{1}/(\eta \sigma_{1}+1)+\mu} \\
        \tfrac{\bar{L}_{-2}}{\sigma_{2}/(\eta\sigma_{2}+1)+\mu} & \tfrac{\mu}{\sigma_{2}/(\eta\sigma_{2}+1)+\mu} & \dots & \tfrac{\bar{L}_{-2}}{\sigma_{2}/(\eta\sigma_{2}+1)+\mu} \\
        \vdots & \vdots & \ddots & \vdots \\
        \tfrac{\bar{L}_{-N}}{\sigma_{N}/(\eta\sigma_{N}+1)+\mu} & \tfrac{\bar{L}_{-N}}{\sigma_{N}/(\eta\sigma_{N}+1)+\mu} & \dots & \tfrac{\mu}{\sigma_{N}/(\eta\sigma_{N}+1)+\mu}
    \end{bmatrix}.
\end{equation}
The spectral norm of matrix $\Gamma_{1}$ is strictly less than $1$, i.e., $\| \Gamma_{1} \|<1$.

\begin{remark}\em
    For ease of exposition, we suppress the dependence of $\bar{L}_{-i}$ on $\eta$. $\hfill \Box$
\end{remark}

\begin{theorem}[Almost sure convergence of \textbf{MS-SBR}]\label{as-convergence-MS-SBR} \em
    Consider the $N$-player strongly convex game $\mathcal{G}({\bf f}, X, \bxi)$ and its Moreau-smoothed game $\mathcal{G}({\bf \bar{f}}^{\eta}, \mathbb{R}^{n}, \bxi)$. Let $\{ x^{k} \}^{\infty}_{k = 0}$ be generated by \textbf{MS-SBR}. Suppose Assumption $\mathrm{A}$ holds and $\varepsilon^{k}_{i} \geq 0$ with $\sum_{k=0}^{\infty}\varepsilon^{k}_{i} < \infty$ for any $i\in [N]$. Then the following hold. \\
    (i) The mapping $\widehat{x}^{\eta}(\bullet) = (\widehat{x}^{\eta}_{i}(\bullet))_{i=1}^{N}$ is contractive hence the fixed point is unique. \\
    (ii) The tuple $x^{\ast}$ is the unique fixed point of $\widehat{x}^{\eta}(\bullet)$ if and only if $x^{\ast}$ is the unique NE. \\
    (iii) The sequence $\{x^{k}\}_{k=0}^{\infty}$ converges to $x^{\ast}$ a.s.
\end{theorem}
\begin{proof}
    (i) Since $f_{i}(\bullet, x_{-i})$ is $\sigma_{i}$-strongly convex, we know that $\bar{f}_{i}(\bullet, x_{-i})$ is also $\sigma_{i}$-strongly convex. By Proposition \ref{cvx-ME-property}, we know that $\bar{f}^{\eta}_{i}(\bullet, x_{-i})$ is $\tfrac{\sigma_{i}}{\eta\sigma_{i}+1}$-strongly convex. Similar as the proof of \cite[pp. 479]{facchinei-pang-2009}, for any $i\in [N]$ and any $w, y \in X$, 
    \begin{align*}
            0 &\leq (\widehat{x}^{\eta}_{i}(y)-\widehat{x}^{\eta}_{i}(w))^{\top} \big[ (\nabla_{x_{i}} \bar{f}^{\eta}_{i}(\widehat{x}^{\eta}_{i}(w), w_{-i}) + \mu(\widehat{x}^{\eta}_{i}(w)-w_{i})) \\
            &\quad - (\nabla_{x_{i}} \bar{f}^{\eta}_{i}(\widehat{x}^{\eta}_{i}(y), y_{-i}) + \mu(\widehat{x}^{\eta}_{i}(y)-y_{i})) \big] \\
            &= (\widehat{x}^{\eta}_{i}(y)-\widehat{x}^{\eta}_{i}(w))^{\top}(\nabla_{x_{i}} \bar{f}^{\eta}_{i}(\widehat{x}^{\eta}_{i}(w), w_{-i})-\nabla_{x_{i}} \bar{f}^{\eta}_{i}(\widehat{x}^{\eta}_{i}(y), y_{-i})) \\
            &\quad + (\widehat{x}^{\eta}_{i}(y)-\widehat{x}^{\eta}_{i}(w))^{\top}\mu (\widehat{x}^{\eta}_{i}(w)-\widehat{x}^{\eta}_{i}(y)) + (\widehat{x}^{\eta}_{i}(y)-\widehat{x}^{\eta}_{i}(w))^{\top}\mu(y_{i}-w_{i}) \\
            &\leq (\widehat{x}^{\eta}_{i}(y)-\widehat{x}^{\eta}_{i}(w))^{\top}(\nabla_{x_{i}} \bar{f}^{\eta}_{i}(\widehat{x}^{\eta}_{i}(w), w_{-i})-\nabla_{x_{i}} \bar{f}^{\eta}_{i}(\widehat{x}^{\eta}_{i}(y), w_{-i})) \\
            &\quad + (\widehat{x}^{\eta}_{i}(y)-\widehat{x}^{\eta}_{i}(w))^{\top}(\nabla_{x_{i}} \bar{f}^{\eta}_{i}(\widehat{x}^{\eta}_{i}(y), w_{-i})-\nabla_{x_{i}} \bar{f}^{\eta}_{i}(\widehat{x}^{\eta}_{i}(y), y_{-i})) \\
            &\quad - \mu \| \widehat{x}^{\eta}_{i}(y)-\widehat{x}^{\eta}_{i}(w) \|^{2} + \mu \| \widehat{x}^{\eta}_{i}(y)-\widehat{x}^{\eta}_{i}(w) \| \| y_{i} - w_{i} \|.
    \end{align*}
    By Assumption $\mathrm{A}$, for any $i\in [N]$, we have that
    \begin{align*}
            0 \leq - \left( \tfrac{\sigma_{i}}{\eta \sigma_{i}+1} + \mu \right) &\| \widehat{x}^{\eta}_{i}(y)-\widehat{x}^{\eta}_{i}(w) \|^{2} + \bar{L}_{-i} \| \widehat{x}^{\eta}_{i}(y)-\widehat{x}^{\eta}_{i}(w) \| \| y_{-i} - w_{-i} \| \\
            &\quad + \mu \| \widehat{x}^{\eta}_{i}(y)-\widehat{x}^{\eta}_{i}(w) \| \| y_{i} - w_{i} \| \\
       \implies      \| \widehat{x}^{\eta}_{i}(y)-\widehat{x}^{\eta}_{i}(w) \| &\leq \tfrac{\mu}{\tfrac{\sigma_{i}}{\eta \sigma_{i}+1} + \mu} \| y_{i} - w_{i} \| + \tfrac{\bar{L}_{-i}}{\tfrac{\sigma_{i}}{\eta \sigma_{i}+1} + \mu} \| y_{-i} - w_{-i} \| \\
            &\leq \tfrac{\mu}{\tfrac{\sigma_{i}}{\eta \sigma_{i}+1} + \mu} \| y_{i} - w_{i} \| + \tfrac{\bar{L}_{-i}}{\tfrac{\sigma_{i}}{\eta \sigma_{i}+1} + \mu} \sum_{j\neq i} \| y_{j} - w_{j} \|
        \end{align*}
    holds for any $i\in [N]$. Therefore, we deduce that
    \begin{equation*}
        \begin{bmatrix}
            \| \widehat{x}^{\eta}_{1}(y)-\widehat{x}^{\eta}_{1}(w) \| \\
            \vdots \\
            \| \widehat{x}^{\eta}_{N}(y)-\widehat{x}^{\eta}_{N}(w) \|
        \end{bmatrix} \leq \Gamma_{1}
        \begin{bmatrix}
            \| y_{1}-w_{1} \| \\
            \vdots \\
            \| y_{N}-w_{N} \|
        \end{bmatrix}.
    \end{equation*}
    By assumption $\mathrm{(A3)}$, we have proved that $\widehat{x}^{\eta}(\bullet) = (\widehat{x}^{\eta}_{i}(\bullet))_{i=1}^{N}$ is contractive.

    The proof of (ii) follows directly from the first-order optimality condition for unconstrained strongly convex smooth optimization, together with Proposition \ref{NE-equals-NE}. The proof of (iii) is the same as that of \cite[Proposition 1]{lei-shanbhag-pang-sen-2020} and is omitted here.
\end{proof}

\begin{remark}\label{MS-SBR-remark}\em
    Our result here relaxes the C$^{2}$ smoothness assumption required by contraction-based BR schemes for solving strongly convex games \cite{facchinei-pang-2009, lei-shanbhag-pang-sen-2020} to a significantly weakened C$^{0}$ requirement. Further, the analysis in ~\cite{facchinei-pang-2009} can be weakened from C$^2$ requirements to C$^1$ via the above approach which only necessitates Lipschitzian rather than eigenvalue bounds. Two additional remarks are made here. (i) It can be seen from \eqref{contraction-matrix_1} that if $(N-1)\bar{L}_{-i} < \tfrac{\sigma_{i}}{\eta\sigma_{i}+1}$ holds for any $i\in [N]$ (where $\bar{L}_{-i}$ depends on $\eta$), we have that $\| \Gamma_{1} \|_{\infty} < 1$. (ii) Such an avenue does not work for weakly convex settings, which are addressed by incorporating surrogation into our smoothing-enabled BR framework to address weak convexity. $\hfill \Box$
\end{remark}

Next, we derive a  linear rate for the iterates generated by \textbf{MS-SBR}.

\begin{theorem}[Linear rate of {\bf MS-SBR}]\em \label{linear-rate-MS-SBR}
    Consider \textbf{MS-SBR} where for any $i\in [N]$, we have that $\mathbb{E}[\|x_{i}^{0} - x^{\ast}_{i}\|] \leq C_{1}$ for some $C_{1} > 0$ and $\varepsilon_{i}^{k} \triangleq \nu^{k+1}$ for some $\nu \in (0,1)$. Suppose that Assumption $\mathrm{A}$ holds. Define $c_{1} \triangleq \max\{ \| \Gamma_{1} \|, \nu \}$ and
    \begin{equation}\label{expected-error}
        e_{k} \triangleq \mathbb{E} \left[ \left\| 
        \begin{bmatrix}
            \|x^{k}_{1} - x^{\ast}_{1}\| \\
            \vdots \\
            \|x^{k}_{N} - x^{\ast}_{N}\|
        \end{bmatrix}
        \right\| \right].
    \end{equation}
    Then for any $q_{1}\in (c_{1}, 1)$ and $D_{1} \triangleq 1/\log{((q_{1}/c_{1})^{e})}$, 
        $e_{k}\leq \sqrt{N}(C_{1} + D_{1})q_{1}^{k},~ \forall k\geq 0.$
\end{theorem}
\begin{proof}
    The proof is identical to that of \cite[Proposition 3]{lei-shanbhag-pang-sen-2020}, we omit it here.
\end{proof}

One may ask how to tractably verify assumption $\mathrm{(A2)}$. We proceed to demonstrate such verifiability by considering a setting that subsumes many classical games, such as Nash–Cournot games.

\begin{lemma}\em Consider an $N$-player game where the $i$th player-specific objective function is given by $f_{i}(\bullet, x_{-i}) \triangleq g_{i}(\bullet) + p_{i}(x_{-i})^{\top} (\bullet)$, where $g_{i}(\bullet)$ is convex. Suppose that $\| p_{i}(x_{-i}) - p_{i}(y_{-i}) \| \leq L^{p}_{-i} \| x_{-i} - y_{-i} \|$ holds for some $L^{p}_{-i} > 0$ and any $x_{-i}, y_{-i}\in X_{-i}$. Then Assumption $\mathrm{(A2)}$ holds, i.e., $\|\nabla_{x_{i}}\bar{f}^{\eta}_{i}(x_{i}, x_{-i}) - \nabla_{x_{i}}\bar{f}^{\eta}_{i}(x_{i}, y_{-i}) \|$ $\leq L^{p}_{-i} \| x_{-i} - y_{-i} \|$ holds for given $x_{i}\in X_{i}$ and any $x_{-i}, y_{-i}\in X_{-i}$.
\end{lemma}
\begin{proof}
    The gradient is given by $\nabla_{x_{i}} \bar{f}^{\eta}_{i}(x_{i}, x_{-i}) = \tfrac{1}{\eta}(x_{i} - \mathrm{prox}_{\eta \bar{f}_{i}(\bullet, x_{-i})}(x_{i}))$. Since $\bar{f}_{i}(x_{i}, x_{-i}) = \bar{g}_{i}(x_{i}) + p_{i}(x_{-i})^{\top}x_{i}$ where $\bar{g}_{i}(x_{i}) \, \triangleq \, g_{i}(x_{i}) + \mathbf{1}_{X_{i}}(x_{i})$, it follows that $\mathrm{prox}_{\eta \bar{f}_{i}(\bullet, x_{-i})}(x_{i}) = \mathrm{prox}_{\eta \bar{g}_{i}(\bullet)}(x_{i} - \eta p_{i}(x_{-i}))$ and
    \begin{align*}
            &\| \nabla_{x_{i}}\bar{f}^{\eta}_{i}(x_{i}, x_{-i}) - \nabla_{x_{i}}\bar{f}^{\eta}_{i}(x_{i}, y_{-i}) \| = \tfrac{1}{\eta} \| \mathrm{prox}_{\eta \bar{f}_{i}(\bullet, x_{-i})}(x_{i}) - \mathrm{prox}_{\eta \bar{f}_{i}(\bullet, y_{-i})}(x_{i}) \| \\
            &= \tfrac{1}{\eta} \| \mathrm{prox}_{\eta \bar{g}_{i}(\bullet)}(x_{i} - \eta p_{i}(x_{-i})) - \mathrm{prox}_{\eta \bar{g}_{i}(\bullet)}(x_{i} - \eta p_{i}(y_{-i})) \| \\
            &\leq \tfrac{1}{\eta} \| \eta (p_{i}(x_{-i}) - p_{i}(y_{-i})) \| = \| p_{i}(x_{-i}) - p_{i}(y_{-i}) \| \leq L^{p}_{-i} \| x_{-i} - y_{-i} \|,
        \end{align*}
    where the penultimate inequality is due to the nonexpansiveness of the proximal operator \cite[Theorem 6.42]{beck-2017} and the fact that $\bar{g}_{i}(\bullet)$ is convex.
\end{proof}

\subsubsection{Complexity analysis}

Inspired by \cite[Section 2.4]{jalilzadeh-shanbhag-blanchet-glynn-2022}, we employ the inexact Moreau-smoothed gradient method (IMGM) for solving \eqref{MS-SBR-eqn2} to obtain an inexact BR solution. For given $x^{k}$ and iteration times $j_{i, k}$, we initialize $z^{k, 0}_{i} \triangleq x^{k}_{i}$, then implement the following inexact gradient descent steps from $t = 0, \cdots, j_{i, k}-1$:
\begin{equation}\label{IMGM}
    z^{k, t+1}_{i} \triangleq z^{k, t}_{i} - \gamma (\nabla_{x_{i}} \bar{f}^{\eta}_{i}(z^{k, t}_{i}, x^{k}_{-i}) + \mu(z^{k, t}_{i} - x^{k}_{i}) + \bar{w}^{k,t}_{i}),~\forall i\in [N], \tag{IMGM}
\end{equation}
where $\gamma$ is a constant stepsize. After $j_{i,k}$ steps, we set $x^{k+1}_{i} = z^{k, j_{i,k}}_{i}$. We know that $\nabla_{x_{i}} \bar{f}^{\eta}_{i}(z^{k, t}_{i}, x^{k}_{-i}) = \tfrac{1}{\eta}(z^{k, t}_{i} - \mathrm{prox}_{\eta \bar{f}_{i}(\bullet, x^{k}_{-i})}(z^{k, t}_{i}))$. Since $\mathrm{prox}_{\eta \bar{f}_{i}(\bullet, x^{k}_{-i})}(z^{k, t}_{i})$ is an exact solution of a stochastic strongly convex problem, which is unavailable in finite time, we may compute increasingly exact analogs $\widehat{\mathrm{prox}}_{\eta \bar{f}_{i}(\bullet, x^{k}_{-i})}(z^{k, t}_{i})$. We want to set
\begin{align}\label{IMGM-eqn1}
    \nabla_{x_{i}} \bar{f}^{\eta}_{i}(z^{k, t}_{i}, x^{k}_{-i}) + \bar{w}^{k,t}_{i} = \tfrac{1}{\eta}(z^{k, t}_{i} - \widehat{\mathrm{prox}}_{\eta \bar{f}_{i}(\bullet, x^{k}_{-i})}(z^{k, t}_{i})) \\
\label{IMGM-eqn2}
\implies    \bar{w}^{k,t}_{i} = \tfrac{1}{\eta}(\mathrm{prox}_{\eta \bar{f}_{i}(\bullet, x^{k}_{-i})}(z^{k, t}_{i}) - \widehat{\mathrm{prox}}_{\eta \bar{f}_{i}(\bullet, x^{k}_{-i})}(z^{k, t}_{i})).
\end{align}
By plugging \eqref{IMGM-eqn1} into \eqref{IMGM}, we can see that in practice, we update $z^{k, t+1}_{i}$ as
\begin{equation*}
    z^{k, t+1}_{i} \triangleq z^{k, t}_{i} - \gamma \left( \tfrac{1}{\eta}(z^{k, t}_{i} - \widehat{\mathrm{prox}}_{\eta \bar{f}_{i}(\bullet, x^{k}_{-i})}(z^{k, t}_{i})) + \mu(z^{k, t}_{i} - x^{k}_{i}) \right),~\forall i\in [N]
\end{equation*}
We employ the projected stochastic subgradient method (PSSM) \cite[Section
2.1]{nemirovski-juditsky-lan-shapiro-2009} for obtaining the inexact
solution $\widehat{\mathrm{prox}}_{\eta \bar{f}_{i}(\bullet,
x^{k}_{-i})}(z^{k, t}_{i})$. For given $z^{k, t}_{i}$ and iteration times
$T^{k}_{i}(t)$, we initialize $y^{k, t}_{i, 0} \triangleq z^{k,t}_{i}$ and
take stochastic subgradient steps from $t^{\prime} = 0,
\cdots, T^{k}_{i}(t) - 1$, as specified by
\begin{equation}\label{PSSM}
    y^{k, t}_{i, t^{\prime}+1} = \Pi_{X_{i}} \left[ y^{k, t}_{i, t^{\prime}} - \tfrac{1}{(\sigma_{i}+\eta^{-1})(t^{\prime}+1)} g_{i}(y^{k, t}_{i, t^{\prime}}, x^{k}_{-i}, \xi^{k, t}_{i, t^{\prime}}) \right],~\forall i\in [N], \tag{PSSM}
\end{equation}
where $g_{i}(y^{k, t}_{i, t^{\prime}}, x^{k}_{-i}, \xi^{k, t}_{i, t^{\prime}}) \in \partial_{x_{i}} \Tilde{f}_{i}(y^{k, t}_{i, t^{\prime}}, x^{k}_{-i}, \xi^{k, t}_{i, t^{\prime}}) + \tfrac{1}{\eta}(y^{k, t}_{i, t^{\prime}} - z^{k,t}_{i})$ is an selection of the unbiased subgradient. After $T^{k}_{i}(t)$ steps, we set $\widehat{\mathrm{prox}}_{\eta \bar{f}_{i}(\bullet, x^{k}_{-i})}(z^{k, t}_{i}) \triangleq y^{k, t}_{i, T^{k}_{i}(t)}$. The following lemma shows the linear rate in expectation of \eqref{IMGM}. The proof is similar to that of \cite[Theorem 3-(i)]{jalilzadeh-shanbhag-blanchet-glynn-2022}.

\begin{proposition}[Linear rate of \ref{IMGM}]\label{linear-rate-IMGM}\em
    Suppose that assumption $\mathrm{(A1)}$ holds and the subgradient oracles are bounded. Consider the iterates $\{z^{k, t}_{i}\}_{t\geq 0}$ generated by \eqref{IMGM} with constant stepsize $\gamma = 1/(\eta^{-1}+\mu)$. Suppose we employ $T^{k}_{i}(t) \triangleq \lfloor T^{k}_{i}(0) \beta^{-(t+1)} \rfloor$ steps of \eqref{PSSM} for some $\beta \in (0, 1)$ and some $T^{k}_{i}(0) > 0$ in the $t^{\rm th}$ update of \eqref{IMGM}. Then for any $k$, we have
    \begin{equation*}
        \mathbb{E}[ \| z^{k,t}_{i} - \widehat{x}^{\eta}_{i}(x^{k}) \|^{2} \,\vert\, x^{k} ] \leq \Theta \widehat{p}^{t},~ \forall t\geq 0,
    \end{equation*}
    where $\widehat{p}\in (\beta, 1)$ and $\Theta \geq 1$ depends on the second-moment bound of the stochastic subgradient oracles.
\end{proposition}
\begin{proof}
    Before proceeding, we define $\bar{z}^{k, t+1}_{i}$ as
    \begin{equation*}
        \bar{z}^{k, t+1}_{i} \triangleq z^{k, t}_{i} - \gamma (\nabla_{x_{i}} \bar{f}^{\eta}_{i}(z^{k, t}_{i}, x^{k}_{-i}) + \mu(z^{k, t}_{i} - x^{k}_{i})),~\forall i\in [N].
    \end{equation*}
    Since we choose constant stepsize $\gamma = 1/(\eta^{-1}+\mu)$, by invoking \cite[Theorem 3.10]{bubeck-2015}, we may bound $\| \bar{z}^{k, t+1}_{i} - \widehat{x}^{\eta}_{i}(x^{k}) \|^{2} \leq q\| z^{k, t}_{i} - \widehat{x}^{\eta}_{i}(x^{k}) \|^{2}$, where $q \triangleq 1 - \tfrac{1}{\kappa} \in (0, 1)$ such that $\kappa \triangleq \tfrac{1/\eta + \mu}{\sigma_{i}/(\eta \sigma_{i}+1) + \mu}$. Since $q\in (0, 1)$, there exists some $\delta > 0$ such that $(1+\delta)q < 1$. Therefore, we may obtain that
    \begin{align}\label{linear-rate-IMGM-eqn1}
        & \mathbb{E}[ \| z^{k, t+1}_{i} - \widehat{x}^{\eta}_{i}(x^{k}) \|^{2} \,\vert\, x^{k} ] \leq (1+\tfrac{1}{\delta}) \mathbb{E}[ \| z^{k, t+1}_{i} - \bar{z}^{k, t+1}_{i} \|^{2} \,\vert\, x^{k} ] + (1+\delta) \mathbb{E}[ \| \bar{z}^{k, t+1}_{i} - \widehat{x}^{\eta}_{i}(x^{k}) \|^{2} \,\vert\, x^{k} ] \notag \\
        & \leq (1+\tfrac{1}{\delta}) \mathbb{E}[ \| z^{k, t+1}_{i} - \bar{z}^{k, t+1}_{i} \|^{2} \,\vert\, x^{k} ] + (1+\delta) q \mathbb{E}[ \| z^{k, t}_{i} - \widehat{x}^{\eta}_{i}(x^{k}) \|^{2} \,\vert\, x^{k} ] \notag \\
        &= (1+\tfrac{1}{\delta}) \tfrac{\gamma^{2}}{\eta^{2}} \mathbb{E}[ \| \widehat{\mathrm{prox}}_{\eta \bar{f}_{i}(\bullet, x^{k}_{-i})}(z^{k, t}_{i})) - \mathrm{prox}_{\eta \bar{f}_{i}(\bullet, x^{k}_{-i})}(z^{k, t}_{i})) \|^{2} \,\vert\, x^{k} ] + (1+\delta) q \mathbb{E}[ \| z^{k, t}_{i} - \widehat{x}^{\eta}_{i}(x^{k}) \|^{2} \,\vert\, x^{k} ] \notag \\
        &< (1+\tfrac{1}{\delta}) \mathbb{E}[ \| \widehat{\mathrm{prox}}_{\eta \bar{f}_{i}(\bullet, x^{k}_{-i})}(z^{k, t}_{i})) - \mathrm{prox}_{\eta \bar{f}_{i}(\bullet, x^{k}_{-i})}(z^{k, t}_{i})) \|^{2} \,\vert\, x^{k} ] + (1+\delta) q \mathbb{E}[ \| z^{k, t}_{i} - \widehat{x}^{\eta}_{i}(x^{k}) \|^{2} \,\vert\, x^{k} ],
    \end{align}
    where the last inequality is due to the fact that $\gamma = 1/(\eta^{-1}+\mu) < \eta$. By \cite[Section 2.1]{nemirovski-juditsky-lan-shapiro-2009}, we have
    \begin{equation*}
        \mathbb{E}[ \| \widehat{\mathrm{prox}}_{\eta \bar{f}_{i}(\bullet, x^{k}_{-i})}(z^{k, t}_{i})) - \mathrm{prox}_{\eta \bar{f}_{i}(\bullet, x^{k}_{-i})}(z^{k, t}_{i})) \|^{2} \,\vert\, x^{k} ] \leq \tfrac{Q}{T^{k}_{i}(t)},
    \end{equation*}
    where $Q>0$ depends on the second-moment bound of the stochastic subgradient oracles. Then it leads to
    \begin{equation*}
        \mathbb{E}[ \| z^{k, t+1}_{i} - \widehat{x}^{\eta}_{i}(x^{k}) \|^{2} \,\vert\, x^{k} ] \leq (1+\delta) q \mathbb{E}[ \| z^{k, t}_{i} - \widehat{x}^{\eta}_{i}(x^{k}) \|^{2} \,\vert\, x^{k} ] + \tfrac{(1+1/\delta)Q}{T^{k}_{i}(t)}.
    \end{equation*}
    Let $\beta \triangleq (1+\delta)q < 1$ and $T^{k}_{i}(t) \triangleq \lfloor T^{k}_{i}(0) \beta^{-(t+1)} \rfloor$ where $T^{k}_{i}(0) > 0$. Consequently, we may deduce the following result
    \begin{align*}
        \mathbb{E}[ \| z^{k, t+1}_{i} - \widehat{x}^{\eta}_{i}(x^{k}) \|^{2} \,\vert\, x^{k} ] &\leq \beta \mathbb{E}[ \| z^{k, t}_{i} - \widehat{x}^{\eta}_{i}(x^{k}) \|^{2} \,\vert\, x^{k} ] + \tfrac{(1+1/\delta)Q \beta^{t+1}}{T^{k}_{i}(0)} \\
        &\leq \beta \left( \beta \mathbb{E}[ \| z^{k, t-1}_{i} - \widehat{x}^{\eta}_{i}(x^{k}) \|^{2} \,\vert\, x^{k} ] + \tfrac{(1+1/\delta)Q \beta^{t}}{T^{k}_{i}(0)} \right) + \tfrac{(1+1/\delta)Q \beta^{t+1}}{T^{k}_{i}(0)} \\
        &= \beta^{2} \mathbb{E}[ \| z^{k, t-1}_{i} - \widehat{x}^{\eta}_{i}(x^{k}) \|^{2} \,\vert\, x^{k} ] + 2\tfrac{(1+1/\delta)Q \beta^{t+1}}{T^{k}_{i}(0)} \\
        &\leq \cdots \leq \beta^{t+1} \mathbb{E}[ \| z^{k, 0}_{i} - \widehat{x}^{\eta}_{i}(x^{k}) \|^{2} \,\vert\, x^{k} ] + (t+1)\tfrac{(1+1/\delta)Q \beta^{t+1}}{T^{k}_{i}(0)}.
    \end{align*}
    It can be shown that there exists some $\widehat{p}\in (\beta, 1)$ such that $(t+1)\beta^{t+1} < \widehat{D}\widehat{p}^{t+1}$ for $t\geq 0$ and $\widehat{D} > \tfrac{1}{\log{(\widehat{p}/\beta)^{e}}}$. Therefore, we arrive that
    \begin{equation*}
        \mathbb{E}[ \| z^{k,t}_{i} - \widehat{x}^{\eta}_{i}(x^{k}) \|^{2} \,\vert\, x^{k} ] \leq \Theta \widehat{p}^{t},
    \end{equation*}
    where $\Theta = \max\{ \mathbb{E}[ \| z^{k, 0}_{i} - \widehat{x}^{\eta}_{i}(x^{k}) \|^{2} \,\vert\, x^{k} ] + \tfrac{(1+1/\delta)Q \widehat{D}}{T^{k}_{i}(0)}, 1\}$. Then we complete the proof.
\end{proof}

The following lemma  is useful in our subsequent analysis.

\begin{lemma}\label{log-log-lemma}\em
    For any $A > 0$ and $B > 0$, we have $A^{\log{(B)}} = B^{\log{(A)}}$.
\end{lemma}
\begin{proof}
    It can be observed that $A^{\log{(B)}} = A^{\tfrac{\log_{A}(B)}{\log_{A}(e)}} = (A^{\log_{A}(B)})^{\tfrac{1}{\log_{A}(e)}} = B^{\log{A}}$.
\end{proof}

Now we are ready to establish the complexities of {\bf MS-SBR}.

\begin{theorem}[Complexities of MS-SBR]\label{complexity-MS-SBR}\em
    Consider \textbf{MS-SBR} where the stochastic BR problem \eqref{MS-SBR-eqn2} is computed via \eqref{IMGM}. Suppose Assumption $\mathrm{A}$ holds and each subdifferential $\partial_{x_{i}} \Tilde{f}_{i}(\bullet, x_{-i}, \xi)$ is bounded over $X_{i}$ for given $x_{-i}$ and $\xi$. Under the same settings as in Theorem~\ref{linear-rate-MS-SBR} and Proposition~\ref{linear-rate-IMGM}, the following hold.\\
    (i) (Iteration complexity) Given an $\epsilon > 0$, we have that $e_{k}\leq \epsilon$ where $e_{k}$ is defined in \eqref{expected-error} after at most $K(\epsilon)$ iterations, where $K(\epsilon)$ is defined as
    \begin{equation}\label{iteration-complexity-MS-SBR}
        K(\epsilon) \triangleq \left\lceil \tfrac{\log(\sqrt{N}(C_{1}+D_{1})/\epsilon)}{\log{(1/q_{1})}} \right\rceil.
    \end{equation}
    (ii) (Sample complexity) Suppose $T^{k}_{i}(0)\leq T$ holds for some $T > 0$ and any $i$ and $k$. We choose $\nu \triangleq \|\Gamma_{1}\|$, $\delta_{1} > 0$ such that $\widehat{p} \triangleq \beta^{1/(1+\delta_{1})} < 1$, and $\delta_{2} > 0$ such that $q_{1} \triangleq v^{1/(1+\delta_{2})} < 1$. Define $\delta \triangleq \max\{ \delta_{1}, \delta_{2} \}$ and $\bar{v}\triangleq (1/v^{2})^{(\log{(1/\beta)}/\log{(1/\widehat{p})})} > 1$. Then the overall sample complexity $S_i(\epsilon)$ for player $i$ to achieve $e_{k}\leq \epsilon$ can be upper bounded by
    \begin{equation}\label{sample-complexity-MS-SBR}
       S_i(\epsilon)\leq \tfrac{T\bar{v}^{2}}{(1-\beta)(\bar{v}-1)} \left( \tfrac{\sqrt{N}(C_{1}+D_{1})}{\epsilon} \right)^{2(1+\delta)^{2}}.
    \end{equation}
\end{theorem}
\begin{proof}
    (i) By Theorem \ref{linear-rate-MS-SBR}, in order to ensure $e_{k}\leq \epsilon$, it suffices to have that $\sqrt{N}(C_{1} + D_{1})q_{1}^{k}\leq \epsilon$, which leads to \eqref{iteration-complexity-MS-SBR} immediately.  (ii) We first consider the $k$th outer iteration and give an expression of $j_{i,k}$ in \eqref{IMGM}. Recall that we set $\varepsilon_{i}^{k} \triangleq \nu^{k+1}$ in Theorem \ref{linear-rate-MS-SBR}. By Proposition \ref{linear-rate-IMGM}, since $\Theta \geq 1$, we have
    \begin{equation*}
        \mathbb{E}[ \|x^{k+1}_{i} - \widehat{x}^{\eta}_{i}(x^{k})\|^{2} \,\vert\, x^{k}] = \mathbb{E}[ \|z^{k, j_{i,k}}_{i} - \widehat{x}^{\eta}_{i}(x^{k})\|^{2} \,\vert\, x^{k}] \leq \Theta \widehat{p}^{j_{i,k}} \leq (v^{k+1})^{2},
    \end{equation*}
    implying that $j_{i,k} = \left\lceil \tfrac{\log{(1/v^{2(k+1)})}}{\log{(1/\widehat{p})}} \right\rceil$. Therefore at $k$th outer iteration, we may bound $\sum_{t=0}^{j_{i,k}-1} T^{k}_{i}(t)$ as
    \begin{align*}
        \sum_{t=0}^{j_{i,k}-1} T^{k}_{i}(t) &\leq \sum_{t=0}^{j_{i,k}-1} T^{k}_{i}(0) \beta^{-t} \leq T\sum_{t=0}^{j_{i,k}-1} \beta^{-t} \leq \tfrac{\beta T}{1 - \beta} \left( \tfrac{1}{\beta} \right)^{j_{i,k}} \\
        &\leq \tfrac{T}{1-\beta} \left( \tfrac{1}{\beta} \right)^{\tfrac{\log{(1/v^{2(k+1)})}}{\log{(1/\widehat{p})}}} \overset{\text{Lemma } \ref{log-log-lemma}}{=} \tfrac{T}{1-\beta} \left( \tfrac{1}{v^{2(k+1)}} \right)^{\tfrac{\log{(1/\beta)}}{\log{(1/\widehat{p})}}},
    \end{align*}
    where $\beta\in (0, 1)$. Therefore, during $K(\epsilon)$ iterations, the overall sample complexity of player $i$, denoted by $S_i(\epsilon)$, can be bounded as
    \begin{align}\label{complexity-MS-SBR-eqn1}
        S_i(\epsilon) &= \sum_{k=0}^{K(\epsilon)-1} \sum_{t=0}^{j_{i,k}-1} T^{k}_{i}(t) \leq \sum_{k=0}^{K(\epsilon)-1}\! \tfrac{T}{1-\beta} \left( \tfrac{1}{v^{2(k+1)}} \right)^{\tfrac{\log{(1/\beta)}}{\log{(1/\widehat{p})}}} \notag \\
        &= \tfrac{T}{1-\beta} \sum_{k=0}^{K(\epsilon)-1}\! \left( \tfrac{1}{v^{2(k+1)}} \right)^{\tfrac{\log{(1/\beta)}}{\log{(1/\widehat{p})}}} \leq \tfrac{T\bar{v}}{(1-\beta)(\bar{v}-1)} (\bar{v})^{K(\epsilon)} \notag \\
        &\leq \tfrac{T\bar{v}^{2}}{(1-\beta)(\bar{v}-1)} (\bar{v})^{\tfrac{\log(\sqrt{N}(C_{1}+D_{1})/\epsilon)}{\log{(1/q_{1})}}} \overset{\text{Lemma } \ref{log-log-lemma}}{=} \tfrac{T\bar{v}^{2}}{(1-\beta)(\bar{v}-1)} \left( \tfrac{\sqrt{N}(C_{1}+D_{1})}{\epsilon} \right)^{\tfrac{\log{(\bar{v})}}{\log{(1/q_{1})}}}, 
    \end{align}
    where $\bar{v}>1$ is defined as $\bar{v} \triangleq (1/v^{2})^{\tfrac{\log{(1/\beta)}}{\log{(1/\widehat{p})}}}$. Since we define $\widehat{p} \triangleq \beta^{1/(1+\delta_{1})} < 1$, $q_{1} \triangleq v^{1/(1+\delta_{2})} < 1$, and $\delta = \max\{ \delta_{1}, \delta_{2} \}$, it follows that
    \begin{equation}\label{complexity-MS-SBR-eqn2}
        \tfrac{\log{(\bar{v})}}{\log{(1/q_{1})}} = \tfrac{2\log{(1/\beta)}}{\log{(1/\widehat{p})}} \tfrac{\log{(1/v)}}{\log{(1/q_{1})}} = 2(1+\delta_{1})(1+\delta_{2}) \leq 2(1+\delta)^{2}.
    \end{equation}
    By plugging \eqref{complexity-MS-SBR-eqn2} into \eqref{complexity-MS-SBR-eqn1}, we arrive the desired bound \eqref{sample-complexity-MS-SBR}.
\end{proof}

\subsection{MS-ABR under strong convexity}\label{Sec-3.2}

An alternate avenue for deriving convergence guarantees lies in leveraging potentiality, a notion that originates from the seminal work of Monderer and Shapley \cite{monderer-shapley-1996}. We present an asynchronous BR scheme based on potentiality, where a non-asymptotic sublinear rate and a complexity guarantee are derived for a suitably defined residual, the latter of which is novel.

\begin{algorithm}[htb]\caption{Moreau-smoothed Asynchronous BR (\textbf{MS-ABR})}\label{MS-ABR}
{\it Initialize:} Initialize $k = 0$, given $\eta, K, \{\varepsilon_i^k\}_{i,k}$, and $x^{0} = (x_{i}^{0})_{i=1}^{N}\in X$. Let $0 < p_{i} < 1$ be the probability of selecting player $i$ such that $\sum_{i=1}^{N}p_{i}=1$.

{\it Iterate until $k\ge K$:} 

\noindent(i) Pick a player $i(k) \in [N]$ with probability $p_{i(k)}$.

\noindent(ii) The player $i(k)$ updates her strategy $x_{i(k)}^{k+1}$ as 
\begin{equation}\label{MS-ABR-eqn1}
    x_{i(k)}^{k+1}\in \{ z\in X_{i(k)}: \mathbb{E}[\|z-\widehat{x}^{\eta}_{i(k)}(x^{k})\|^2 \,\vert\, x^{k}] \leq (\varepsilon_{i(k)}^{k})^2 ~ \textrm{a.s.} \},
\end{equation}
where the BR solution $\widehat{x}^{\eta}_{i(k)}(x^{k})$ is defined as
\begin{equation}\label{MS-ABR-eqn2}
    \widehat{x}^{\eta}_{i(k)}(x^{k}) \triangleq \argmin\limits_{z_{i}\in \mathbb{R}^{n_{i(k)}}} \left\{ \bar{f}^{\eta}_{i(k)}(z_{i}, x^{k}_{-i(k)}) + \tfrac{\mu}{2}\| z_{i} - x^{k}_{i(k)} \|^{2} \right\}.
\end{equation}
Other players' strategies remain invariant, i.e., $x_{i}^{k+1} = x_{i}^{k}$ and $\varepsilon^{k}_{i} = 0$ for $i\neq i(k)$.

{\it Return:} Return $x^{K}$ as final estimate.
\end{algorithm}

In this subsection, we define the residual $G^{\eta}_{n}(x) \triangleq (G^{\eta}_{n_{i}}(x))_{i=1}^{N}$, where the $i$th component is defined as $G^{\eta}_{n_{i}}(x) \triangleq \nabla_{x_{i}} \bar{f}^{\eta}_{i}(x_{i}, x_{-i})$.
\begin{lemma}\label{cvx_Fnat} \em
    For any $k\geq 0$, 
        $\| G^{\eta}_{n_{i}}(x^{k}) \| \leq \tfrac{1 + |1-\gamma\mu| + \gamma/\eta}{\gamma} \| \widehat{x}^{\eta}_{i}(x^{k}) - x^{k}_{i} \|$
    holds for any $\gamma > 0$. If  $\gamma > 0$ such that $\gamma\mu > 1$, $ \| G^{\eta}_{n_{i}}(x^{k}) \| \leq (\mu + 1/\eta) \| \widehat{x}^{\eta}_{i}(x^{k}) - x^{k}_{i} \|$.
\end{lemma}
\begin{proof}
    For any $k\geq 0$, we have that the following holds for any $\gamma > 0$.
    \begin{align*}
            &\| G^{\eta}_{n_{i}}(x^{k}) \| = \tfrac{1}{\gamma} \| \gamma \nabla_{x_{i}} \bar{f}^{\eta}_{i}(x^{k}_{i}, x^{k}_{-i}) \| \\
            &= \tfrac{1}{\gamma} \| x^{k}_{i} - ( x^{k}_{i} - \gamma \nabla_{x_{i}} \bar{f}^{\eta}_{i}(x^{k}_{i}, x^{k}_{-i}) ) + \widehat{x}^{\eta}_{i}(x^{k}) - \widehat{x}^{\eta}_{i}(x^{k}) \| \\
            &\leq \tfrac{1}{\gamma} \| x^{k}_{i} - \widehat{x}^{\eta}_{i}(x^{k}) \| + \underbrace{\tfrac{1}{\gamma} \| \widehat{x}^{\eta}_{i}(x^{k}) - ( x^{k}_{i} - \gamma \nabla_{x_{i}} \bar{f}^{\eta}_{i}(x^{k}_{i}, x^{k}_{-i}) ) \|}_{(i)}.
    \end{align*}
     Since $\widehat{x}^{\eta}_{i}(x^{k})$ is a solution of the unconstrained BR problem \eqref{MS-ABR-eqn2}, it follows that
    \begin{align*}
            &\mathrm{(i)} = \tfrac{1}{\gamma} \| \widehat{x}^{\eta}_{i}(x^{k}) - \gamma \underbrace{( \nabla_{x_{i}}\bar{f}^{\eta}_{i}(\widehat{x}^{\eta}_{i}(x^{k}), x^{k}_{-i}) + \mu(\widehat{x}^{\eta}_{i}(x^{k})-x^{k}_{i}) )}_{=0} - x^{k}_{i} + \gamma \nabla_{x_{i}} \bar{f}^{\eta}_{i}(x^{k}_{i}, x^{k}_{-i}) \| \\
            &= \tfrac{1}{\gamma} \| (\widehat{x}^{\eta}_{i}(x^{k}) - x^{k}_{i}) - \gamma\mu(\widehat{x}^{\eta}_{i}(x^{k}) - x^{k}_{i}) - \gamma(\nabla_{x_{i}}\bar{f}^{\eta}_{i}(\widehat{x}^{\eta}_{i}(x^{k}), x^{k}_{-i}) - \nabla_{x_{i}}\bar{f}^{\eta}_{i}(x^{k}_{i}, x^{k}_{-i})) \| \\
            &\leq \tfrac{| 1 - \gamma\mu |}{\gamma} \| \widehat{x}^{\eta}_{i}(x^{k}) - x^{k}_{i} \| + \tfrac{\gamma/\eta}{\gamma} \| \widehat{x}^{\eta}_{i}(x^{k}) - x^{k}_{i} \|,
    \end{align*}
    where the last inequality is due to Proposition~\ref{cvx-ME-property}-(ii). Then we have
    \begin{equation*}
        \| G^{\eta}_{n_{i}}(x^{k}) \| \leq \tfrac{1 + |1-\gamma\mu| + \gamma/\eta}{\gamma} \| \widehat{x}^{\eta}_{i}(x^{k}) - x^{k}_{i} \|.
    \end{equation*}
    If we choose $\gamma > 0$ such that $\gamma\mu > 1$, $ \| G^{\eta}_{n_{i}}(x^{k}) \| \leq (\mu + 1/\eta) \| \widehat{x}^{\eta}_{i}(x^{k}) - x^{k}_{i} \|$.
\end{proof}

We consider the following assumption in this subsection.

\vspace{5pt}

\emph{Assumption $\mathrm{B}$.} (B1) Each $X_{i}$ is additionally compact with diameter $d_{i} > 0$. We define $d \triangleq \max_{i\in [N]} d_{i}$. (B2) (Potentiality) Given $\eta > 0$, there exists a potential function $P: X\to \mathbb{R}$ such that $\bar{f}^{\eta}_{i} (x_{i}, x_{-i}) - \bar{f}^{\eta}_{i} (y_{i}, x_{-i}) = P (x_{i}, x_{-i}) - P (y_{i}, x_{-i})$ holds for any $x_{i}, y_{i}\in X_{i}$ and any $x_{-i} \in X_{-i}$ for all $i\in [N]$.

\vspace{5pt}

Next, we consider a class of aggregative games for which assumption $\mathrm{(B2)}$ provably holds.

\begin{lemma}\label{assumption-B2-example} \em
    Consider the $N$-player aggregative game where each $f_{i}(x_{i}, x_{-i}) \triangleq g_{i}(x_{i}) + \sum_{j=1}^{N} h(x_{j})$ and  $g_{i}(\bullet)$ and $h(\bullet)$ are convex. Then there exists a potential function $P(x) \triangleq \sum_{i=1}^{N} [g_{i}+\mathbf{1}_{X_{i}}+h]^{\eta}(x_{i})$ such that assumption $\mathrm{(B2)}$ holds.
\end{lemma}
\begin{proof}
    By definition,  $\bar{f}_{i}(x_{i}, x_{-i}) = g_{i}(x_{i}) + \sum_{j=1}^{N} h(x_{j}) + \mathbf{1}_{X_{i}}(x_{i}) = [g_{i}+\mathbf{1}_{X_{i}}+h](x_{i}) + \sum_{j\neq i} h(x_{j}).$
    Since $(f(\bullet)+a)^{\eta}(x) = f^{\eta}(x) + a$ holds for any $a\in \mathbb{R}$ independent of $x$, 
        $\bar{f}^{\eta}_{i}(x_{i}, x_{-i}) = [g_{i}+\mathbf{1}_{X_{i}}+h]^{\eta}(x_{i}) + \sum_{j\neq i} h(x_{j}).$
    Then we may complete the proof by using the potential function definition.
\end{proof}

\subsubsection{Convergence analysis}\label{Sec-3.2.1}

We begin this subsection with the following recursion on the potential function.

\begin{theorem}[Almost sure convergence of MS-ABR]\label{MS-ABR-key-recursion}
\em Consider the stochastic $N$-player game $\mathcal{G}({\bf f}, X, \bxi)$ where each $f_{i}(\bullet, x_{-i})$ is $\sigma_{i}$-strongly convex for given $x_{-i}$. Suppose that Assumption $\mathrm{B}$ holds and $\eta > 0$. Let $\{x^{k}\}_{k\geq 0}$ be generated by \textbf{MS-ABR}.  Suppose $\mu > \tfrac{1}{2\eta}$ holds. Then for any $k\geq 0$, we have
    \begin{align}\label{MS-ABR-key-recursion-main-result}
            \mathbb{E}[ P(x^{k+1}) \,\vert\, x^{k}] \leq P(x^{k}) - \left( \mu - \tfrac{1}{2\eta} \right) \sum_{i=1}^{N} p_{i} \| \widehat{x}^{\eta}_{i}(x^{k}) - x^{k}_{i} \|^{2} + \tfrac{d}{\eta} \varepsilon^{k}_{\max},
    \end{align}
    where $\varepsilon^{k}_{\max} \triangleq \max_{i\in [N]} \varepsilon^{k}_{i}$ and $d > 0$ is defined in assumption $\mathrm{(B1)}$. If we additionally have that $\sum_{k=0}^{\infty} \varepsilon^{k}_{\max} < \infty$, every limit point of $\{x^{k}\}_{k=0}^{\infty}$ is an NE almost surely.
\end{theorem}
\begin{proof}
    By Proposition \ref{cvx-ME-property}, we know that each $\bar{f}^{\eta}_{i}(\bullet, x_{-i})$ is $\tfrac{1}{\eta}$-Lipschitz smooth. By descent lemma \cite[Lemma 5.7]{beck-2017}, we have
    \begin{equation}\label{MS-ABR-key-recursion-eqn1}
        \begin{aligned}
            & \bar{f}^{\eta}_{i(k)}(\widehat{x}^{\eta}_{i(k)}(x^{k}), x^{k}_{-i(k)}) \leq \bar{f}^{\eta}_{i(k)}(x^{k}_{i(k)}, x^{k}_{-i(k)}) \\
            &\quad + \nabla_{x_{i(k)}} \bar{f}^{\eta}_{i(k)}(x^{k}_{i(k)}, x^{k}_{-i(k)})^{\top}(\widehat{x}^{\eta}_{i(k)}(x^{k}) - x^{k}_{i(k)}) + \tfrac{1}{2\eta} \| \widehat{x}^{\eta}_{i(k)}(x^{k}) - x^{k}_{i(k)} \|^{2}.
        \end{aligned}
    \end{equation}
    By the optimality condition for the unconstrained BR problem \eqref{MS-ABR-eqn2}, it follows that
    \begin{equation}\label{MS-ABR-key-recursion-eqn2}
        0 = \nabla_{x_{i(k)}} \bar{f}^{\eta}_{i(k)}(\widehat{x}^{\eta}_{i(k)}(x^{k}), x^{k}_{-i(k)}) + \mu (\widehat{x}^{\eta}_{i(k)}(x^{k}) - x^{k}_{i(k)}).
    \end{equation}
    By strong convexity, we know that
    \begin{equation}\label{MS-ABR-key-recursion-eqn3}
        \begin{aligned}
            ( \nabla_{x_{i(k)}} \bar{f}^{\eta}_{i(k)}(\widehat{x}^{\eta}_{i(k)}(x^{k}), x^{k}_{-i(k)}) &- \nabla_{x_{i(k)}} \bar{f}^{\eta}_{i(k)}(x^{k}_{i(k)}, x^{k}_{-i(k)}) )^{\top} (\widehat{x}^{\eta}_{i(k)}(x^{k}) - x^{k}_{i(k)}) \\
            &\geq \tfrac{\sigma_{i(k)}}{\eta \sigma_{i(k)} + 1} \| \widehat{x}^{\eta}_{i(k)}(x^{k}) - x^{k}_{i(k)} \|^{2} \geq 0.
        \end{aligned}
    \end{equation}
    By combining \eqref{MS-ABR-key-recursion-eqn2} with \eqref{MS-ABR-key-recursion-eqn3}, we may obtain
    \begin{equation}\label{MS-ABR-key-recursion-eqn4}
        \begin{aligned}
            \nabla_{x_{i(k)}} \bar{f}^{\eta}_{i(k)}(x^{k}_{i(k)}, x^{k}_{-i(k)})^{\top} (\widehat{x}^{\eta}_{i(k)}(x^{k}) - x^{k}_{i(k)}) \leq - \mu \| \widehat{x}^{\eta}_{i(k)}(x^{k}) - x^{k}_{i(k)} \|^{2}.
        \end{aligned}
    \end{equation}
    By adding \eqref{MS-ABR-key-recursion-eqn1} and \eqref{MS-ABR-key-recursion-eqn4}, it leads to
    \begin{equation}\label{MS-ABR-key-recursion-eqn5}
        \bar{f}^{\eta}_{i(k)}(\widehat{x}^{\eta}_{i(k)}(x^{k}), x^{k}_{-i(k)}) \leq \bar{f}^{\eta}_{i(k)}(x^{k}_{i(k)}, x^{k}_{-i(k)}) - \left( \mu - \tfrac{1}{2\eta} \right) \| \widehat{x}^{\eta}_{i(k)}(x^{k}) - x^{k}_{i(k)} \|^{2}.
    \end{equation}
    We know that $\widehat{x}^{\eta}_{i(k)}(x^{k})\in X_{i(k)}$. By the potentiality assumption $\mathrm{(B2)}$, we derive that
    \begin{align*}
        & P(x^{k+1}) - P(x^{k}) = P(x^{k+1}_{i(k)}, x^{k}_{-i(k)}) - P(x^{k}_{i(k)}, x^{k}_{-i(k)}) \\
        &= \bar{f}^{\eta}_{i(k)}(x^{k+1}_{i(k)}, x^{k}_{-i(k)}) - \bar{f}^{\eta}_{i(k)}(x^{k}_{i(k)}, x^{k}_{-i(k)}) \\
        &= \underbrace{\bar{f}^{\eta}_{i(k)}(x^{k+1}_{i(k)}, x^{k}_{-i(k)}) - \bar{f}^{\eta}_{i(k)}(\widehat{x}^{\eta}_{i(k)}(x^{k}), x^{k}_{-i(k)})}_{(i)} \\
        &\quad + \underbrace{\bar{f}^{\eta}_{i(k)}(\widehat{x}^{\eta}_{i(k)}(x^{k}), x^{k}_{-i(k)}) - \bar{f}^{\eta}_{i(k)}(x^{k}_{i(k)}, x^{k}_{-i(k)})}_{(ii)}.
    \end{align*}
    By the mean value theorem, we know that there exists $x^{\prime}_{i(k)} = \lambda x^{k+1}_{i(k)} + (1-\lambda) \widehat{x}^{\eta}_{i(k)}(x^{k})$ for some $\lambda\in (0, 1)$ such that
    \begin{equation*}
        \begin{aligned}
            &(i) \leq \| \nabla_{x_{i(k)}} \bar{f}^{\eta}_{i(k)}(x^{\prime}_{i(k)}, x^{k}_{-i(k)}) \| \| x^{k+1}_{i(k)} - \widehat{x}^{\eta}_{i(k)}(x^{k}) \| \\
            &= \tfrac{1}{\eta} \| x^{\prime}_{i(k)} - \mathrm{prox}_{\eta \bar{f}^{\eta}_{i(k)}(\bullet, x^{k}_{-i(k)})}(x^{\prime}_{i(k)}) \| \| x^{k+1}_{i(k)} - \widehat{x}^{\eta}_{i(k)}(x^{k}) \| \leq \tfrac{d}{\eta} \| x^{k+1}_{i(k)} - \widehat{x}^{\eta}_{i(k)}(x^{k}) \|,
        \end{aligned}
    \end{equation*}
    where the last inequality is due to assumption $\mathrm{(B1)}$. For $(ii)$, by \eqref{MS-ABR-key-recursion-eqn5}, we have
    \begin{equation*}
        (ii) \leq - \left( \mu - \tfrac{1}{2\eta} \right) \| \widehat{x}^{\eta}_{i(k)}(x^{k}) - x^{k}_{i(k)} \|^{2}.
    \end{equation*}
    By combing the bounds on $(i)$ and $(ii)$, we may obtain that
    \begin{equation*}
        \begin{aligned}
            P(x^{k+1}) \leq P(x^{k}) - \left( \mu - \tfrac{1}{2\eta} \right) \| \widehat{x}^{\eta}_{i(k)}(x^{k}) - x^{k}_{i(k)} \|^{2} + \tfrac{d}{\eta} \| x^{k+1}_{i(k)} \!-\! \widehat{x}^{\eta}_{i(k)}(x^{k}) \|.
        \end{aligned}
    \end{equation*}
    By taking the conditional expectation $\mathbb{E}[\bullet \,\vert\, x^{k}]$ and invoking the conditional Jensen's inequality, it leads to
    \begin{equation*}
        \begin{aligned}
            \mathbb{E}[P(x^{k+1}) \,\vert\, x^{k}] \leq P(x^{k}) - \left( \mu - \tfrac{1}{2\eta} \right) \mathbb{E}[\| \widehat{x}^{\eta}_{i(k)}(x^{k}) - x^{k}_{i(k)} \|^{2} \,\vert\, x^{k}] + \tfrac{d}{\eta} \varepsilon^{k}_{i(k)}.
        \end{aligned}
    \end{equation*}
    Note that $\mathbb{P}[i(k)=i] = p_{i}$ and $\varepsilon^{k}_{\max} = \max_{i\in [N]} \varepsilon^{k}_{i}$, we may obtain that
    \begin{equation*}
        \mathbb{E}[ P(x^{k+1}) \,\vert\, x^{k}] \leq P(x^{k}) - \left( \mu - \tfrac{1}{2\eta} \right) \sum_{i=1}^{N} p_{i} \| \widehat{x}^{\eta}_{i}(x^{k}) - x^{k}_{i} \|^{2} + \tfrac{d}{\eta} \varepsilon^{k}_{\max},
    \end{equation*}
    as desired. The remainder of the proof showing that $\{x^{k}\}_{k=0}^{\infty}$ converges to an NE almost surely proceeds similarly to \cite[Theorem 1]{lei-shanbhag-2020}.
\end{proof}

If we consider the randomized output $x^{R_{K}}$ where $R_{K}$ is uniformly distributed in $\{ 0, 1, \cdots, K-1 \}$ where $K$ is the total iteration times, we can obtain the sublinear rate of \textbf{MS-ABR} by utilizing the residual $G^{\eta}_{n}(x)$ defined in Lemma \ref{cvx_Fnat}.

\begin{theorem}[Sublinear rate of MS-ABR]\label{MS-ABR-sublinear-rate}\em
    Consider the stochastic $N$-player game $\mathcal{G}({\bf f}, X, \bxi)$ where for any $i \in [ N ]$, $f_{i}(\bullet, x_{-i})$ is $\sigma_{i}$-strongly convex for given $x_{-i}$. Suppose Assumption $\mathrm{B}$ holds. Let $\{x^{k}\}_{k\geq 0}$ be the sequence generated by \textbf{MS-ABR}. Consider the same setting as Theorem~\ref{MS-ABR-key-recursion}. Suppose $\mu \geq 1/\eta$, $p_{i} = 1/N$ for any $i\in [N]$, $\varepsilon^{k}_{\max} = 1/K$ for any $k\in \{ 0, 1, \cdots, K-1 \}$, and $\gamma > 0$ such that $\gamma \mu > 1$. If $\bar{d}>0$ is some constant dependent on $d$, then for any $K\geq 0$, 
    \begin{equation}\label{MSABR-rate-final-result}
        \mathbb{E}[\| G^{\eta}_{n}(x^{R_{K}}) \|^{2}] \leq \tfrac{8N\bar{d}}{\eta K}.
    \end{equation}
\end{theorem}
\begin{proof}
    By taking unconditional expectation on both sides of \eqref{MS-ABR-key-recursion-main-result}, we have that
    \begin{align*}
        \mathbb{E}[ P(x^{k+1}) ] &\leq \mathbb{E}[ P(x^{k}) ] - \left( \mu - \tfrac{1}{2\eta} \right) \sum_{i=1}^{N} p_{i} \mathbb{E} [\| \widehat{x}^{\eta}_{i}(x^{k}) - x^{k}_{i} \|^{2}] + \tfrac{d}{\eta} \varepsilon^{k}_{\max}.
    \end{align*}
    By plugging in $\mu \geq 1/\eta$, $p_{i} = 1/N$, and $\varepsilon^{k}_{\max} = 1/K$, rearranging the terms, and taking the summation from $k = 0$ to $K-1$, it leads to
    \begin{equation*}
        \tfrac{1}{2\eta} \tfrac{1}{N} \sum_{k=0}^{K-1} \sum_{i=1}^{N} \mathbb{E} [\| \widehat{x}^{\eta}_{i}(x^{k}) - x^{k}_{i} \|^{2}] \leq \mathbb{E}[P(x^{0}) - P_{\min}] + \tfrac{d}{\eta} \triangleq \bar{d}.
    \end{equation*}
    By dividing both sides by $K$, choosing $\gamma > 0$ such that $\gamma\mu > 1$, by Lemma \ref{cvx_Fnat}, it leads to $\mathbb{E}[\| G^{\eta}_{n}(x^{R_{K}}) \|^{2}] \leq \tfrac{8N\bar{d}}{\eta K}$.
\end{proof}

\subsubsection{Complexity analysis}\label{Sec-3.2.3}

Based on the above sublinear rate guarantee, we are ready to derive the complexity results of \textbf{MS-ABR}. Same as \textbf{MS-SBR}, we employ \eqref{IMGM} to solve the lower-level unconstrained BR problem \eqref{MS-ABR-eqn2}.

\begin{theorem}[Complexities of MS-ABR]\label{MS-ABR-complexity}\em
    Consider the stochastic $N$-player game $\mathcal{G}({\bf f}, X, \bxi)$ where for any $i\in [ N ]$, $f_{i}(\bullet, x_{-i})$ is $\sigma_{i}$-strongly convex for given $x_{-i}$. Suppose Assumption~$\mathrm{B}$ holds. Let $\{x^{k}\}_{k\geq 0}$ be the sequence generated by \textbf{MS-ABR}. Under the settings of Proposition~\ref{linear-rate-IMGM} and Theorems~\ref{MS-ABR-key-recursion}--\ref{MS-ABR-sublinear-rate}, the following hold.\\
    (i) (Iteration complexity) We have that $\mathbb{E}[\| G^{\eta}_{n}(x^{R_{K}}) \|]\leq \epsilon$ after at most $K(\epsilon)$ iterations, where $K(\epsilon)$ is defined as
    \begin{equation*}
        K(\epsilon) \, \triangleq \,\left\lceil \tfrac{8 N \bar{d}}{\eta\epsilon^{2}} \right\rceil.
    \end{equation*}
    (ii) (Sample complexity) Suppose $T^{k}_{i}(0)\leq T$ holds for some $T > 0$ and any $i$ and $k$. We choose any $\delta > 0$ such that $\widehat{p} \triangleq \beta^{1/(1+\delta)} < 1$. Then the expected overall sample complexity $S_i(\epsilon)$ for player $i$ to achieve $\mathbb{E}[\| G^{\eta}_{n}(x^{R_{K}}) \|]\leq \epsilon$ can be upper bounded by
    \begin{equation*}
        S_i(\epsilon) \leq \tfrac{9^{3+2\delta}TN^{2+2\delta}\bar{d}^{3+2\delta}}{(1-\beta)\eta^{(3+2\delta)}\epsilon^{(6+4\delta)}}.
    \end{equation*}
\end{theorem}
\begin{proof}
    The bound on $K(\epsilon)$ in (i) follows from Theorem \ref{MS-ABR-sublinear-rate} immediately. Now we derive the sample complexity bound in (ii). By Proposition \ref{linear-rate-IMGM}, we have
    \begin{equation*}
        \mathbb{E}[\| x^{k+1}_{i(k)} - \widehat{x}^{\eta}_{i(k)}(x^{k}) \|^{2} \,\vert\, x^{k}] = \mathbb{E}[\| z^{k, j_{i(k), k}}_{i(k)} - \widehat{x}^{\eta}_{i(k)}(x^{k}) \|^{2} \,\vert\, x^{k}] \leq \Theta\widehat{p}^{j_{i(k), k}} \leq \tfrac{1}{K^{2}},
    \end{equation*}
    which implies that $j_{i(k), k} = \left\lceil \tfrac{\log{(K^{2})}}{\log{(1/\widehat{p})}} \right\rceil$ since $\Theta \geq 1$. Therefore at $k$th outer iteration, we bound the sample complexity as
    \begin{align*}
            \sum^{j_{i(k), k}-1}_{t=0} T^{k}_{i(k)}(t) &\leq \sum^{j_{i(k), k}-1}_{t=0} T^{k}_{i(k)}(0) \beta^{-t} \leq T \sum^{j_{i(k), k}-1}_{t=0} \beta^{-t} \leq \tfrac{\beta T}{1-\beta} \left( \tfrac{1}{\beta} \right)^{j_{i(k), k}} \\
            &\leq \tfrac{T}{1-\beta} \left( \tfrac{1}{\beta} \right)^{\tfrac{\log{(K^{2})}}{\log{(1/\widehat{p})}}} \overset{\text{Lemma } \ref{log-log-lemma}}{=} \tfrac{T}{1-\beta} \left( K^{2} \right)^{\tfrac{\log{(1/\beta)}}{\log{(1/\widehat{p})}}},
    \end{align*}
    where $\beta\in (0, 1)$. Let $j_{i, k} = 0$ for any $i\neq i(k)$. Therefore, after $K(\epsilon)$ iterations, the expected sample complexity for player $i$ to achieve $\mathbb{E}[\| G^{\eta}_{n}(x^{R_{K}}) \|]\leq \epsilon$ is bounded as
    \begin{align*}
        S_i(\epsilon) &= \mathbb{E} \left[ \sum_{k=0}^{K(\epsilon)-1} \sum_{t=0}^{j_{i,k}-1} T^{k}_{i}(t) \right] = p_{i} \mathbb{E} \left[ \sum_{k=0}^{K(\epsilon)-1} \sum_{t=0}^{j_{i,k}-1} T^{k}_{i}(t) \:\middle\vert\: i(k) = i \right] \\
        &\leq p_{i} \sum_{k=0}^{K(\epsilon)-1} \tfrac{T}{1-\beta} \left( K^{2} \right)^{\tfrac{\log{(1/\beta)}}{\log{(1/\widehat{p})}}} = \tfrac{p_{i}T}{1-\beta} K^{3+2\delta},
    \end{align*}
    since we define $\widehat{p} \triangleq \beta^{1/(1+\delta)}$ hence $\tfrac{\log{(1/\beta)}}{\log{(1/\widehat{p})}} = 1+\delta$. By plugging $K(\epsilon)$ defined in (i) and $p_{i} = 1/N$, we may derive the final result as follows.
    \begin{align*}
        S_i(\epsilon) \leq \tfrac{p_{i}T}{1-\beta} \left( \tfrac{9 N \bar{d}}{\eta\epsilon^{2}} \right)^{3+2\delta} = \tfrac{9^{3+2\delta}TN^{2+2\delta}\bar{d}^{3+2\delta}}{(1-\beta)\eta^{(3+2\delta)}\epsilon^{(6+4\delta)}}.
    \end{align*}
\end{proof}

\begin{remark}\em
    Prior asynchronous schemes~\cite{lei-shanbhag-2020} do not provide clean rate and complexity guarantees. We overcome this shortcoming by utilizing a residual function. Further, we qualify the dependence of $\eta$ to help in understanding the behavior of the schemes but it bears reminding that $\eta$ does not need to be driven to zero. $\hfill \Box$
\end{remark}

\section{Inexact BR schemes for stochastic weakly convex games}\label{Sec-4}

In this section, we focus on stochastic weakly convex games. In subsection \ref{Sec-4.1}, we derive the approximation relation between the smoothed QNE and the original QNE. The \textbf{MS-SSBR} and \textbf{MS-SABR} schemes, as well as their convergence properties will be presented in subsections \ref{Sec-4.2} and \ref{Sec-4.3}, respectively.

\subsection{Approximate QNE under Moreau envelope}\label{Sec-4.1}

The invariance of minimizers under the Moreau envelope for unconstrained convex and weakly convex problems is established in \cite[Lemma 2]{jalilzadeh-shanbhag-blanchet-glynn-2022} and \cite[Proposition 7.3]{renaud-leclaire-papadakis-2025}, respectively. However, no result exists for the constrained case, the challenge arising from the fact that the proximal point may fall outside the constrained set. Here we consider the stochastic $N$-player game $\mathcal{G}({\bf f}, X, \bxi)$, where each $f_{i}(\bullet, x_{-i})$ is $\rho_{i}$-weakly convex for any $x_{-i}$. Consider the Moreau-smoothed game $\mathcal{G}({\bf f}^{\eta}, X, \bxi)$ where for any $i \in [ N ]$ and given $x_{-i}$, the $i$th player solves
\begin{equation*}
    \min_{x_{i}\in X_{i}} f^{\eta}_{i}(x_{i}, x_{-i}), ~ \forall i\in [N],
\end{equation*}
where $f^{\eta}_{i}(\bullet, x_{-i})$ is the the Moreau envelope of $f_{i}(\bullet, x_{-i})$ with $\eta < \min_{i\in [N]} \rho_{i}^{-1}$. We will show that a QNE of $\mathcal{G}({\bf f}^{\eta}, X, \bxi)$ is an $\mathcal{O}(\eta)$-approximate QNE of $\mathcal{G}({\bf f}, X, \bxi)$. We first define an $\epsilon$-QNE as follows.

\begin{definition} [$\epsilon$-QNE] \em
    Consider the $N$-player game $\mathcal{G}({\bf f},X,\bxi)$, where for any $i \in [ N ]$, the $i$th player-specific function $f_{i}(\bullet, x_{-i})$ is weakly convex for given $x_{-i}$. Then $x^{\ast}$ is an $\epsilon$-QNE if for any $i\in [N]$, $f'_{i}(x^{\ast}_{i}, x^{\ast}_{-i}; x_{i}-x^{\ast}_{i})\geq -\epsilon$ holds for any $x_{i}\in X_{i}$. $\hfill \Box$
\end{definition}

Recall that a set-valued mapping $F: X \rightrightarrows Y$ is said to be $L$-Lipschitz on $X$ \cite[Definition 9.26]{rockafellar-wets-1998} if there exists a positive constant $L$ such that $F(x_{1}) \subseteq F(x_{2}) \,+\, L\|x_{1}-x_{2}\|\mathbb{B}_{Y}$ for any $x_{1}, x_{2}\in X$, where $\mathbb{B}_{Y}$ is the unit ball in $Y$.

\begin{theorem}[QNE approximation]\label{QNE-approx-thm}\em Consider the $N$-player stochastic weakly convex game $\mathcal{G}({\bf f},X,\bxi)$ and its Moreau-smoothed counterpart $\mathcal{G}({\bf f}^{\eta}, X, \bxi)$ with $\eta < \min_{i\in [N]} \rho_{i}^{-1}$. Suppose that (i) each $X_{i}$ is convex and compact; and (ii) for any $i\in [N]$, each subdifferential $\partial_{x_{i}} f_{i}(\bullet, x_{-i})$ is $L_{i}$-Lipschitz for any $x_{-i}$. Suppose that $x^{\ast}$ is a QNE of $\mathcal{G}({\bf f}^{\eta}, X, \bxi)$. If the upper bound $\| \nabla_{x_{i}} f_{i}^{\eta}(x^{\ast}_{i}, x^{\ast}_{-i}) \| \leq M^{\ast}$ holds uniformly for any $i\in [N]$ and some $M^{\ast} > 0$, $x^{\ast}$ is an $\eta L D M^{\ast}$-QNE of $\mathcal{G}({\bf f}, X, \bxi)$, where $L \triangleq \max_{i\in [N]}L_{i}$ and $D \triangleq \max_{i\in [N]} D_{X_{i}}$.
\end{theorem}
\begin{proof}
    For any $i\in [N]$, we know from the QNE definition that
    \begin{equation*}
        \nabla_{x_{i}} f^{\eta}_{i}(x^{*}_{i}, x^{*}_{-i})^{\top} (x_{i}-x^{*}_{i}) \geq 0, ~ \forall x_{i}\in X_{i}.
    \end{equation*}
    By the $\epsilon$-QNE definition, it suffices to show that for any $i\in [N]$, we have that
    \begin{align*}
        f'_{i}(x^{*}_{i}, x^{*}_{-i}; x_{i}-x^{*}_{i}) &  \geq -\eta L D M^{\ast}, ~ \forall x_{i}\in X_{i} \\
        \, \equiv \, \max_{u_{i}\in \partial_{x_{i}} f_{i}(x^{*}_{i}, x^{*}_{-i})} & \left\{ \langle u_{i}, x_{i} - x^{*}_{i} \rangle \right\} \geq -\eta L D M^{\ast}, ~ \forall x_{i}\in X_{i}.
    \end{align*}
    by Proposition \ref{ws-dd-subg}. We define the proximal point $\widehat{x}^{\ast}_{i} \triangleq \mathrm{prox}_{\eta f_{i}(\bullet, x^{*}_{-i})}(x^{*}_{i})$. By \eqref{Moreau-property} and assumptions, it follows that
    \begin{align*}
        \nabla_{x_{i}} f^{\eta}_{i}(x^{*}_{i}, x^{*}_{-i}) &\in \partial_{x_{i}} f_{i}(\widehat{x}^{*}_{i}, x^{*}_{-i}) \subseteq \partial_{x_{i}} f_{i}(x^{*}_{i}, x^{*}_{-i}) + L_{i} \| \widehat{x}^{\ast}_{i} - x^{\ast}_{i} \| \mathbb{B}^{n_{i}} \\
        &= \partial_{x_{i}} f_{i}(x^{*}_{i}, x^{*}_{-i}) + \eta L_{i} \| \nabla_{x_{i}} f_{i}^{\eta}(x^{*}_{i}, x^{*}_{-i}) \| \mathbb{B}^{n_{i}} \\
        &\subseteq \partial_{x_{i}} f_{i}(x^{*}_{i}, x^{*}_{-i}) + \eta L_{i} M^{\ast} \mathbb{B}^{n_{i}} \subseteq \partial_{x_{i}} f_{i}(x^{*}_{i}, x^{*}_{-i}) + \eta L M^{\ast} \mathbb{B}^{n_{i}}.
    \end{align*}
    Therefore, there exists $v\in \mathbb{B}^{n_{i}}$ such that $\| v \| \leq 1$ and
        $\nabla_{x_{i}} f^{\eta}_{i}(x^{*}_{i}, x^{*}_{-i}) - \eta L M^{\ast} v \in \partial_{x_{i}} f_{i}(x^{*}_{i}, x^{*}_{-i}).$
    Therefore we can derive that for any $x_{i}\in X_{i}$, we have that
    \begin{align*}
        \max_{u_{i}\in \partial_{x_{i}} f_{i}(x^{*}_{i}, x^{*}_{-i})} & \left\{ \langle u_{i}, x_{i}-x^{*}_{i} \rangle \right\} \geq \langle \nabla_{x_{i}} f^{\eta}_{i}(x^{*}_{i}, x^{*}_{-i}) - \eta L M^{\ast} v, x_{i}-x^{*}_{i} \rangle \\
        &= \langle \nabla_{x_{i}} f^{\eta}_{i}(x^{*}_{i}, x^{*}_{-i}), x_{i}-x^{*}_{i} \rangle - \langle \eta L M^{\ast} v, x_{i}-x^{*}_{i} \rangle \geq - \eta L D M^{\ast},
    \end{align*}
    since we have $\langle \eta L M v, x_{i}-x^{*}_{i} \rangle \leq \eta L D M^{\ast}$ by Cauchy–Schwarz inequality. Therefore, $x^{*}$ is an $(\eta L D M^{\ast})$-QNE of the weakly convex game $\mathcal{G}({\bf f},X,\bxi)$.
\end{proof}

\begin{remark}\label{ME-boundedness}\em
    One may ask when $\|\nabla_{x_{i}} f_{i}^{\eta}(x^{\ast}_{i}, x^{\ast}_{-i}) \| \leq M^{\ast}$ is valid. Consider two special cases. (i) If $x^{\ast}\in \mathrm{int}(X)$, we have that $M^{\ast} = 0$, implying that $x^{\ast}$ is an exact QNE of $\mathcal{G}({\bf f}, X, \bxi)$. (ii) If each $f_{i}(\bullet, x_{-i})$ is additionally $L^{0}_{i}$-Lipschitz continuous for any $x_{-i}$, we know from \cite[Lemma 3.3]{bohm-wright-2021} that $\| \nabla_{x_{i}} f_{i}^{\eta}(x^{\ast}_{i}, x^{\ast}_{-i}) \| \leq L_0 \triangleq \max_{i\in [N]} L^{0}_{i}$. Then $M^{\ast} = L_0$ and $x^{\ast}$ is an $\eta L D L_0$-QNE of weakly convex game $\mathcal{G}({\bf f}, X, \bxi)$. $\hfill \Box$
\end{remark}

\subsection{MS-SSBR under weak convexity}\label{Sec-4.2}
Next we consider a Moreau-smoothed surrogation-based synchronous BR scheme for stochastic weakly convex games, where we employ a quadratic surrogate of the Moreau envelope $f_i^{\eta}(\bullet, y_{-i})$ near $y_{i}$, denoted by $\widehat{f}_i^{\eta}(\bullet, y_{-i}; y_{i})$ and defined as
\begin{equation}\label{MS-RS-surrogate}
    \widehat{f}_i^{\eta}(\bullet, y_{-i}; y_i) \triangleq f_i^{\eta}(y_i, y_{-i}) + \nabla_{x_i}f_i^{\eta}(y_i, y_{-i})^{\top}(\bullet-y_i) + \tfrac{\mu}{2}\| \bullet - y_i \|^2,
\end{equation}
where $\mu > 0$ is some positive constant. The quadratic surrogate satisfies two matching properties near $y_{i}$: \\
(i) (function-value matching) $\widehat{f}_i^{\eta}(y_i, y_{-i}; y_i) = f_i^{\eta}(y_i, y_{-i})$; \\ 
(ii) (gradient matching) $\nabla_{x_i}\widehat{f}_i^{\eta}(y_i, y_{-i}; y_i) = \nabla_{x_i} f_i^{\eta}(y_i, y_{-i})$.

\:

We present our \textbf{MS-SSBR} scheme in Algorithm \ref{MS-SSBR}. Now the best-response solution $\widehat{x}^{\eta}_{i}(x^{k})$ at $k$th iteration is defined via a surrogated BR problem \eqref{MS-SSBR-eqn2}.

\begin{algorithm}[htbp]\caption{Moreau-smoothed Surrogated Synchronous BR (\textbf{MS-SSBR})}\label{MS-SSBR}
{\it Initialize:} Initialize $k=0$, given $\eta, K, \left\{\varepsilon_{i}^{k} \right\}_{i, k}$, and $x^{0} = (x_{i}^{0})_{i=1}^{N}\in X$.

{\it Iterate until $k\ge K$:} Compute $x^{k+1} = (x_{i}^{k+1})_{i=1}^{N}\in X$ by solving the subproblems
\begin{equation}\label{MS-SSBR-eqn1}
    x_{i}^{k+1}\in \{ z\in X_{i}: \mathbb{E}[\|z-\widehat{x}^{\eta}_{i}(x^{k})\|^2 \,\vert\, x^{k}] \leq (\varepsilon_{i}^{k})^2 ~ \textrm{a.s.} \},~\forall i\in [N],
\end{equation}
where $\varepsilon_{i}^{k} > 0$ is the inexactness employed by player $i$ at iteration $k$ and 
\begin{equation}\label{MS-SSBR-eqn2}
    \widehat{x}^{\eta}_{i}(x^{k}) \triangleq \argmin\limits_{z_{i}\in X_{i}}\: \widehat{f}^{\eta}_{i}(z_{i}, x^{k}_{-i}; x^{k}_{i}),~\forall i\, \in \, [ N ].
\end{equation}

{\it Return:} Return $x^{K}$ as the final estimate.
\end{algorithm}

\subsubsection{Convergence analysis}

We impose the following assumption to establish the almost sure convergence of \textbf{MS-SSBR}. For brevity, we denote $\widehat{f}^{\eta}_{i}(\bullet, y) \triangleq \widehat{f}^{\eta}_{i}(\bullet, y_{-i}; y_{i})$ for given $y\in X$ and any $i\in [N]$.

\vspace{5pt}

\emph{Assumption $\mathrm{C}$.} (C1) For any $i\in [N]$, for given $x_{i}\in X_{i}$, there exists some $\widehat{L}_{i}, \widehat{L}_{-i} > 0$ dependent on $\mu$ employed in \eqref{MS-RS-surrogate} such that  for any $y, w\in X$, we have $\| \nabla_{x_{i}} \widehat{f}^{\eta}_{i}(x_{i}, y) - \nabla_{x_{i}} \widehat{f}^{\eta}_{i}(x_{i}, w) \| \leq \widehat{L}_{i} \| y_{i} - w_{i} \| + \widehat{L}_{-i} \| y_{-i} - w_{-i} \|.$ (C2) (Contractive property) The spectral norm of matrix $\Gamma_{2}$ is strictly less than $1$, i.e., $\| \Gamma_{2} \|<1$, where $\Gamma_{2}$ is defined by
\begin{equation}\label{contraction-matrix}
    \Gamma_{2} \triangleq \begin{bmatrix}
        \tfrac{\widehat{L}_{1}}{\mu} & \tfrac{\widehat{L}_{-1}}{\mu} & \dots & \tfrac{\widehat{L}_{-1}}{\mu} \\
        \tfrac{\widehat{L}_{-2}}{\mu} & \tfrac{\widehat{L}_{2}}{\mu} & \dots & \tfrac{\widehat{L}_{-2}}{\mu} \\
        \vdots & \vdots & \ddots & \vdots \\
        \tfrac{\widehat{L}_{-N}}{\mu} & \tfrac{\widehat{L}_{-N}}{\mu} & \dots & \tfrac{\widehat{L}_{N}}{\mu}
    \end{bmatrix}.
\end{equation}

\begin{remark}\em
    One may ask whether we can choose sufficiently large $\mu > 0$ to satisfy assumption $\mathrm{(C2)}$. In fact, due to the implicit dependence of each $\widehat{L}_i$ on $\mu$, increasing $\mu$ does not necessarily ensure that $\mathrm{(C2)}$ will hold. $\hfill \Box$
\end{remark}

\begin{theorem}[Almost sure convergence of MS-SSBR]\label{MS-SSBR-thm}\em
    Consider the stochastic $N$-player weakly convex game $\mathcal{G}({\bf f}, X, \bxi)$ and its Moreau-smoothed game $\mathcal{G}({\bf f}^{\eta}, X, \bxi)$ with $\eta < \min_{i\in [N]} \rho_{i}^{-1}$. Let $\{x^{k}\}_{k=0}^{\infty}$ be generated by \textbf{MS-SSBR}. Suppose Assumption $\mathrm{C}$ holds, and we have that $\varepsilon^{k}_{i} \geq 0$ with $\sum_{k=0}^{\infty}\varepsilon^{k}_{i} < \infty$ for any $i\in [N]$. Then the following hold.\\
    (i) The mapping $\widehat{x}^{\eta}(\bullet) = (\widehat{x}^{\eta}_{i}(\bullet))_{i=1}^{N}$ is contractive with a unique fixed point. \\
    (ii) If $x^{\ast}$ is the unique fixed point of $\widehat{x}^{\eta}(\bullet)$, then $x^{\ast}$ is the unique QNE of $\mathcal{G}({\bf f}^{\eta}, X, \bxi)$ over $X$. Therefore, $x^{\ast}$ is an $\mathcal{O}(\eta)$-QNE of $\mathcal{G}({\bf f}, X, \bxi)$. \\
    (iii) The sequence $\{x^{k}\}_{k=0}^{\infty}$ converges to $x^{\ast}$ a.s.
\end{theorem}
\begin{proof}
    The proof of (i) is similar to that of Theorem \ref{as-convergence-MS-SBR}-(i). Indeed, for any $i\in [N]$ and any $y, w\in X$, we may arrive that
    \begin{align*}
        0&\leq (\widehat{x}^{\eta}_{i}(y) - \widehat{x}^{\eta}_{i}(w))^{\top}(\nabla_{x_{i}}\widehat{f}^{\eta}_{i}(\widehat{x}^{\eta}_{i}(w), w) - \nabla_{x_{i}}\widehat{f}^{\eta}_{i}(\widehat{x}^{\eta}_{i}(y), y)) \\
        &= (\widehat{x}^{\eta}_{i}(y) - \widehat{x}^{\eta}_{i}(w))^{\top}(\nabla_{x_{i}}\widehat{f}^{\eta}_{i}(\widehat{x}^{\eta}_{i}(w), w) - \nabla_{x_{i}}\widehat{f}^{\eta}_{i}(\widehat{x}^{\eta}_{i}(y), w)) \\
        &\quad + (\widehat{x}^{\eta}_{i}(y) - \widehat{x}^{\eta}_{i}(w))^{\top}(\nabla_{x_{i}}\widehat{f}^{\eta}_{i}(\widehat{x}^{\eta}_{i}(y), w) - \nabla_{x_{i}}\widehat{f}^{\eta}_{i}(\widehat{x}^{\eta}_{i}(y), y)) \\
        &\leq -\mu \|\widehat{x}^{\eta}_{i}(y) - \widehat{x}^{\eta}_{i}(w)\|^{2} + \|\widehat{x}^{\eta}_{i}(y) - \widehat{x}^{\eta}_{i}(w)\| \|\nabla_{x_{i}}\widehat{f}^{\eta}_{i}(\widehat{x}^{\eta}_{i}(y), w) - \nabla_{x_{i}}\widehat{f}^{\eta}_{i}(\widehat{x}^{\eta}_{i}(y), y))\|.
        \end{align*}
    By the above inequality and assumption $(\mathrm{C1})$, it follows that
    \begin{align*}
        &\mu \|\widehat{x}^{\eta}_{i}(y) - \widehat{x}^{\eta}_{i}(w)\| \leq \|\nabla_{x_{i}}\widehat{f}^{\eta}_{i}(\widehat{x}^{\eta}_{i}(y), w) - \nabla_{x_{i}}\widehat{f}^{\eta}_{i}(\widehat{x}^{\eta}_{i}(y), y))\| \\
        &\leq \widehat{L}_{i} \| y_{i} - w_{i} \| + \widehat{L}_{-i} \| y_{-i} - w_{-i} \| \leq \widehat{L}_{i} \| y_{i} - w_{i} \| + \widehat{L}_{-i} \sum_{j\neq i} \| y_{j} - w_{j} \|,
    \end{align*}
    implying that
    \begin{align*}
        \begin{bmatrix}
            \| \widehat{x}^{\eta}_{1}(y)-\widehat{x}^{\eta}_{1}(w) \| \\
            \vdots \\
            \| \widehat{x}^{\eta}_{N}(y)-\widehat{x}^{\eta}_{N}(w) \|
        \end{bmatrix} \leq \Gamma_{2}
        \begin{bmatrix}
            \| y_{1}-w_{1} \| \\
            \vdots \\
            \| y_{N}-w_{N} \|
        \end{bmatrix}.
    \end{align*}
    By assumption $\mathrm{(C2)}$, we have proved that $\widehat{x}^{\eta}(\bullet) = (\widehat{x}^{\eta}_{i}(\bullet))_{i=1}^{N}$ is contractive.  (ii) follows directly from the first-order optimality condition for constrained strongly convex smooth optimization, combined with the gradient matching property and Theorem \ref{QNE-approx-thm}. The proof of (iii) is the same as that of \cite[Proposition 1]{lei-shanbhag-pang-sen-2020} and is omitted here.
\end{proof}

Similar to Theorem \ref{linear-rate-MS-SBR}, we may achieve the linear rate of \textbf{MS-SSBR}.

\begin{theorem}[Linear rate of MS-SSBR]\label{linear-rate-MS-SSBR}
    \em Consider \textbf{MS-SSBR} where for any $i\in [N]$, we have that $\mathbb{E}[\|x_{i}^{0} - x^{\ast}_{i}\|] \leq C_{2}$ for some $C_{2} > 0$ and $\varepsilon_{i}^{k} \triangleq \nu^{k+1}$ for some $\nu \in (0,1)$. Suppose that Assumption $\mathrm{C}$ holds. Define $c_{2} \triangleq \max\{ \| \Gamma_{2} \|, \nu \}$. Consider the expected error $e_{k}$ defined in \eqref{expected-error}. Then the following holds for any $q_{2}\in (c_{2}, 1)$, $D_{2} \triangleq 1/\log{((q_{2}/c_{2})^{e})}$, and any $k \ge 0$.  $e_{k}\leq \sqrt{N}(C_{2} + D_{2})q_{2}^{k}.$ $\hfill \Box$
\end{theorem}

\begin{remark}\em
    In contrast to \cite{cui-pang-2021, pang-razaviyayn-2016} where a limiting consistency condition is required, Theorem \ref{linear-rate-MS-SSBR} establishes a linear convergence in the weakly convex regime without imposing such a limiting condition. However, assumption $\mathrm{(C2)}$ implicitly ensures that the weakly convex game exhibits some locally convexity-like behavior over $X$. When $N = 1$, such a behavior aligns with the result in \cite[Theorem 4.3]{liao-ding-zheng-2023}, where linear convergence of weakly convex optimization is obtained under a quadratic growth condition and appropriately chosen initialization. $\hfill \Box$
\end{remark}

\subsubsection{Complexity analysis}\label{Sec-4.2.2}

In this subsection, we consider the efficient resolution of the Moreau-smoothed surrogated BR problem. Consider the $i$th player's problem at epoch $k$, defined as \eqref{MS-SSBR-eqn2}. Before proceeding, we show that minimizing a suitably defined (possibly expectation-valued) strongly convex quadratic function with the same strong convexity and Lipschitz smoothness constants, over a closed convex set reduces to the projection of an expectation-valued vector.

\begin{lemma} \em 
    Consider the minimization of an $\alpha$-strongly convex and $\alpha$-Lipschitz smooth function $f(y)\triangleq \tfrac{1}{2} \alpha y^{\top}y + \mathbb{E}[\tilde b(\bxi)]^{\top}y + c$, over a closed convex set $Y\subseteq \mathbb{R}^{n}$. Then the following hold. (i) The unique minimizer of $h$ on $Y$ is given by $y^\ast = \Pi_{Y}\left[-\tfrac{b}{\alpha}\right]$, where $b \triangleq \mathbb{E}[\Tilde{b}(\bxi)]$. (ii) Given any $y_0 \in \mathbb{R}^n$, we have $y^\ast = \Pi_{Y}\left[y_0 - \tfrac{1}{\alpha} \nabla f(y_0)\right]$. 
\end{lemma}
\begin{proof} (i) The unique minimizer $y^\ast$ is the unique fixed point of the gradient mapping \cite[pp. 177]{beck-2017}. Therefore, we have that
    \begin{align*} y^\ast = \Pi_Y \left[ y^\ast - \tfrac{1}{\alpha}\nabla f(y^\ast) \right] =  
    \Pi_Y \left[ y^\ast - (y^\ast + \tfrac{b}{\alpha}) \right] = \Pi_Y \left[ - \tfrac{b}{\alpha} \right].\end{align*}
    (ii) Observe that $\Pi_{Y}\left[ y_0 - \tfrac{1}{\alpha} \nabla f(y_0) \right] = \Pi_{Y} \left[ -\tfrac{b}{\alpha} \right] = y^\ast.$
\end{proof}

It can be seen that obtaining $y^\ast$ requires estimating $b = \mathbb{E}[\tilde{b}(\bxi)]$ and a single projected gradient step from any $y^{0}\in \mathbb{R}^{n}$. Since $\widehat{f}^{\eta}_{i}$ is precisely a quadratic function with the same strong convexity and Lipschitz smoothness constants, we may see
\begin{equation}\label{MS-SSBR-BR-one-step-exact-update}
    \widehat{x}^{\eta}_{i}(x^{k}) = \Pi_{X_{i}} \left[ z_{i} - \tfrac{1}{\mu}\nabla_{x_{i}}\widehat{f}^{\eta}_{i}(z_{i}, x^{k}_{-i}; x^{k}_{i}) \right]
\end{equation}
holds for any $z_{i}\in X_{i}$. In short, the smoothed surrogated best-response problem reduces to an estimation problem of the underlying expectation. Inspired by the above result, we propose the following \emph{one-step} inexact Moreau-smoothed gradient method (O-IMGM) to solve this problem.
\begin{equation}\label{O-IMGM}
    \hspace*{-2em}
    z^{0}_{i} \leftarrow x^{k}_{i}, \: z^{1}_{i} \triangleq \Pi_{X_{i}} \left[ z^{0}_{i} - \tfrac{1}{\mu}(\nabla_{x_{i}}\widehat{f}^{\eta}_{i}(z^{0}_{i}, x^{k}_{-i}; x^{k}_{i}) + w^{k}_{i}) \right], \: x^{k+1}_{i} \leftarrow z^{1}_{i}. \tag{O-IMGM}
\end{equation}
By the gradient matching property, we can see that
\begin{equation*}
    z^{1}_{i} = \Pi_{X_{i}} \left[ z^{0}_{i} - \tfrac{1}{\mu}(\nabla_{x_{i}}f^{\eta}_{i}(x^{k}_{i}, x^{k}_{-i}) + \mu(z^{0}_{i} - x^{k}_{i}) + w^{k}_{i}) \right].
\end{equation*}
Akin to the discussion of \eqref{IMGM}, $\nabla_{x_{i}}f^{\eta}_{i}(x^{k}_{i}, x^{k}_{-i}) + w^{k}_{i} = \tfrac{1}{\eta}(x^{k}_{i} - \widehat{\mathrm{prox}}_{\eta f_{i}(\bullet, x^{k}_{-i})}(x^{k}_{i})),$ where $\widehat{\mathrm{prox}}_{\eta f_{i}(\bullet, x^{k}_{-i})}(x^{k}_{i})$ is an approximation of $\mathrm{prox}_{\eta f_{i}(\bullet, x^{k}_{-i})}(x^{k}_{i})$. It follows that
\begin{equation}\label{MS-SSBR-BR-one-step-error}
    w^{k}_{i} = \tfrac{1}{\eta}(\mathrm{prox}_{\eta f_{i}(\bullet, x^{k}_{-i})}(x^{k}_{i}) - \widehat{\mathrm{prox}}_{\eta f_{i}(\bullet, x^{k}_{-i})}(x^{k}_{i})).
\end{equation}
Therefore, in a practical implementation, we update $z^{1}_{i}$ as
\begin{equation*}
    z^{1}_{i} = \Pi_{X_{i}} \left[ z^{0}_{i} - \tfrac{1}{\mu} \left( \tfrac{1}{\eta}(x^{k}_{i} - \widehat{\mathrm{prox}}_{\eta f_{i}(\bullet, x^{k}_{-i})}(x^{k}_{i})) + \mu(z^{0}_{i} - x^{k}_{i}) \right) \right],
\end{equation*}
where we still employ stochastic subgradient method (SSM) for obtaining the inexact solution $\widehat{\mathrm{prox}}_{\eta f_{i}(\bullet, x^{k}_{-i})}(x^{k}_{i})$. Therefore, the second moment bound of $w^{k}_{i}$ largely depends on the iteration times $T^{k}_{i}$ of (SSM).

\begin{lemma}[Second moment error bound of O-IMGM]\em \label{second-moment-error-O-IMGM}
    Suppose that the subgradient oracle is uniformly bounded. Suppose that we employ stochastic subgradient method $\mathrm{(SSM)}$ with iteration times $T^{k}_{i}$ within the implementation of \eqref{O-IMGM} in the $k$th outer iteration. Then 
        $\mathbb{E} [ \| x^{k+1}_{i} - \widehat{x}^{\eta}_{i}(x^{k}) \|^{2} \,\vert\, x^{k} ] \leq \tfrac{Q^{\prime}}{\mu^{2}\eta^{2}T^{k}_{i}},$
    where $Q^{\prime} > 0$ is some constant that depends on the bound of the subgradient oracles.
\end{lemma}
\begin{proof}
    By invoking the similar sublinear rate result from \cite[Section 2.1]{nemirovski-juditsky-lan-shapiro-2009}, 
    \begin{align*}
            & \mathbb{E}[\| x^{k+1}_{i} - \widehat{x}^{\eta}_{i}(x^{k}) \|^{2} \,\vert\, x^{k}] \overset{\eqref{O-IMGM}}{=} \mathbb{E}[\| z^{1}_{i} - \widehat{x}^{\eta}_{i}(x^{k}) \|^{2} \,\vert\, x^{k}] \overset{\eqref{MS-SSBR-BR-one-step-exact-update}}{\leq} \tfrac{1}{\mu^{2}}\mathbb{E}[\| w^{k}_{i} \|^{2}] \\
            &\overset{\eqref{MS-SSBR-BR-one-step-error}}{=} \tfrac{1}{\mu^{2}\eta^{2}}\mathbb{E}[\| \widehat{\mathrm{prox}}_{\eta f_{i}(\bullet, x^{k}_{-i})}(x^{k}_{i}) - \mathrm{prox}_{\eta f_{i}(\bullet, x^{k}_{-i})}(x^{k}_{i}) \|^{2}] \leq \tfrac{Q^{\prime}}{\mu^{2}\eta^{2}T^{k}_{i}},
    \end{align*}
    where $Q^{\prime} > 0$ depends on the bound of the subgradient oracles.
\end{proof}

Based on the above lemma, we derive the complexity result for \textbf{MS-SSBR}.

\begin{theorem}[Complexities of MS-SSBR]\label{complexity-MS-SSBR}\em Consider \textbf{MS-SSBR} where the stochastic Moreau-smoothed BR problem \eqref{MS-SSBR-eqn2} is computed via \eqref{O-IMGM}. Suppose that Assumption $\mathrm{C}$ holds. Under the same parameter settings as in Theorem \ref{linear-rate-MS-SSBR} and Lemma \ref{second-moment-error-O-IMGM}, the following hold.\\
(i) (Iteration complexity) We have that $e_{k}\leq \epsilon$ where $e_{k}$ is defined in \eqref{expected-error} after most $K(\epsilon)$ iterations, where $K(\epsilon)$ is defined as
\begin{equation*}
    K(\epsilon) \triangleq \left\lceil \tfrac{\log(\sqrt{N}(C_{2}+D_{2})/\epsilon)}{\log{(1/q_{2})}} \right\rceil.
\end{equation*}
(ii) (Sample complexity) We choose $\nu \triangleq \| \Gamma_{2} \|$ and $\delta > 0$ such that $q_{2} \triangleq \nu^{1/(1+\delta)} < 1$. The overall sample complexity $S_{i}(\epsilon)$ for player $i$ to achieve $e_{k}\leq \epsilon$ can be bounded as
\begin{equation*}
    S_i(\epsilon)  \leq \tfrac{2Q'}{\mu^{2}\eta^{2}(1-\nu^{2})\nu^{2}} \left( \tfrac{\sqrt{N}(C_{2}+D_{2})}{\epsilon} \right)^{2(1+\delta)}.
\end{equation*}
\end{theorem}
\begin{proof}
    The proof of (i) is similar to that of Theorem \ref{complexity-MS-SBR}-(i). Now we prove (ii). We first consider the $k$th outer iteration and give an expression of $T^{k}_{i}$ in $\mathrm{(SSM)}$. Recall that we set $\varepsilon^{k}_{i} \triangleq \nu^{k+1}$ in Theorem \ref{linear-rate-MS-SSBR}. By Lemma \ref{second-moment-error-O-IMGM}, for player $i$, we have
    \begin{align*}
        \mathbb{E} \left[ \| x^{k+1}_{i} - \widehat{x}^{\eta}_{i}(x^{k}) \|^{2} \:\middle\vert\: x^{k} \right] \leq \tfrac{Q^{\prime}}{\mu^{2}\eta^{2}T^{k}_{i}} \leq (\nu^{k+1})^{2},
    \end{align*}
    which implies that $T^{k}_{i} = \left\lceil \tfrac{Q^{\prime}}{\mu^{2}\eta^{2}\nu^{2(k+1)}} \right\rceil \leq \tfrac{2Q^{\prime}}{\mu^{2}\eta^{2}\nu^{2(k+1)}}$. Therefore, during $K(\epsilon)$ iterations, $S_i(\epsilon)$ can be upper bounded by
    \begin{equation*}
        \begin{aligned}
            &\sum_{k=0}^{K(\epsilon)-1} T^{k}_{i} \leq \tfrac{2Q^{\prime}}{\mu^{2}\eta^{2}} \sum_{k=0}^{K(\epsilon)-1} \tfrac{1}{\nu^{2(k+1)}} \leq \tfrac{2Q^{\prime}}{\mu^{2}\eta^{2}(1-\nu^{2})\nu^{2}} \left( \tfrac{1}{\nu^{2}} \right)^{\tfrac{\log{(\sqrt{N}(C_{2}+D_{2})/\epsilon)}}{\log{(1/q_{2})}}} \\
            &\overset{\text{Lemma } \ref{log-log-lemma}}{=} \tfrac{2Q^{\prime}}{\mu^{2}\eta^{2}(1-\nu^{2})\nu^{2}} \left( \tfrac{\sqrt{N}(C_{2}+D_{2})}{\epsilon} \right)^{\tfrac{\log{(1/\nu^{2})}}{\log{(1/q_{2})}}}.
        \end{aligned}
    \end{equation*}
    Since we set $q_{2} \triangleq \nu^{1/(1+\delta)} < 1$, it results in $\tfrac{\log{(1/\nu^{2})}}{\log{(1/q_{2})}} = 2(1+\delta)$. Therefore $S_i(\epsilon)$ can be upper bounded by
    \begin{equation*}
        S_i(\epsilon) \leq \tfrac{2Q'}{\mu^{2}\eta^{2}(1-\nu^{2})\nu^{2}} \left( \tfrac{\sqrt{N}(C_{2}+D_{2})}{\epsilon} \right)^{2(1+\delta)},
    \end{equation*}
    as desired.
\end{proof}

\subsection{MS-SABR under weak convexity}\label{Sec-4.3}
 
In this subsection, we consider an asynchronous counterpart in which a randomly selected player can make her update.

\begin{algorithm}[htb]\caption{Moreau-smoothed Surrogated Asynchronous BR (\textbf{MS-SABR})}\label{MS-SABR}
{\it Initialize:} Initialize $k=0$, given $\eta, K, \{\varepsilon_i^k\}_{i,k}$, and $x^{0} = (x_{i}^{0})_{i=1}^{N}\in X$. Let $0 < p_{i} < 1$ be the probability of selecting player $i$ such that $\sum_{i=1}^{N}p_{i}=1$.

{\it Iterate until $k\ge K$:} 

\noindent(i) Pick a player $i(k) \in [N]$ with probability $p_{i(k)}$.

\noindent(ii) The player $i(k)$ updates her strategy $x_{i(k)}^{k+1}$ as follows:
\begin{equation}\label{MS-SABR-eqn1}
    x_{i(k)}^{k+1}\in \{ z\in X_{i(k)}: \mathbb{E}[\|z-\widehat{x}^{\eta}_{i(k)}(x^{k})\|^2 \,\vert\, x^{k}] \leq (\varepsilon_{i(k)}^{k})^2 ~ \textrm{a.s.} \},
\end{equation}
where
\begin{equation}\label{MS-SABR-eqn2}
    \widehat{x}^{\eta}_{i(k)}(x^{k}) \triangleq \argmin\limits_{z_{i}\in X_{i(k)}}\: \widehat{f}^{\eta}_{i(k)}(z_{i}, x^{k}_{-i(k)}; x^{k}_{i(k)}).
\end{equation}
Other players' strategies remain invariant, i.e., $x_{i}^{k+1} = x_{i}^{k}$ and $\varepsilon^{k}_{i} = 0$ for $i\neq i(k)$.

{\it Return:} Return $x^{K}$ as final estimate.
\end{algorithm}

We define the residual $G^{\eta}_{X, \gamma}(x) \triangleq (G^{\eta}_{X_{i}, \gamma}(x))_{i=1}^{N}$, where the $i$th component is defined as $G^{\eta}_{X_{i}, \gamma}(x) \triangleq \tfrac{1}{\gamma} \left( x_{i} - \Pi_{X_{i}}[x_{i} - \gamma \nabla_{x_{i}} f^{\eta}_{i}(x_{i}, x_{-i})] \right)$ for any $i\in [N]$ and some $\gamma > 0$.

\begin{lemma}\label{lm:bd_Fnat}\em
For any $k \geq 0$, we have that
\begin{equation*}
    \| G^{\eta}_{X_{i}, \gamma}(x^{k}) \| \leq \tfrac{1 + | 1 - \gamma \mu |}{\gamma} \| \widehat{x}^{\eta}_{i}(x^{k}) - x^{k}_{i} \|.
\end{equation*}
If $\gamma > 0$ is such that $\gamma\mu > 1$,  $\| G^{\eta}_{X_{i}, \gamma}(x^{k}) \| \leq \mu \| \widehat{x}^{\eta}_{i}(x^{k}) - x^{k}_{i} \|$.
\end{lemma}
\begin{proof}
\allowdisplaybreaks
For any $k \geq 0$, we have that
\begin{align}
    &\| G^{\eta}_{X_{i}, \gamma}(x^{k}) \| = \tfrac{1}{\gamma} \| x^{k}_{i} - \Pi_{X_{i}}[x^{k}_{i} - \gamma \nabla_{x_{i}} f^{\eta}_{i}(x^{k}_{i}, x^{k}_{-i})] \| \notag \\
    &= \tfrac{1}{\gamma} \| x^{k}_{i} - \Pi_{X_{i}}[x^{k}_{i} - \gamma \nabla_{x_{i}} f^{\eta}_{i}(x^{k}_{i}, x^{k}_{-i})] + \widehat{x}^{\eta}_{i}(x^{k}) - \widehat{x}^{\eta}_{i}(x^{k}) \| \notag \\
    &\leq \tfrac{1}{\gamma} \| x^{k}_{i} - \widehat{x}^{\eta}_{i}(x^{k}) \| + \underbrace{ \tfrac{1}{\gamma} \| \widehat{x}^{\eta}_{i}(x^{k}) - \Pi_{X_{i}}[x^{k}_{i} - \gamma \nabla_{x_{i}} f^{\eta}_{i}(x^{k}_{i}, x^{k}_{-i})] \|}_{(i)}. \label{lm:bd_Fnat-eqn1}
\end{align}
By the fixed point property and the gradient matching property, we have that
\begin{align*}
    \widehat{x}^{\eta}_{i}(x^{k}) &= \Pi_{X_{i}}[\widehat{x}^{\eta}_{i}(x^{k}) - \gamma \nabla_{x_{i}} \widehat{f}^{\eta}_{i}(\widehat{x}^{\eta}_{i}(x^{k}), x^{k}_{-i}; x^{k}_{i})] \\
    &= \Pi_{X_{i}} \Big[\widehat{x}^{\eta}_{i}(x^{k}) - \gamma \left[\nabla_{x_{i}} f^{\eta}_{i}(x^{k}_{i}, x^{k}_{-i}) + \mu (\widehat{x}^{\eta}_{i}(x^{k}) - x^{k}_{i}) \right] \Big].
\end{align*}
Therefore, by the nonexpansiveness of the projection, it follows that
\begin{equation}\label{lm:bd_Fnat-eqn2}
    (i) \leq \tfrac{1}{\gamma} \| (\widehat{x}^{\eta}_{i}(x^{k})-x^{k}_{i}) - \gamma\mu(\widehat{x}^{\eta}_{i}(x^{k})-x^{k}_{i}) \| = \tfrac{|1-\gamma\mu|}{\gamma} \| \widehat{x}^{\eta}_{i}(x^{k}) - x^{k}_{i} \|.
\end{equation}
By combining \eqref{lm:bd_Fnat-eqn1} with \eqref{lm:bd_Fnat-eqn2}, we have that
    $\| G^{\eta}_{X_{i}, \gamma}(x^{k}) \| \leq \tfrac{1 + \left| 1 - \gamma \mu \right|}{\gamma} \| \widehat{x}^{\eta}_{i}(x^{k}) - x^{k}_{i} \|.$
If we choose $\gamma > 0$ such that $\gamma\mu > 1$, we have $\| G^{\eta}_{X_{i}, \gamma}(x^{k}) \| \leq \mu \| \widehat{x}^{\eta}_{i}(x^{k}) - x^{k}_{i} \|$.
\end{proof}

We consider the following assumption in this subsection.

\vspace{5pt}

\emph{Assumption $\mathrm{D}$.} (D1) $X_{i}$ is compact for any $i \in [ N ]$. (D2) (Potentiality) Given $\eta > 0$, there exists a potential function $P: X\to \mathbb{R}$ such that $f^{\eta}_{i} (x_{i}, x_{-i}) - f^{\eta}_{i} (y_{i}, x_{-i}) = P (x_{i}, x_{-i}) - P (y_{i}, x_{-i}), ~ \forall i\in [N]$ holds for any $x_{i}, y_{i}\in X_{i}$ and any $x_{-i} \in X_{-i}$. (D3) There exists some constant $M > 0$ such that $\|\nabla_{x_{i}}f^{\eta}_{i}(x_{i}, x_{-i})\|\leq M$ holds uniformly for any $i\in [N]$ and any $x\in X$.

\vspace{5pt}

\begin{remark}\em
    Same as Remark \ref{ME-boundedness}, when each $f_{i}(\bullet, x_{-i})$ is $L^{0}_{i}$-Lipschitz continuous for given $x_{-i}$, we may set $M = \max_{i\in [N]} L^{0}_{i}$ in assumption $\mathrm{(D3)}$. $\hfill \Box$
\end{remark}

\subsubsection{Convergence analysis}\label{Sec-4.3.1}

The proof of the following theorem is similar to that of Theorem \ref{MS-ABR-key-recursion}.

\begin{theorem}[Almost sure convergence of MS-SABR]\label{MS-SABR-main-proposition}\em
     Consider the stochastic $N$-player game $\mathcal{G}({\bf f}, X, \bxi)$ where each $f_{i}(\bullet, x_{-i})$ is $\rho_{i}$-weakly convex for given $x_{-i}$. Suppose that Assumption $\mathrm{D}$ holds. Let $\{x^{k}\}_{k\geq 0}$ be the sequence generated by \textbf{MS-SABR}. We choose $\eta > 0$ such that $\eta \rho_{i} \leq \tfrac{1}{2}$ for any $i\in [N]$. Suppose that $\mu > \tfrac{1}{2\eta}$ holds. Then for any $k\geq 0$, we have that
     \begin{equation}\label{MS-SABR-main-proposition-result}
         \begin{aligned}
             \mathbb{E}[P(x^{k+1}) \,\vert\, x^{k}] &\leq P(x^{k}) - \left( \mu - \tfrac{1}{2\eta} \right) \sum_{i=1}^{N} p_{i} \left\| \widehat{x}^{\eta}_{i}(x^{k}) - x^{k}_{i} \right\|^{2} + M \varepsilon^{k}_{\max},
         \end{aligned}
     \end{equation}
     where $\varepsilon^{k}_{\max} \triangleq \max\limits_{i\in [N]} \varepsilon^{k}_{i}$. If each $\nabla_{x_{i}} f^{\eta}(\bullet)$ is C$^{1}$ for any $i\in [N]$ and we have that $\sum_{k=0}^{\infty} \varepsilon^{k}_{\max} < \infty$, every limit point of $\{x^{k}\}_{k=0}^{\infty}$ is an $\mathcal{O}(\eta)$-QNE almost surely.
\end{theorem}
\begin{proof}
    By Proposition \ref{Lipschitz-wcvx} and the fact that $\eta \rho_{i}\leq \tfrac{1}{2}$ for any $i\in [N]$, we know each $f^{\eta}_{i}(\bullet, x_{-i})$ is $\tfrac{1}{\eta}$-Lipschitz smooth. By descent lemma \cite[Lemma 5.7]{beck-2017}, we have
    \begin{equation*}
        \begin{aligned}
            &f^{\eta}_{i(k)}(\widehat{x}^{\eta}_{i(k)}(x^{k}), x^{k}_{-i(k)}) \leq f^{\eta}_{i(k)}(x^{k}_{i(k)}, x^{k}_{-i(k)})  \\
            &\quad + \nabla_{x_{i(k)}} f^{\eta}_{i(k)}(x^{k}_{i(k)}, x^{k}_{-i(k)})^{\top}(\widehat{x}^{\eta}_{i(k)}(x^{k}) - x^{k}_{i(k)}) + \tfrac{1}{2\eta}\|\widehat{x}^{\eta}_{i(k)}(x^{k}) - x^{k}_{i(k)}\|^{2}.
        \end{aligned}
    \end{equation*}
    By the optimality condition of the surrogated BR problem \eqref{MS-SABR-eqn2} and the gradient matching property, it follows that
    \begin{equation*}
        \begin{aligned}
            0 &\leq \nabla_{x_{i(k)}} \widehat{f}^{\eta}_{i(k)}(\widehat{x}^{\eta}_{i(k)}(x^{k}), x^{k}_{-i(k)}; x^{k}_{i(k)})^{\top} (x^{k}_{i(k)}-\widehat{x}^{\eta}_{i(k)}(x^{k})) \\
            &= \nabla_{x_{i(k)}} f^{\eta}_{i(k)}(x^{k}_{i(k)}, x^{k}_{-i(k)})^{\top} (x^{k}_{i(k)}-\widehat{x}^{\eta}_{i(k)}(x^{k})) - \mu \| x^{k}_{i(k)}-\widehat{x}^{\eta}_{i(k)}(x^{k}) \|^{2}.
        \end{aligned}
    \end{equation*}
    By adding the above two inequalities, it leads to
    \begin{equation}\label{MS-SABR-main-proposition-eqn1}
        f^{\eta}_{i(k)}(\widehat{x}^{\eta}_{i(k)}(x^{k}), x^{k}_{-i(k)}) \leq f^{\eta}_{i(k)}(x^{k}_{i(k)}, x^{k}_{-i(k)}) - \left( \mu - \tfrac{1}{2\eta} \right) \| \widehat{x}^{\eta}_{i(k)}(x^{k}) - x^{k}_{i(k)} \|^{2}.
    \end{equation}
    By the potentiality assumption $\mathrm{(D2)}$, we derive that
    \begin{align*}
        & P(x^{k+1}) - P(x^{k}) = P(x^{k+1}_{i(k)}, x^{k}_{-i(k)}) - P(x^{k}_{i(k)}, x^{k}_{-i(k)}) \\
        &= f^{\eta}_{i(k)}(x^{k+1}_{i(k)}, x^{k}_{-i(k)}) - f^{\eta}_{i(k)}(x^{k}_{i(k)}, x^{k}_{-i(k)}) \\
        &= \underbrace{f^{\eta}_{i(k)}(x^{k+1}_{i(k)}, x^{k}_{-i(k)}) - f^{\eta}_{i(k)}(\widehat{x}^{\eta}_{i(k)}(x^{k}), x^{k}_{-i(k)})}_{(i)} \\
        & \quad + \underbrace{f^{\eta}_{i(k)}(\widehat{x}^{\eta}_{i(k)}(x^{k}), x^{k}_{-i(k)}) - f^{\eta}_{i(k)}(x^{k}_{i(k)}, x^{k}_{-i(k)})}_{(ii)}.
    \end{align*}
    By applying the mean value theorem on (i) and invoking \eqref{MS-SABR-main-proposition-eqn1} to bound (ii), we know that there exists $z^{k+1}_{i(k)} = \lambda x^{k+1}_{i(k)} + (1-\lambda) \widehat{x}^{\eta}_{i(k)}(x^{k})$ for some $\lambda\in (0, 1)$ such that
    \begin{equation*}
        \begin{aligned}
            P(x^{k+1}) - P(x^{k}) &\leq \| \nabla_{x_{i(k)}} f^{\eta}_{i(k)}(z^{k+1}_{i(k)}, x^{k}_{-i(k)}) \| \| x^{k+1}_{i(k)} - \widehat{x}^{\eta}_{i(k)}(x^{k}) \| \\
            &- \left( \mu - \tfrac{1}{2\eta} \right) \| \widehat{x}^{\eta}_{i(k)}(x^{k}) - x^{k}_{i(k)} \|^{2}.
        \end{aligned}
    \end{equation*}
    By the boundedness assumption $\mathrm{(D3)}$, we have that
    \begin{equation*}
        P(x^{k+1}) \!\leq\! P(x^{k}) + M \left\| x^{k+1}_{i(k)} - \widehat{x}^{\eta}_{i(k)}(x^{k}) \right\| - \left( \! \mu - \tfrac{1}{2\eta} \! \right) \left\| \widehat{x}^{\eta}_{i(k)}(x^{k}) - x^{k}_{i(k)} \right\|^{2}.
    \end{equation*}
    By taking the conditional expectation $\mathbb{E}[\bullet \,\vert\, x^{k}]$ and invoking the conditional Jensen's inequality, it follows that
    \begin{equation*}
        \begin{aligned}
            \mathbb{E}[P(x^{k+1}) \,\vert\, x^{k}] \leq P(x^{k}) - \left( \mu - \tfrac{1}{2\eta} \right) \mathbb{E} \left[ \left\| \widehat{x}^{\eta}_{i(k)}(x^{k}) - x^{k}_{i(k)} \right\|^{2} \:\middle\vert\: x^{k} \right] + M \varepsilon^{k}_{i(k)}.
        \end{aligned}
    \end{equation*}
    Note that $\mathbb{P}[i(k) = i] = p_{i}$ and $\varepsilon^{k}_{\max} = \max\limits_{i\in [N]} \varepsilon^{k}_{i}$, we may derive that
    \begin{equation*}
         \begin{aligned}
             \mathbb{E}[P(x^{k+1}) \,\vert\, x^{k}] \leq P(x^{k}) - \left( \mu - \tfrac{1}{2\eta} \right) \sum_{i=1}^{N} p_{i} \left\| \widehat{x}^{\eta}_{i}(x^{k}) - x^{k}_{i} \right\|^{2} + M \varepsilon^{k}_{\max},
         \end{aligned}
     \end{equation*}
     which completes the proof. If we additionally have that $\sum_{k=0}^{\infty} \varepsilon^{k}_{\max} < \infty$, it leads to
     \begin{equation*}
         \sum_{k=0}^{\infty}\sum_{i=1}^{N}p_{i}\| \widehat{x}^{\eta}_{i}(x^{k}) - x_{i}^{k} \|^2 < \infty \text{ a.s. }
     \end{equation*}
     by invoking \cite[Theorem 1]{robbins-siegmund-1971}. It follows that for any $i\in [N]$ we have $\sum\limits_{k=0}^{\infty} \| \widehat{x}^{\eta}_{i}(x^{k}) - x_{i}^{k} \|^2 < \infty$ a.s. hence
     \begin{equation}\label{MS-SABR-main-proposition-eqn2}
         \lim\limits_{k\to \infty} \| \widehat{x}^{\eta}_{i}(x^{k}) - x_{i}^{k} \| = 0 ~ \text{ a.s.}
     \end{equation}
     Let $\bar{x}$ be a cluster point of sequence $\{x^{k}\}_{k=0}^{\infty}$ due to assumption $\mathrm{(D1)}$. Then, there exists a subsequence $\mathcal{K}$ such that
    \begin{equation}\label{MS-SABR-main-proposition-eqn3}
        \lim\limits_{k\to \infty, k\in \mathcal{K}} x^{k} = \bar{x}.
    \end{equation}
    By \eqref{MS-SABR-main-proposition-eqn2} and \eqref{MS-SABR-main-proposition-eqn3}, it follows that for any $i\in [N]$ we have
    \begin{equation}\label{MS-SABR-main-proposition-eqn4}
        \lim\limits_{k\to \infty, k\in \mathcal{K}} \widehat{x}^{\eta}_{i}(x^{k}) = \bar{x}_{i} ~ \text{ a.s.}
    \end{equation}
    We intend to show that the cluster point $\bar{x}$ is a QNE of $\mathcal{G}({\bf f}^{\eta}, X, \bxi)$. We assume by contradiction that there exists $\bar{y}_{i}\in X_{i}$ such that
    \begin{equation}\label{MS-SABR-main-proposition-eqn5}
        0 > (\bar{y}_{i}-\bar{x}_{i})^{\top} \nabla_{x_{i}} f^{\eta}_{i}(\bar{x}_{i}, \bar{x}_{-i}).
    \end{equation}
    Since $\bar{y}_{i}\in X_{i}$, by the first-order optimality condition, we may obtain
    \begin{equation}\label{MS-SABR-main-proposition-eqn6}
        \begin{aligned}
            0 &\leq \nabla_{x_{i}} \widehat{f}^{\eta}_{i}(\widehat{x}^{\eta}_{i}(x^{k}), x^{k}_{-i}; x^{k}_{i})^{\top}(\bar{y}_{i} - \widehat{x}^{\eta}_{i}(x^{k})) \\
            &= [ \nabla_{x_{i}} f^{\eta}_{i}(x^{k}_{i}, x^{k}_{-i}) + \mu(\widehat{x}^{\eta}_{i}(x^{k}) - x^{k}_{i}) ]^{\top} (\bar{y}_{i} - \widehat{x}^{\eta}_{i}(x^{k})).
        \end{aligned}
    \end{equation}
    Since we assume that each $\nabla_{x_{i}} f^{\eta}(\bullet)$ is C$^{1}$ for any $i\in [N]$, by \eqref{MS-SABR-main-proposition-eqn3} and \eqref{MS-SABR-main-proposition-eqn4}, we may obtain that $\nabla_{x_{i}} f^{\eta}_{i}(\bar{x}_{i}, \bar{x}_{-i})^{\top} (\bar{y}_{i} - \bar{x}_{i})\geq 0$ by taking the limit $k\overset{\mathcal{K}}{\to}\infty$ on \eqref{MS-SABR-main-proposition-eqn6}, which contradicts \eqref{MS-SABR-main-proposition-eqn5}. Therefore, $\bar{x}$ is a QNE of Moreau-smoothed game $\mathcal{G}({\bf f}^{\eta}, X, \bxi)$, thus $\bar{x}$ is an $\mathcal{O}(\eta)$-QNE by Theorem \ref{QNE-approx-thm}.
\end{proof}

Again, we may derive the sublinear rate by considering the randomized output.

\begin{theorem}[Sublinear rate of MS-SABR]\label{sublinear-rate-MS-SABR}\em
    Consider the stochastic $N$-player game $\mathcal{G}({\bf f}, X, \bxi)$ where each $f_{i}(\bullet, x_{-i})$ is $\rho_{i}$-weakly convex for given $x_{-i}$. Suppose that Assumption $\mathrm{D}$ holds. Let $\{x^{k}\}_{k\geq 0}$ be the sequence generated by \textbf{MS-SABR}. Consider the same setting as Theorem \ref{MS-SABR-main-proposition}. Suppose that we choose $\gamma > 0$ such that $\gamma\mu > 1$, $\mu \geq 1/\eta$, $p_{i} = 1/N$ for any $i\in [N]$, and $\varepsilon^{k}_{\max} = 1/K$ for any $k\in \{ 0, 1, \cdots, K-1 \}$. Then for any $k\geq 0$, we have that
        $\mathbb{E}[\| G^{\eta}_{X, \gamma}(x^{R_{K}}) \|^{2}] \leq \tfrac{2 N \bar{M}}{\eta K}$,
    where $\bar{M} > 0$ is some constant that depends on $M$ defined in assumption $\mathrm{(D3)}$.
\end{theorem}
\begin{proof}
    By taking unconditional expectation on both sides of \eqref{MS-SABR-main-proposition-result}, we obtain
    \begin{equation*}
         \begin{aligned}
             \mathbb{E}[P(x^{k+1})] \leq \mathbb{E}[P(x^{k})] - \left( \mu - \tfrac{1}{2\eta} \right) \sum_{i=1}^{N} p_{i} \mathbb{E} \left[ \left\| \widehat{x}^{\eta}_{i}(x^{k}) - x^{k}_{i} \right\|^{2} \right] + M \varepsilon^{k}_{\max}.
         \end{aligned}
     \end{equation*}
     By plugging in $\mu \geq 1/\eta$, $p_{i} = 1/N$, and $\varepsilon^{k}_{\max} = 1/K$, rearranging the terms, and taking the summation from $k = 0$ to $K-1$, it leads to
     \begin{equation*}
         \begin{aligned}
             \tfrac{1}{2\eta} \tfrac{1}{N} \sum_{k=0}^{K-1} \sum_{i=1}^{N} \mathbb{E} \left[ \left\| \widehat{x}^{\eta}_{i}(x^{k}) - x^{k}_{i} \right\|^{2} \right] \leq \underbrace{\mathbb{E}[P(x^{0})-P_{\min}] + M}_{\triangleq \bar{M}}.
         \end{aligned}
     \end{equation*}
     Then we divide both sides by $K$. Since we choose $\gamma > 0$ such that $\gamma \mu > 1$ and $\mu = 1/\eta$, by invoking Lemma \ref{lm:bd_Fnat}, we may obtain that
         $\mathbb{E}[\| G^{\eta}_{X, \gamma}(x^{R_{K}}) \|^{2}] \leq \tfrac{2 N \bar{M}}{\eta K}$,
     where $R_{K}$ is uniformly distributed in $\{0,1,\cdots, K-1\}$, as desired.
\end{proof}

\subsubsection{Complexity analysis}\label{Sec-4.3.2}

Based on the above sublinear rate guarantee, we are ready to derive the complexity results of \textbf{MS-SABR}. Same as \textbf{MS-SSBR}, we employ \eqref{O-IMGM} to solve the lower-level surrogated BR problem \eqref{MS-SABR-eqn2}.

\begin{theorem}[Complexities of MS-SABR]\label{complexity-MS-SABR}\em
    Consider the stochastic $N$-player game $\mathcal{G}({\bf f}, X, \bxi)$ where each $f_{i}(\bullet, x_{-i})$ is $\rho_{i}$-weakly convex for given $x_{-i}$. Suppose that Assumption $\mathrm{D}$ holds. Let $\{x^{k}\}_{k\geq 0}$ be the sequence generated by \textbf{MS-SABR}. Under the same settings as Theorem \ref{MS-SABR-main-proposition} and Theorem \ref{sublinear-rate-MS-SABR}, we have that\\
    (i) (Iteration complexity) We have that $\mathbb{E}[\| G^{\eta}_{X, \gamma}(x^{R_{K}}) \|]\leq \epsilon$ after most $K(\epsilon)$ iterations, where $K(\epsilon)$ is defined as
    \begin{equation*}
        K(\epsilon) \triangleq \left\lceil \tfrac{2 N \bar{M}}{\eta \epsilon^{2}} \right\rceil.
    \end{equation*}
    (ii) (Sample complexity) The expected overall sample complexity $S_i(\epsilon)$ for player $i$ to achieve $\mathbb{E}[\| G^{\eta}_{X, \gamma}(x^{R_{K}}) \|]\leq \epsilon$ can be bounded as
    \begin{equation*}
        S_i(\epsilon) \leq \tfrac{54 Q' N^{2} \bar{M}^{3}}{\eta^{5} \epsilon^{6}},
    \end{equation*}
    where $Q'>0$ is some constant from Lemma \ref{second-moment-error-O-IMGM}.
\end{theorem}
\begin{proof}
    The bound on $K(\epsilon)$ in (i) follows from Theorem \ref{sublinear-rate-MS-SABR} immediately. Now we derive the bound on sample complexity in (ii). By Lemma \ref{second-moment-error-O-IMGM}, we know that
    \begin{align*}
        \mathbb{E}&  \left[ \| x^{k+1}_{i(k)} - \widehat{x}^{\eta}_{i(k)}(x^{k}) \|^{2} \:\middle\vert\: x^{k} \right] \leq \tfrac{Q^{\prime}}{\mu^{2}\eta^{2}T^{k}_{i(k)}} \leq (\varepsilon^{k}_{i(k)})^{2} \leq \tfrac{1}{K^{2}},
    \end{align*}
    which implies that $T^{k}_{i(k)} = \left\lceil \tfrac{Q'K^{2}}{\mu^{2}\eta^{2}} \right\rceil \leq \tfrac{2Q'K^{2}}{\mu^{2}\eta^{2}}$. Let $T^{k}_{j} = 0$ for any $j\neq i(k)$. Therefore, during $K(\epsilon)$ iterations, the expected number of samples for player $i$ to achieve $\mathbb{E}[\| G^{\eta}_{X, \gamma}(x^{R_{K}}) \|]\leq \epsilon$ is no greater than
    \begin{align*}
        \mathbb{E} \left[ \sum_{k=0}^{K(\epsilon)-1} T^{k}_{i} \right] &= p_{i} \mathbb{E} \left[ \sum_{k=0}^{K(\epsilon)-1} T^{k}_{i} \:\middle\vert\: i(k) = i \right] + (1 - p_{i}) \mathbb{E} \left[ \sum_{k=0}^{K(\epsilon)-1} T^{k}_{i} \:\middle\vert\: i(k) \neq i \right] \\
        &\leq p_{i} \sum_{k=0}^{K(\epsilon)-1} \tfrac{2Q'K^{2}}{\mu^{2}\eta^{2}} = p_{i} \tfrac{2Q'K^{3}}{\mu^{2}\eta^{2}}.
    \end{align*}
    Recall that we choose $\mu = 1/\eta$ and $p_{i} = 1/N$, by the fact that $K(\epsilon) = \left\lceil \tfrac{2 N \bar{M}}{\eta\epsilon^{2}} \right\rceil \leq \tfrac{3 N \bar{M}}{\eta\epsilon^{2}}$, we may obtain that
    \begin{equation*}
        S_{i}(\epsilon) \leq p_{i} \tfrac{2Q'K^{3}}{\mu^{2}\eta^{2}} \leq \tfrac{54 Q' N^{2} \bar{M}^{3}}{\eta^{5} \epsilon^{6}},
    \end{equation*}
    which completes the proof.
\end{proof}

\begin{remark}\em
    By comparing \textbf{MS-ABR} and \textbf{MS-SABR}, it can be seen that in the weakly convex case, the smoothing parameter $\eta$ has a significant impact on the overall sample complexity. A smaller smoothing parameter leads to a better approximation, but correspondingly requires a higher overall sample complexity. $\hfill \Box$
\end{remark}

\section{Numerical experiments}\label{Sec-5}

In this section, we test our proposed schemes on three distinct game-theoretic problems.

\subsection{Stochastic strongly convex Nash-Cournot games}
We begin by applying \textbf{MS-SBR} on strongly convex Cournot games, where the $i$th player solves
\begin{equation*}
    \min_{x_{i}\in X_{i}} f_{i}(x_{i}, x_{-i}) \triangleq \mathbb{E}[\Tilde{c}_{i}(\bxi)]g_{i}(x_{i}) - \mathbb{E}[\Tilde{p}(\bar{x}, \bxi)]x_{i}, ~ \forall i\in [N],
\end{equation*}
where $\mathbb{E}[\Tilde{c}_{i}(\bxi)]g_{i}(x_{i})$ is the private cost function of player $i$ and $\mathbb{E}[\Tilde{p}(\bar{x}, \bxi)]$ is the expected price function with $\bar{x} = \sum_{i=1}^{N} x_{i}$. (I) \emph{Problem parameters.} We specify $N = 4$ and assume that for any $i$, $X_{i} = [0, 20]$, $\xi \sim \mathrm{U}[0, 1]$, $\Tilde{c}_{i}(\xi) = (2 + i/N)\xi$ for any $i\in [N]$, while the linear inverse demand function is defined as $\Tilde{p}(\bar{x}, \xi) \triangleq a(\xi) - b(\xi)\bar{x}$ where $a(\xi) = 4\xi$ and $b(\xi) = 0.02\xi$. Let $g_{i}$ be a strongly convex nonsmooth function, given by $g_{i}(x_{i}) \triangleq \max\{ 0.5x^{2}_{i}, x^{2}_{i}-2 \}$. In the numerics, we test three smoothing parameters $\eta_1 = 1.0$, $\eta_2 = 1.5$, and $\eta_3 = 3.0$, all of which satisfy the contractivity condition $\mathrm{(A3)}$. (II) \emph{Implementation.} We set  $x^{0} = {\bf 0}$. The empirical linear convergence rate of $e_k$, the expected error given in \eqref{expected-error}, for $\mu = 2.0$ is shown in Fig. \ref{MS-SBR-Fig.}-(a), a result of averaging over $10$ sample paths. It can be observed that a larger $\eta$ leads to a slower convergence since the norm of $\Gamma$ is closer to $1$, when $\eta$ gets larger. We plot the trajectories of $\{x^{k}_{1}\}_{k=0}$ in Fig. \ref{MS-SBR-Fig.}-(b) where the black dashed line denotes the true equilibrium $x^{*}_{1}$. The empirical expected errors for differing choices of $\mu$ are provided in Table~\ref{MS-SBR-table}.

\begin{figure}[htb]
    \centering
    \subfigure[Linear rate (synch)]{
        \includegraphics[width=0.45\textwidth]{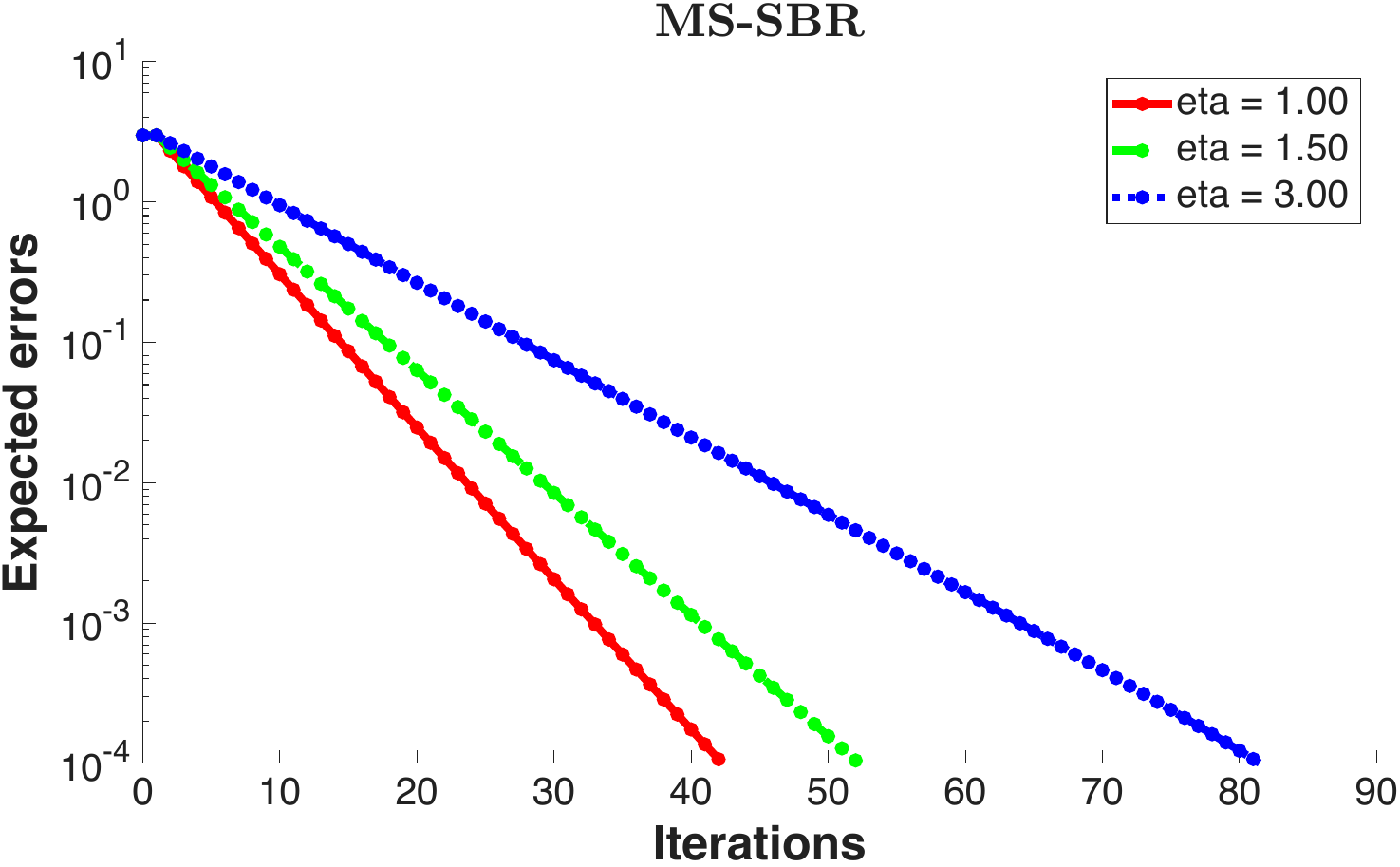}
    }
    \subfigure[a.s. conv. (synch)]{
        \includegraphics[width=0.45\textwidth]{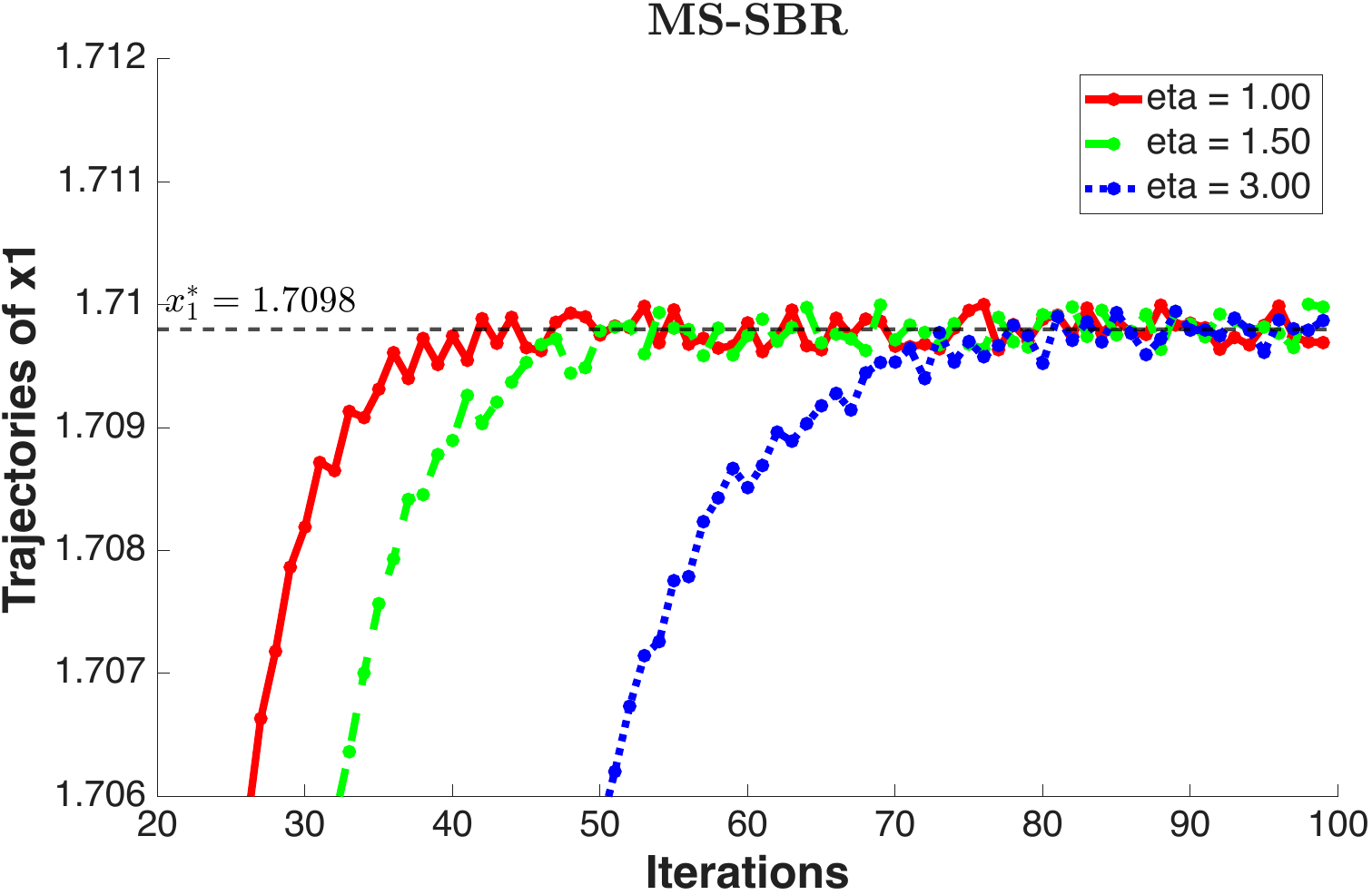}
    }

    \subfigure[Sub. rate (asynch)]{
        \includegraphics[width=0.45\textwidth]{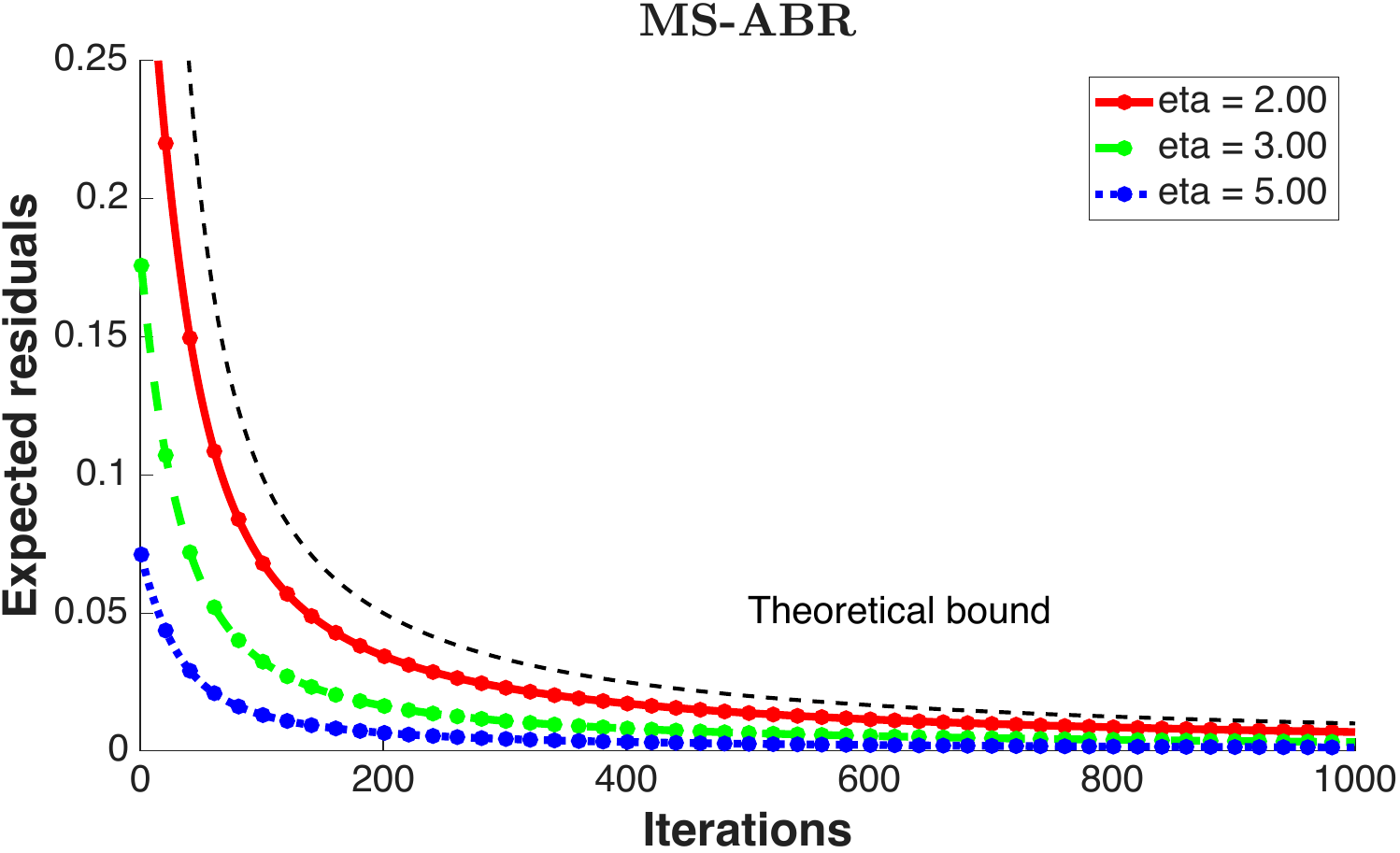}
    }
    \subfigure[a.s. conv. (asynch)]{
        \includegraphics[width=0.45\textwidth]{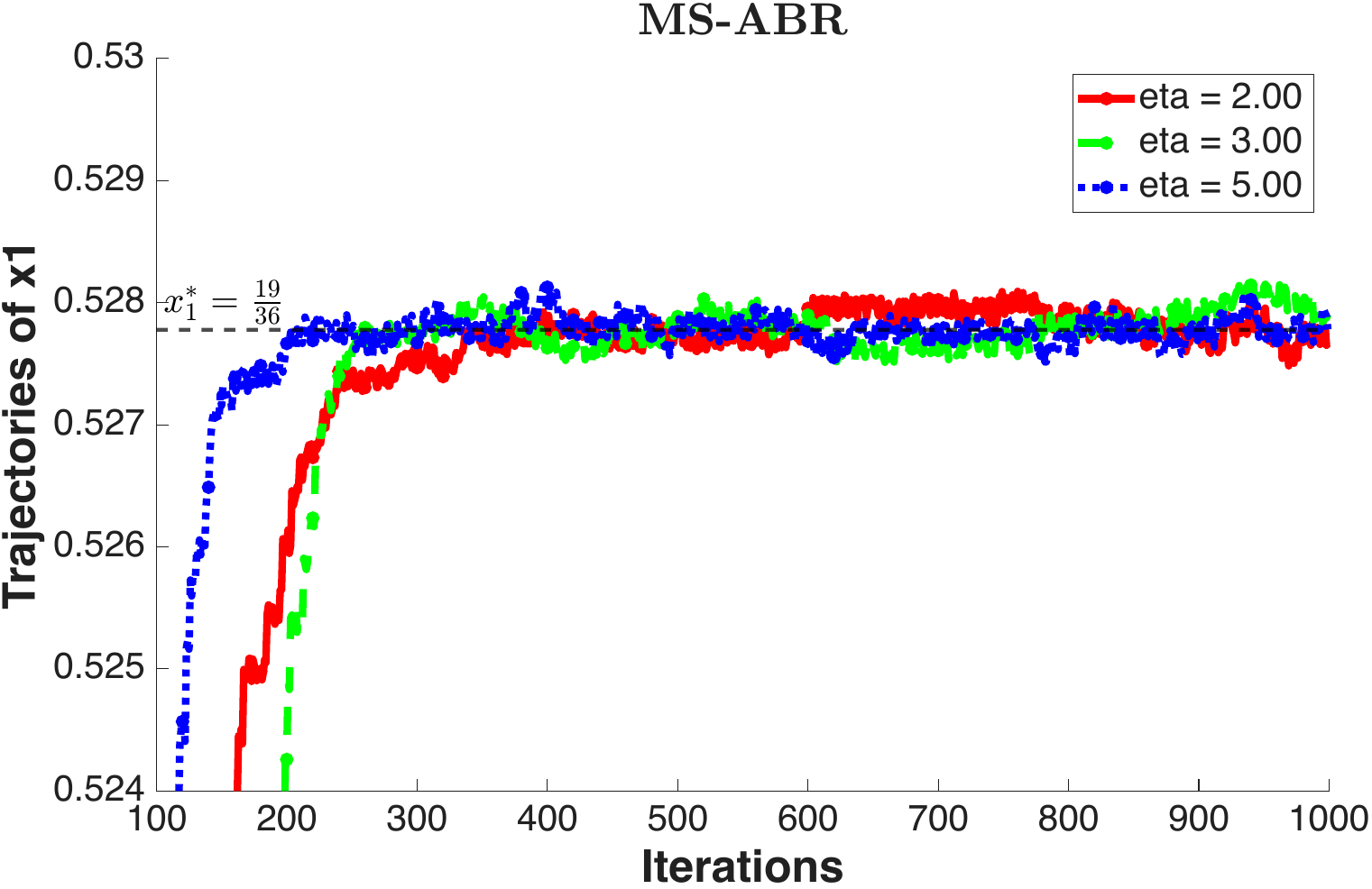}
    }

    \caption{\textbf{MS-SBR} and \textbf{MS-ABR} for stochastic strongly convex games.}
    \label{MS-SBR-Fig.}
\end{figure}

\begin{table}
\centering
{\renewcommand{\arraystretch}{1.25}
\begin{tabular}{@{\hspace{15pt}}c@{\hspace{25pt}}c@{\hspace{25pt}}c@{\hspace{25pt}}c@{\hspace{25pt}}c@{\hspace{15pt}}}
\toprule
\small{$e_{K}$} & \small{$\mu = 2.0$} & \small{$\mu = 4.0$} & \small{$\mu = 6.0$} & \small{$\mu = 8.0$} \\
\midrule
\small{$\eta = 1.0$} & \small{2.49\text{e}-11} & \small{9.08\text{e}-7} & \small{3.66\text{e}-5} & \small{2.33\text{e}-4} \\
\small{$\eta = 1.5$} & \small{1.87\text{e}-9} & \small{9.71\text{e}-6} & \small{1.82\text{e}-4} & \small{7.63\text{e}-4} \\
\small{$\eta = 3.0$} & \small{1.45\text{e}-6} & \small{3.11\text{e}-4} & \small{1.71\text{e}-3} & \small{3.74\text{e}-3} \\
\bottomrule
\end{tabular}
}
\caption{The empirical expected errors after $K = 100$ iterations for different $\mu$.}
\label{MS-SBR-table}
\end{table}

\subsection{Stochastic strongly convex congestion problem}
Inspired by \cite{yin-shanbhag-mehta-2011}, we consider an $N$-player strongly convex congestion game, where the $i$th player solves
\begin{equation*}
    \max_{x_{i}\in X_{i}} ~ f_{i}(x_{i}, x_{-i}) \triangleq \mathbb{E}[\Tilde{g}_{i}(x_{i}, \bxi)] - \sum_{j=1}^{N} h(x_{j}),
\end{equation*}
where $g_{i}(x_{i}) \triangleq \mathbb{E}[\Tilde{g}_{i}(x_{i}, \bxi)]$ denotes the utility function of player $i$, and $\sum_{j=1}^{N} h(x_{j})$ is the aggregate congestion cost. It can be known from Lemma \ref{assumption-B2-example} that such a game satisfies the potentiality assumption $\mathrm{(B2)}$. (I) \emph{Problem parameters.} In this example, we consider $N = 6$ players. Suppose that each $X_{i} = [0, 10]$, $\xi \sim \mathrm{U}[-1, 1]$, $\Tilde{g}_{i}(x_{i}, \xi) \triangleq (1 + i/(3N) + 0.5\xi)\min\{ x_{i}, \tfrac{1}{2}x_{i} + 3 \}$, and $h(x_{i}) = x^{2}_{i}$ for any $i\in [N]$. We consider three different Moreau smoothing values $\eta_{1} = 2.0$, $\eta_{2} = 3.0$, and $\eta_{3} = 5.0$ in the numerical experiment. (II) \emph{Implementation.} We set the initialization $x^{0} = {\bf 0}$ and employ the expected residual defined in \eqref{MSABR-rate-final-result}. The sublinear rate averaged over $10$ sample paths and almost sure convergence of \textbf{MS-ABR} are shown in Fig. \ref{MS-SBR-Fig.}-(c) and Fig. \ref{MS-SBR-Fig.}-(d), respectively. In contrast to \textbf{MS-SBR}, it can be observed that the asynchronous BR update procedure converges faster with larger $\eta$.

\subsection{Stochastic weakly convex Nash-Cournot problem}

In many economic models, production costs are not inherently monotone. Economies of scale \cite{osullivan-sheffrin-swan-2003} may initially lower average or marginal costs through fixed-cost spreading, specialization, and efficiency gains. However, these benefits are limited. At larger scales, factors like congestion, coordination frictions, and resource constraints can raise average or marginal costs, leading to diseconomies of scale. 

Here we consider the following stochastic weakly convex Nash-Cournot problem:
\begin{equation*}
    \min_{x_{i}\in X_{i}} ~ f_{i}(x_{i}, x_{-i}) \triangleq \underbrace{\mathbb{E}[\Tilde{c}_{i}(\bxi)](a_{i}(x_{i})x_{i} + g_{i}(x_{i}))}_{\triangleq c_{i}(x_{i})} - \mathbb{E}[\Tilde{p}(\bar{x}, \bxi)]x_{i}, ~ \forall i\in [N],
\end{equation*}
where $a_{i}$ denotes the scale-average cost, which initially decreases due to economies of scale and increases thereafter. The training cost $g_{i}$ declines linearly to zero once sufficient proficiency is attained. The total cost is defined by $c_{i}(x_{i}) \triangleq \mathbb{E}[\Tilde{c}_{i}(\bxi)] (a_{i}(x_{i})x_{i} + g_{i}(x_{i}))$. The expected price is $\mathbb{E}[\Tilde{p}(\bar{x}, \bxi)]$ with $\bar{x} = \sum_{i=1}^{N} x_{i}$. (I) \emph{Problem parameters.} In this example, we consider $N = 4$ players. Suppose that each $X_{i} = [3, 12]$, $\xi \sim \mathrm{U}[-1, 1]$, $\Tilde{c}_{i}(\xi) = 1 + 0.1\xi$ for any $i\in [N]$, $\Tilde{p}(\bar{x}, \xi) = a(\xi) - b(\xi)\bar{x}$ where $a(\xi) = 2 + \xi$ and $b(\xi) = 0.02 + 0.01\xi$. Suppose that $a_{i}(x_{i}) = \max \left\{ -\tfrac{1}{8}x_{i}+1, \tfrac{1}{8}x_{i} \right\}$ and $g_{i}(x_{i}) = \max\{ 4-x_{i}, 0 \}$. Therefore, we can see that $c_{i}(x_{i}) = \max \left\{ -\tfrac{1}{8}x^{2}_{i}+4, \tfrac{1}{8}x^{2}_{i} \right\}$, which is a weakly convex function with $\rho_{i} = \tfrac{1}{4}$ for any $i\in [N]$. Three different Moreau smoothing values $\eta_{1} = 0.3$, $\eta_{2} = 0.5$, and $\eta_{3} = 0.8$ will be tested in the numerical experiment since we require $\eta\rho_{i} \leq \tfrac{1}{2}$ for any $i\in [N]$ in the asynchronous BR update. With these parameter choices, the game under consideration satisfies the contractivity condition $\mathrm{(C2)}$ and the smoothed potentiality assumption $\mathrm{(D2)}$ since all players solve the same problem.
(II) \emph{Implementation.} We set the initialization $x^{0} = 4\mathbf{e}_{4}$ for both \textbf{MS-SSBR} and \textbf{MS-SABR}. Fig. \ref{MS-SSBR-and-MS-SABR}-(a) displays the empirical linear convergence of \textbf{MS-SSBR} via averaging over $10$ sample paths, while Fig. \ref{MS-SSBR-and-MS-SABR}-(b) illustrates the almost sure convergence. For \textbf{MS-SABR}, the sublinear rate averaged over $10$ sample paths is presented in Fig. \ref{MS-SSBR-and-MS-SABR}-(c), and the asymptotic behavior is shown in Fig. \ref{MS-SSBR-and-MS-SABR}-(d). It is worth noting that $x^{\ast} = \tfrac{40}{7}\mathbf{e}_{4}$ lies in the interior of $X$. By Remark \ref{ME-boundedness}, $x^{\ast}$ is an exact QNE of the original stochastic weakly convex game $\mathcal{G}({\bf f}, X, \bxi)$.

\begin{figure}[htb]
    \centering
    \subfigure[Linear rate (synch)]{
        \includegraphics[width=0.45\textwidth]{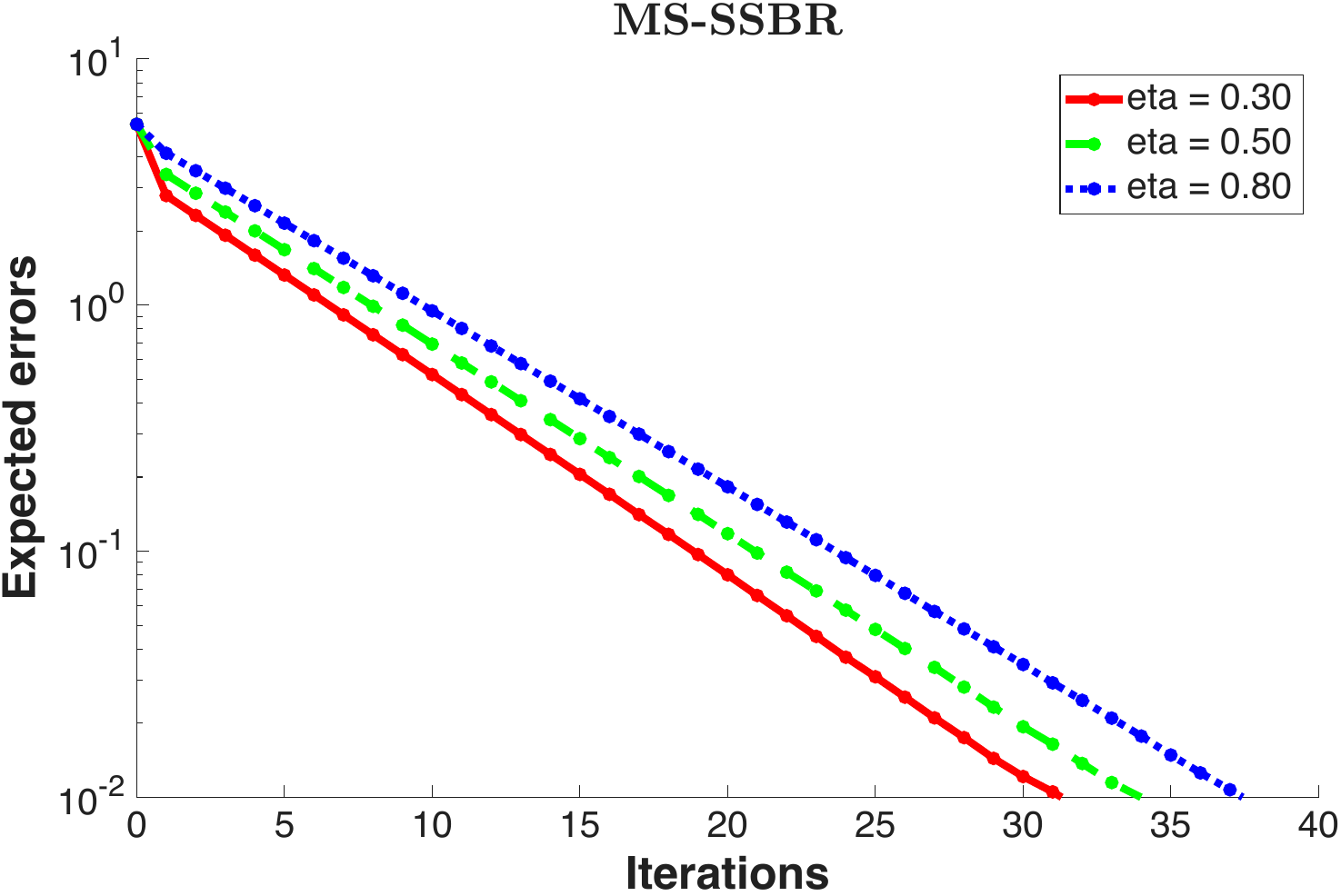}
    }
    \subfigure[a.s. conv. (synch)]{
        \includegraphics[width=0.45\textwidth]{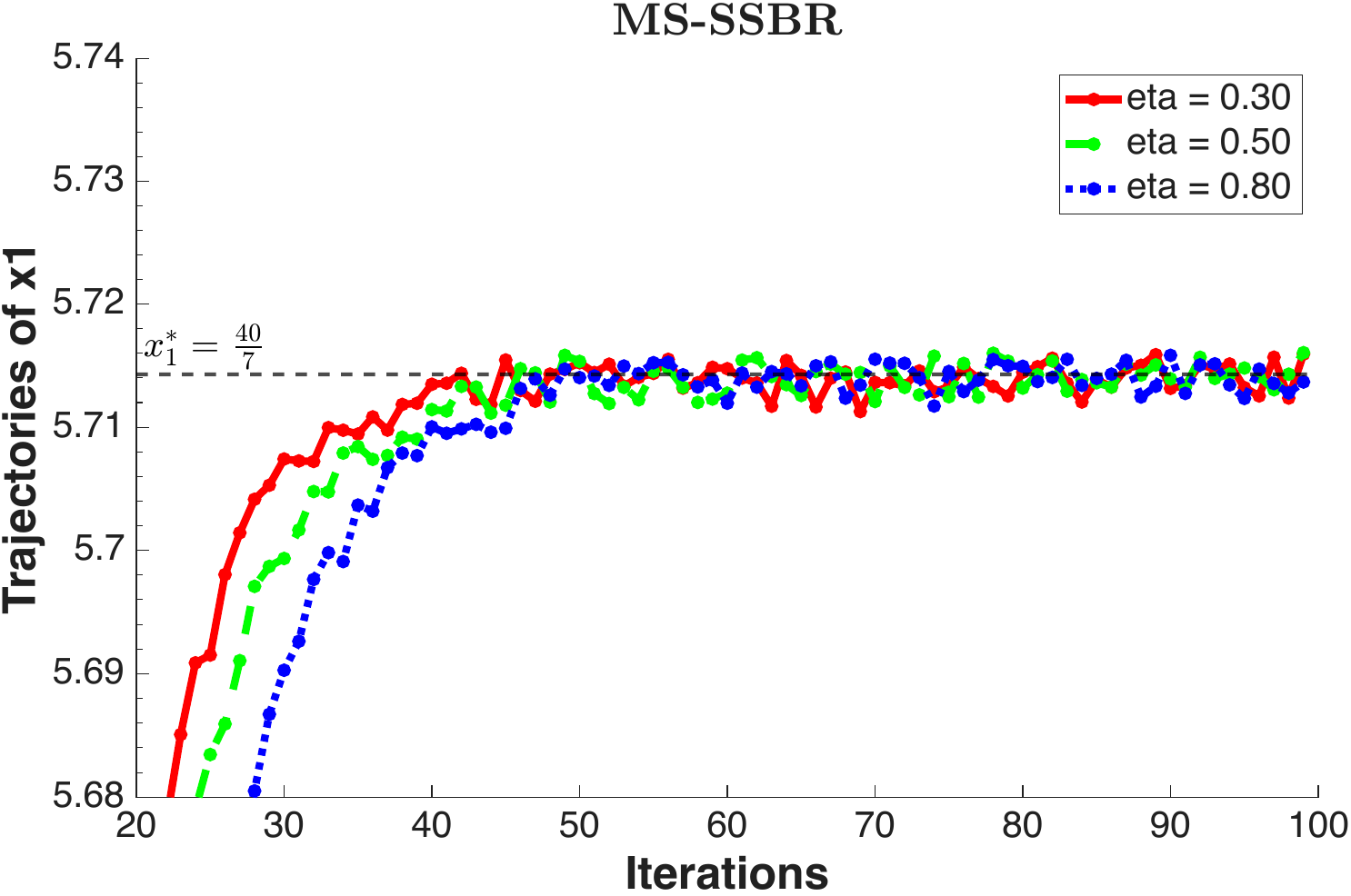}
    }

    \subfigure[Sub. rate (asynch)]{
        \includegraphics[width=0.45\textwidth]{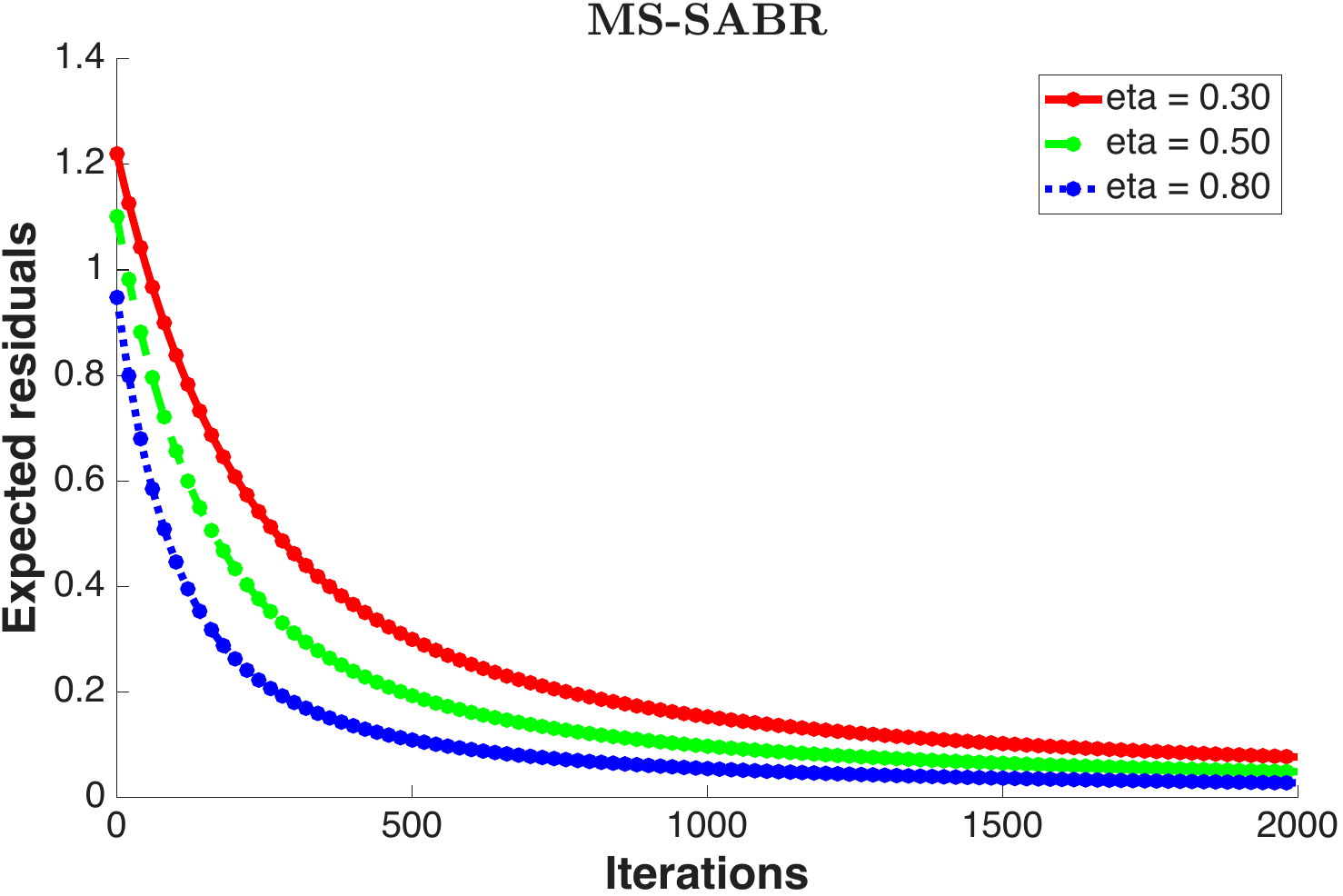}
    }
    \subfigure[a.s. conv. (asynch)]{
        \includegraphics[width=0.45\textwidth]{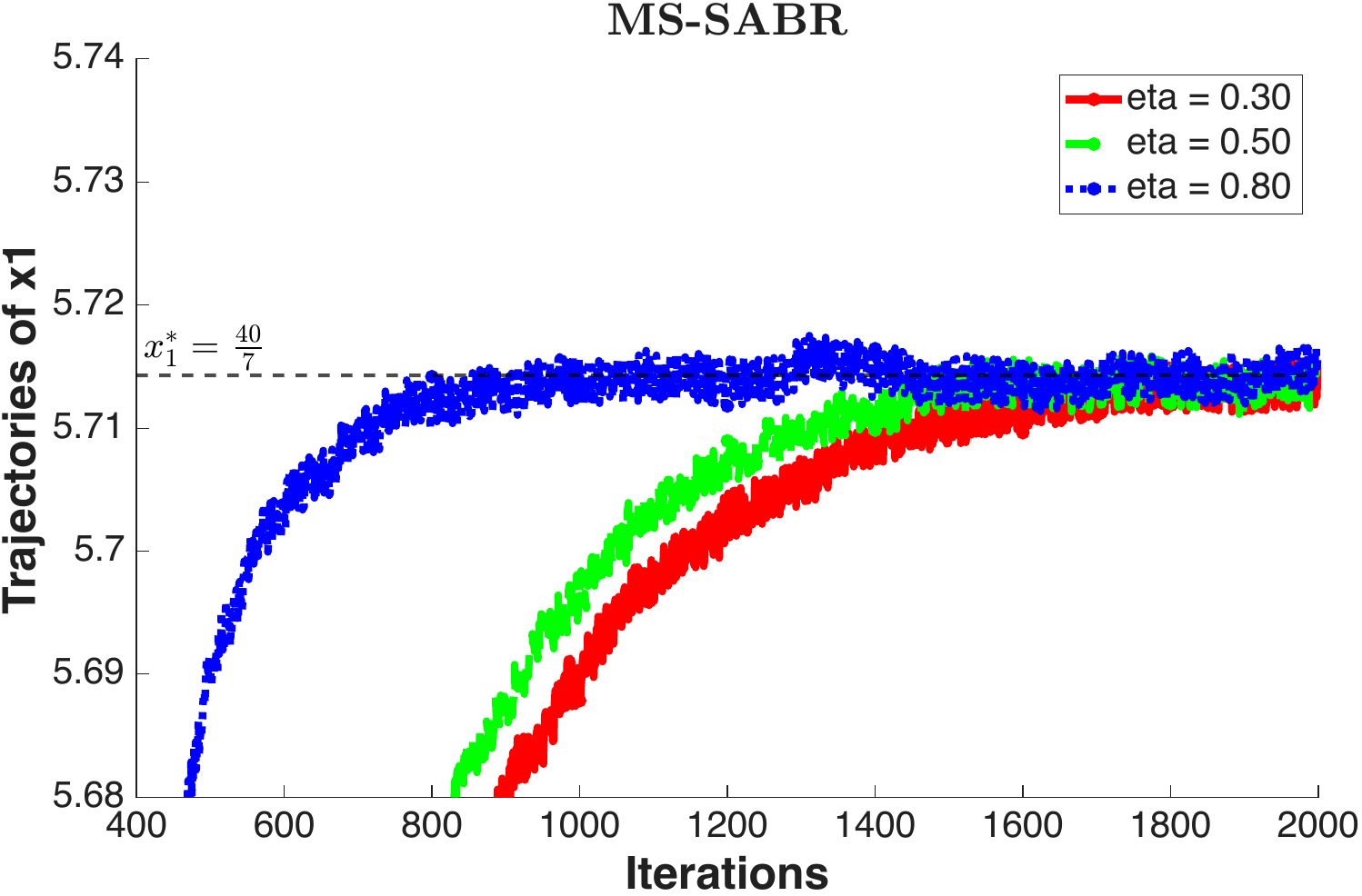}
    }

    \caption{\textbf{MS-SSBR} and \textbf{MS-SABR} for stochastic weakly convex games.}
    \label{MS-SSBR-and-MS-SABR}
\end{figure}

\section{Concluding remarks}\label{Sec-6}
Extant best-response schemes for computing equilibria in stochastic continuous-strategy convex/nonconvex nonsmooth games are afflicted by several shortcomings. In particular, synchronous schemes often impose either stringent smoothness or necessitate satisfying intricate conditions at the limit point, while asynchronous schemes are afflicted by a glaring absence of rate and complexity guarantees, particularly in terms of computing QNE. In this work, we address these gaps by developing synchronous and asynchronous Moreau-smoothed BR schemes for the computation of equilibria for (i) strongly and (ii) weakly convex nonsmooth games with expectation-valued objectives. To contend with (i), we show that Moreau-smoothing maintains the set of equilibria as invariant, providing an avenue for developing a linearly convergent synchronous scheme and a sublinearly convergent asynchronous counterpart. We address (ii) by overlaying surrogation, which leads to a relatively cheap BR subproblem and again develop linearly convergent synchronous schemes and a sublinearly convergent asynchronous counterpart. Preliminary numerics are provided to support our theoretical claims. Future work is expected to address a range of concerns including two-period problems and hierarchical settings.

\bibliographystyle{plain}
\bibliography{references}

@inproceedings{arefizadeh-nedic-2024,
  title={Non-monotone variational inequalities},
  author={Arefizadeh, Sina and Nedi{\'c}, Angelia},
  booktitle={IEEE 60th Annual Allerton Conference on Communication, Control, and Computing},
  pages={1--7},
  year={2024}
}

@book{beck-2017,
  title={First-Order Methods in Optimization},
  author={Beck, Amir},
  year={2017},
  publisher={MOS-SIAM Series on Optimization}
}

@inproceedings{ben-tal-teboulle-2006,
  title={A smoothing technique for nondifferentiable optimization problems},
  author={Ben-Tal, Aharon and Teboulle, Marc},
  booktitle={Optimization: Proceedings of the Fifth French-German Conference},
  pages={1--11},
  year={2006}
}

@article{bohm-wright-2021,
  title={Variable smoothing for weakly convex composite functions},
  author={B{\"o}hm, Axel and Wright, Stephen J.},
  journal={Journal of Optimization Theory and Applications},
  volume={188},
  number={3},
  pages={628--649},
  year={2021},
  publisher={Springer}
}

@article{bubeck-2015,
  title={Convex Optimization: Algorithms and Complexity},
  author={Bubeck, Sebastien},
  journal={Foundations and Trends in Machine Learning},
  volume={8},
  number={3-4},
  pages={231--357},
  year={2015}
}

@book{cui-pang-2021,
  title={Modern Nonconvex Nondifferentiable Optimization},
  author={Cui, Ying and Pang, Jong-Shi},
  year={2021},
  publisher={MOS-SIAM Series on Optimization}
}

@article{cui-shanbhag-2023,
  title={On the computation of equilibria in monotone and potential stochastic hierarchical games},
  author={Cui, Shisheng and Shanbhag, Uday V.},
  journal={Mathematical Programming},
  volume={198},
  number={2},
  pages={1227--1285},
  year={2023},
  publisher={Springer}
}

@article{dang-lan-2015,
  title={On the convergence properties of non-euclidean extragradient methods for variational inequalities with generalized monotone operators},
  author={Dang, Cong D. and Lan, Guanghui},
  journal={Computational Optimization and Applications},
  volume={60},
  number={2},
  pages={277--310},
  year={2015},
  publisher={Springer}
}

@article{davis-drusvyatskiy-2019,
  title={Stochastic model-based minimization of weakly convex functions},
  author={Davis, Damek and Drusvyatskiy, Dmitriy},
  journal={SIAM Journal on Optimization},
  volume={29},
  number={1},
  pages={207--239},
  year={2019},
  publisher={SIAM}
}

@book{facchinei-pang-2009,
  title={Nash Equilibria: The Variational Approach},
  author={Facchinei, Francisco and Pang, Jong-Shi},
  year={2009},
  publisher={Convex Optimization in Signal Processing and Communications, Cambridge University Press}
}

@book{fudenberg-levine-1998,
  title={The Theory of Learning in Games},
  author={Fudenberg, Drew and Levine, David K.},
  year={1998},
  publisher={MIT press}
}

@article{hong-wang-razaviyayn-luo-2017,
  title={Iteration complexity analysis of block coordinate descent methods},
  author={Hong, Mingyi and Wang, Xiangfeng and Razaviyayn, Meisam and Luo, Zhi-Quan},
  journal={Mathematical Programming},
  volume={163},
  number={1},
  pages={85--114},
  year={2017},
  publisher={Springer}
}

@article{huang-zhang-2024,
  title={Beyond monotone variational inequalities: Solution methods and iteration complexities},
  author={Huang, Kevin and Zhang, Shuzhong},
  journal={Pacific Journal of Optimization},
  volume={20},
  number={3},
  pages={403--428},
  year={2024},
  publisher={Yokohama}
}

@article{iusem-jofre-oliveira-thompson-2017,
  title={Extragradient method with variance reduction for stochastic variational inequalities},
  author={Iusem, Alfredo N. and Jofr{\'e}, Alejandro and Oliveira, Roberto Imbuzeiro and Thompson, Philip},
  journal={SIAM Journal on Optimization},
  volume={27},
  number={2},
  pages={686--724},
  year={2017},
  publisher={SIAM}
}

@article{iusem-jofre-oliveira-thompson-2019,
  title={Variance-based extragradient methods with line search for stochastic variational inequalities},
  author={Iusem, Alfredo N. and Jofr{\'e}, Alejandro and Oliveira, Roberto I and Thompson, Philip},
  journal={SIAM Journal on Optimization},
  volume={29},
  number={1},
  pages={175--206},
  year={2019},
  publisher={SIAM}
}

@article{jalilzadeh-shanbhag-blanchet-glynn-2022,
  title={Smoothed variable sample-size accelerated proximal methods for nonsmooth stochastic convex programs},
  author={Jalilzadeh, Afrooz and Shanbhag, Uday V. and Blanchet, Jose H. and Glynn, Peter W.},
  journal={Stochastic Systems},
  volume={12},
  number={4},
  pages={373--410},
  year={2022},
  publisher={INFORMS}
}

@article{kannan-shanbhag-2019,
  title={Optimal stochastic extragradient schemes for pseudomonotone stochastic variational inequality problems and their variants},
  author={Kannan, Aswin and Shanbhag, Uday V.},
  journal={Computational Optimization and Applications},
  volume={74},
  number={3},
  pages={779--820},
  year={2019},
  publisher={Springer}
}

@article{koshal-nedic-shanbhag-2013,
  title={Regularized iterative stochastic approximation methods for stochastic variational inequality problems},
  author={Koshal, Jayash and Nedi{\'c}, Angelia and Shanbhag, Uday V.},
  journal={IEEE Transactions on Automatic Control},
  volume={58},
  number={3},
  pages={594--609},
  year={2013},
  publisher={IEEE}
}

@article{kotsalis-lan-li-2022,
  title={Simple and optimal methods for stochastic variational inequalities, {I}: operator extrapolation},
  author={Kotsalis, Georgios and Lan, Guanghui and Li, Tianjiao},
  journal={SIAM Journal on Optimization},
  volume={32},
  number={3},
  pages={2041--2073},
  year={2022},
  publisher={SIAM}
}

@article{lei-shanbhag-2020,
  title={Asynchronous schemes for stochastic and misspecified potential games and nonconvex optimization},
  author={Lei, Jinlong and Shanbhag, Uday V.},
  journal={Operations Research},
  volume={68},
  number={6},
  pages={1742--1766},
  year={2020},
  publisher={INFORMS}
}

@article{lei-shanbhag-2022,
  title={Distributed variable sample-size gradient-response and best-response schemes for stochastic Nash equilibrium problems},
  author={Lei, Jinlong and Shanbhag, Uday V.},
  journal={SIAM Journal on Optimization},
  volume={32},
  number={2},
  pages={573--603},
  year={2022},
  publisher={SIAM}
}

@article{lei-shanbhag-pang-sen-2020,
  title={On synchronous, asynchronous, and randomized best-response schemes for stochastic Nash games},
  author={Lei, Jinlong and Shanbhag, Uday V. and Pang, Jong-Shi and Sen, Suvrajeet},
  journal={Mathematics of Operations Research},
  volume={45},
  number={1},
  pages={157--190},
  year={2020},
  publisher={INFORMS}
}

@article{liao-ding-zheng-2023,
  title={Error bounds, {PL} condition, and quadratic growth for weakly convex functions, and linear convergences of proximal point methods},
  author={Liao, Feng-Yi and Ding, Lijun and Zheng, Yang},
  journal={arXiv:2312.16775v2},
  year={2023}
}

@inproceedings{lin-jin-jordon-2020,
  title={Near-optimal algorithms for minimax optimization},
  author={Lin, Tianyi and Jin, Chi and Jordan, Michael I.},
  booktitle={Conference on Learning Theory},
  pages={2738--2779},
  year={2020},
  organization={PMLR}
}

@article{liu-cui-pang-2022,
  title={Solving nonsmooth and nonconvex compound stochastic programs with applications to risk measure minimization},
  author={Liu, Junyi and Cui, Ying and Pang, Jong-Shi},
  journal={Mathematics of Operations Research},
  volume={47},
  number={4},
  pages={3051--3083},
  year={2022},
  publisher={INFORMS}
}

@article{mairal-2015,
  title={Incremental majorization-minimization optimization with application to large-scale machine learning},
  author={Mairal, Julien},
  journal={SIAM Journal on Optimization},
  volume={25},
  number={2},
  pages={829--855},
  year={2015},
  publisher={SIAM}
}

@article{marks-wright-1978,
  title={A general inner approximation algorithm for nonconvex mathematical programs},
  author={Marks, Barry R. and Wright, Gordon P.},
  journal={Operations Research},
  volume={26},
  number={4},
  pages={681--683},
  year={1978},
  publisher={INFORMS}
}

@article{monderer-shapley-1996,
  title={Potential games},
  author={Monderer, Dov and Shapley, Lloyd S.},
  journal={Games and Economic Behavior},
  volume={14},
  number={1},
  pages={124--143},
  year={1996},
  publisher={Elsevier}
}

@book{mordukhovich-nam-2023,
  title={An Easy Path to Convex Analysis and Applications},
  author={Mordukhovich, Boris S. and Nam, Nguyen M.},
  year={2023},
  publisher={Springer}
}

@article{moreau-1965,
  title={Proximit{\'e} et dualit{\'e} dans un espace hilbertien},
  author={Moreau, Jean-Jacques},
  journal={Bulletin de la Soci{\'e}t{\'e} math{\'e}matique de France},
  volume={93},
  pages={273--299},
  year={1965}
}

@article{nash-1951,
  title={Non-cooperative games},
  author={Nash, John},
  journal={Annals of Mathematics},
  volume={54},
  number={2},
  pages={286-295},
  year={1951}
}

@article{nemirovski-juditsky-lan-shapiro-2009,
  title={Robust stochastic approximation approach to stochastic programming},
  author={Nemirovski, Arkadi and Juditsky, Anatoli and Lan, Guanghui and Shapiro, Alexander},
  journal={SIAM Journal on optimization},
  volume={19},
  number={4},
  pages={1574--1609},
  year={2009},
  publisher={SIAM}
}

@article{nesterov-2005,
  title={Smooth minimization of non-smooth functions},
  author={Nesterov, Yu},
  journal={Mathematical Programming},
  volume={103},
  number={1},
  pages={127--152},
  year={2005},
  publisher={Springer}
}

@article{nurminskii-1973,
  title={The quasigradient method for the solving of the nonlinear programming problems},
  author={Nurminskii, E. A.},
  journal={Cybernetics},
  volume={9},
  pages={145--150},
  year={1973}
}

@book{osullivan-sheffrin-swan-2003,
  title={Economics: Principles in Action},
  author={O'sullivan, Arthur and Sheffrin, Steven M. and Swan, Kathy},
  year={2003}
}

@book{pang-razaviyayn-2016,
  title={A unified distributed algorithm for noncooperative games},
  author={Pang, Jong-Shi and Razaviyayn, Meisam},
  year={2016},
  publisher={Big Data over Networks, Cambridge University Press}
}

@article{pang-scutari-2011,
  title={Nonconvex games with side constraints},
  author={Pang, Jong-Shi and Scutari, Gesualdo},
  journal={SIAM Journal on Optimization},
  volume={21},
  number={4},
  pages={1491--1522},
  year={2011},
  publisher={SIAM}
}

@article{pang-scutari-2013,
  title={Joint sensing and power allocation in nonconvex cognitive radio games: Quasi-Nash equilibria},
  author={Pang, Jong-Shi and Scutari, Gesualdo},
  journal={IEEE Transactions on Signal Processing},
  volume={61},
  number={9},
  pages={2366--2382},
  year={2013},
  publisher={IEEE}
}

@phdthesis{razaviyayn-2014,
  title={Successive Convex Approximation: Analysis and Applications},
  author={Razaviyayn, Meisam},
  year={2014},
  school={University of Minnesota}
}

@article{razaviyayn-hong-luo-2013,
  title={A unified convergence analysis of block successive minimization methods for nonsmooth optimization},
  author={Razaviyayn, Meisam and Hong, Mingyi and Luo, Zhi-Quan},
  journal={SIAM Journal on Optimization},
  volume={23},
  number={2},
  pages={1126--1153},
  year={2013},
  publisher={SIAM}
}

@article{renaud-leclaire-papadakis-2025,
  title={On the {M}oreau envelope properties of weakly convex functions},
  author={Renaud, Marien and Leclaire, Arthur and Papadakis, Nicolas},
  journal={arXiv:2509.13960v2},
  year={2025}
}

@article{robbins-siegmund-1971,
  title={A convergence theorem for nonnegative almost supermartingales and some applications},
  author={Robbins, Herbert and Siegmund, David},
  journal={Optimizing Methods in Statistics},
  pages={233--257},
  year={1971},
  publisher={Elsevier}
}

@book{rockafellar-wets-1998,
  title={Variational Analysis},
  author={Rockafellar, R. Tyrrell and Wets, Roger J-B},
  year={1998},
  publisher={Springer}
}

@article{steklov-1907,
  title={Sur les expressions asymptotiques de certaines fonctions définies par les equations differentielles du second ordre et leurs applications au probleme du developement dune fonction arbitraire en series procedant suivant les diverses fonctions},
  author={Steklov, V.A.},
  journal={Communications de la Société mathématique de Kharkow},
  volume={10},
  pages={97--199},
  year={1907}
}

@article{sun-Babu-Palomar-2016,
  title={Majorization-minimization algorithms in signal processing, communications, and machine learning},
  author={Sun, Ying and Babu, Prabhu and Palomar, Daniel P.},
  journal={IEEE Transactions on Signal Processing},
  volume={65},
  number={3},
  pages={794--816},
  year={2016},
  publisher={IEEE}
}

@inproceedings{vankov-nedic-sankar-2023,
  title={Last iterate convergence of Popov method for non-monotone stochastic variational inequalities},
  author={Vankov, Daniil and Nedi{\'c}, Angelia and Sankar, Lalitha},
  booktitle={OPT2023: 15th Annual Workshop on Optimization for Machine Learning},
  pages={1--27},
  year={2023}
}

@article{xiao-shanbhag-2025-GR,
  title={Computing equilibria in stochastic nonconvex and non-monotone games via gradient-response schemes},
  author={Xiao, Zhuoyu and Shanbhag, Uday V.},
  journal={arXiv:2504.14056v2},
  year={2025}
}

@article{yin-shanbhag-mehta-2011,
  title={Nash equilibrium problems with scaled congestion costs and shared constraints},
  author={Yin, Huibing and Shanbhag, Uday V. and Mehta, Prashant G.},
  journal={IEEE transactions on automatic control},
  volume={56},
  number={7},
  pages={1702--1708},
  year={2011},
  publisher={IEEE}
}

@article{yousefian-nedic-shanbhag-2016,
  title={Self-tuned stochastic approximation schemes for non-Lipschitzian stochastic multi-user optimization and Nash games},
  author={Yousefian, Farzad and Nedi{\'c}, Angelia and Shanbhag, Uday V},
  journal={IEEE Transactions on Automatic Control},
  volume={61},
  number={7},
  pages={1753--1766},
  year={2016},
  publisher={IEEE}
}

@article{yousefian-nedic-shanbhag-2017,
  title={On smoothing, regularization, and averaging in stochastic approximation methods for stochastic variational inequality problems},
  author={Yousefian, Farzad and Nedi{\'c}, Angelia and Shanbhag, Uday V.},
  journal={Mathematical Programming},
  volume={165},
  number={1},
  pages={391--431},
  year={2017},
  publisher={Springer}
}

\end{document}